\documentclass[11pt]{article}

\usepackage{latexsym,amsfonts,amsmath,theorem,amssymb,stmaryrd,array,tabularx,bb
m}
\usepackage[usenames,dvipsnames]{xcolor}
\usepackage{epsfig}
\usepackage{psfrag}
\usepackage{mathrsfs}   

\usepackage{a4wide}
\usepackage{url}

\newtheorem{theorem}{Theorem}
\newtheorem{lemma}[theorem]{Lemma}
\newtheorem{corollary}[theorem]{Corollary}

\newenvironment{proof}{\begin{trivlist}
    \item[\hskip\labelsep{\bf Proof.}]}{$\hfill\Box$\end{trivlist}}

{\theoremstyle{plain} \theorembodyfont{\rmfamily}

\newtheorem{remark}[theorem]{Remark}}
\newtheorem{algorithm}{Algorithm}

\newcommand{\satop}[2]{\stackrel{\scriptstyle{#1}}{\scriptstyle{#2}}}

\newcommand{\bsbeta}{{\boldsymbol{\beta}}}
\newcommand{\bsDelta}{{\boldsymbol{\Delta}}}
\newcommand{\bszeta}{{\boldsymbol{\zeta}}}

\newcommand{\bsnu}{{\boldsymbol{\nu}}}
\newcommand{\bsgamma}{{\boldsymbol{\gamma}}}
\newcommand{\bszero}{{\boldsymbol{0}}}
\newcommand{\bsell}{{\boldsymbol{\ell}}}

\newcommand{\bsk}{{\boldsymbol{k}}}

\newcommand{\bsm}{{\boldsymbol{m}}}

\newcommand{\bsb}{{\boldsymbol{b}}}
\newcommand{\bsq}{{\boldsymbol{q}}}

\newcommand{\bsy}{{\boldsymbol{y}}}
\newcommand{\bsz}{{\boldsymbol{z}}}

\newcommand{\rd}{\mathrm{d}}

\newcommand{\bbA}{\mathbb{A}}
\newcommand{\bbB}{\mathbb{B}}
\newcommand{\bbR}{\mathbb{R}}

\newcommand{\bbN}{\mathbb{N}}

\newcommand{\calC}{\mathcal{C}}
\newcommand{\calI}{\mathcal{I}}
\newcommand{\calP}{\mathcal{P}}

\newcommand{\calW}{\mathcal{W}}
\newcommand{\calF}{\mathcal{F}}
\newcommand{\cD}{\mathcal{D}}
\newcommand{\cN}{\mathcal{N}}
\newcommand{\cQ}{\mathcal{Q}}

\newcommand{\calG}{\mathcal{G}}

\newcommand{\setu}{{\mathrm{\mathfrak{u}}}}

\newcommand{\mask}[1]{{}}

\newcommand{\eps}{\varepsilon}
\newcommand{\indx}{{\mathfrak F}}

\newcommand{\vq}{\vec{q}}
\newcommand{\vx}{\vec{x}}

\newcommand{\cV}{\mathcal{V}}
\newcommand{\cC}{\mathcal{C}}
\newcommand{\cW}{\mathcal{W}}

\newcommand{\R}{\mathbb{R}}
\newcommand{\cT}{\mathcal{T}}

\newcommand{\be}{\begin{equation}}
\newcommand{\ee}{\end{equation}}

\newtheorem{proposition}[theorem]{Proposition}

\newcommand{\scA}{{\mathscr{A}}}
\newcommand{\scL}{{\mathscr{L}}}

\DeclareMathOperator*{\Frac}{frac}

\graphicspath{{./figures/}}

\numberwithin{equation}{section}

\author{F.Y. Kuo, R. Scheichl, Ch. Schwab, I.H. Sloan and E. Ullmann}
\title{Multilevel Quasi-Monte Carlo Methods for Lognormal Diffusion Problems}
\date{} 
\begin{document}
\maketitle

\begin{abstract}
In this paper we present a rigorous cost and error analysis of a
multilevel estimator based on randomly shifted Quasi-Monte Carlo (QMC)
lattice rules for lognormal diffusion problems. These problems are
motivated by uncertainty quantification problems in subsurface flow. We
extend the convergence analysis in [Graham et al., Numer.~Math.~2014] to
\emph{multilevel} Quasi-Monte Carlo finite element discretizations and
give a constructive proof of the dimension-independent convergence of the
QMC rules. More precisely, we provide suitable parameters for the
construction of such rules that yield the required variance reduction for
the multilevel scheme to achieve an $\varepsilon$-error with a cost of
$\mathcal{O}(\varepsilon^{-\theta})$ with $\theta < 2$, and in practice
even $\theta \approx 1$, for sufficiently fast decaying covariance kernels
of the underlying Gaussian random field inputs. This confirms that the
computational gains due to the application of multilevel sampling methods
and the gains due to the application of QMC methods, both demonstrated in
earlier works for the same model problem, are complementary. A series of
numerical experiments confirms these gains. The results show that in
practice the multilevel QMC method consistently outperforms both the
multilevel MC method and the single-level variants even for non-smooth
problems.
\end{abstract}

\section{Introduction}
This paper gives a rigorous error analysis, together with numerical
experiments, for a multilevel Quasi-Monte Carlo scheme applied to linear
functionals of the solution of a typical model elliptic problem of
steady-state flow in random porous media. This problem is of central
importance in the development of efficient uncertainty quantification
tools for subsurface flow problems. The random elliptic partial
differential equation (PDE) reads
\begin{equation}\label{log-diff-primal}
- \nabla \cdot \big(a(\vx,\omega) \nabla u(\vx,\omega)\big) \,=\, f(\vx),\qquad
\textnormal{for } \ \vx \in D, \;\;\omega\in \Omega\,,
\end{equation}
where $D$ is a bounded domain in $\R^d$ for $d=1,2$ or $3$, and $\Omega$
is the sample space of a probability space $(\Omega,\mathcal{A},P)$, with
$\sigma$-algebra $\mathcal{A}$ and probability measure $P$.  A key
feature is the coefficient $a(\cdot,\omega)$, which is a lognormal random
field on the domain $D$.

In the context of flow through a porous medium, $u$ is the hydrostatic
pressure, $a$ is the permeability and $\vq := - a \nabla u$ is the
Darcy flux. This empirical relation between pressure and flux is known
as Darcy's law. When complemented by the conservation condition $\nabla
\cdot \vq = f$, where $f(\vx)$ is a deterministic source term, this leads
to \eqref{log-diff-primal}.

In this paper, the uncertain permeability is assumed to take the form
\begin{equation}\label{kle}
 a(\vx,\omega) \,=\, a_*(\vx) + a_0(\vx) \exp\bigg( \sum_{j=1}^\infty
 \sqrt{\mu_j}\, \xi_j(\vx)\, Y_j(\omega)\bigg), \quad \text{ with }
 Y_j  \sim \cN(0,1) \text{ i.i.d.}\,,
\end{equation}
where $a_*$ and $a_0$ are given deterministic functions on $D$,
satisfying $a_*(\vx)\ge 0$ and $a_0(\vx) > 0$. The sequence
$\{\mu_j\}$ of nonnegative values is assumed to be enumerated in
nonincreasing order, accumulating only at zero, and the sequence
$\{\xi_j\}$ is $L_2(D)$-orthonormal. If they correspond to the eigenvalues
and eigenfunctions of the covariance operator of a correlated Gaussian
random field, then the infinite sum under the bracket in \eqref{kle} is
known as the Karhunen-Lo\`{e}ve (KL) expansion of this Gaussian random
field (see e.g. \cite{Loeve:1978}).

For simplicity, we only study this problem subject to deterministic
boundary conditions. In general, we may have mixed Dirichlet/Neumann
conditions. Let the boundary $\Gamma = \partial D$ be partitioned into two
open, disjoint parts $\Gamma_\cD$ and $\Gamma_\cN$, and let $\vec{n}(\vx)$
denote the exterior unit normal vector to $D$ at $\vx \in \Gamma_\cN$.
Then we set
\begin{align}\label{eq:Diri}
 u(\vx,\cdot) &\,=\,  \phi_\cD(\vx)
 &&\textnormal{for } \ \vx \in \Gamma_\cD \,,\\ \label{eq:Neum}
 \vec{n}(\vx) \cdot \big(a(\vx,\cdot) \nabla u(\vx,\cdot)\big)
 &\,=\, \phi_\cN(\vx)
 &&\textnormal{for } \ \vx \in  \Gamma_\cN \,.
\end{align}
For $d=2,3$, we assume $D$ to be Lipschitz polygonal/polyhedral and each
of $\Gamma_\cD$ and $\Gamma_\cN$ to consist of the union of a finite
number of edges/faces.

Our goal is to obtain statistical information on certain linear
functionals $\calG$ of the solution~$u$ to \eqref{log-diff-primal};
we write $F:= \calG(u)$. In particular, we are interested in the
expected value $\mathbb{E}[F] = \mathbb{E}[\mathcal{G}(u)]$ (with
respect to the probability measure $P$). We need to perform several 
discretisation/truncation steps to obtain \textit{computable}
approximations to $\mathbb{E}[F]$:
\begin{enumerate}
\item[(a)] For a sample $\omega$, we employ a standard Galerkin finite
    element (FE) method with continuous, piecewise linear elements to
    discretise the solution to the PDE \eqref{log-diff-primal} on a
    family of simplicial meshes $\cT_h$ parametrised by their mesh
    size $h$. We approximate entries of the element stiffness matrices
    by a one-point Gauss rule, that is, we evaluate the coefficient at
    the mid point of each mesh element. We denote the FE approximation
    on $\cT_h$ by $u_h$.
\item[(b)] We truncate the KL expansion of $\log (a-a_*)$ in
    \eqref{kle} after a finite number of $s$ terms; we denote the
    $s$-term truncated diffusion coefficient by $a_s$ and the
    corresponding PDE solution by $u_s$. The FE approximation to
    \eqref{log-diff-primal} on $\cT_h$ with $a$ replaced by $a_s$ then
    reduces to a function $u_{h,s}$ of $s$ i.i.d.\ standard Gaussian
    random variables $Y_{j}$, $j=1,\dots,s$. Denoting
    the approximation of $F$ by $F_{h,s}:=\calG(u_{h,s})$, the
    expected value $\mathbb{E}[F]$ is then approximated by
\begin{equation}\label{EG}
 \mathbb{E}[F_{h,s}] \,=\, \int_{\mathbb{R}^s}
 \mathcal{G}(u_{h,s}(\cdot,\bsy)) \prod_{j=1}^s\phi(y_j) \,\rd \bsy\,,
\end{equation}
where $\phi(y)$ denotes the standard Gaussian probability density
function. In porous media flow applications,
the truncation dimension $s$ is often very large.
\item[(c)]
The $s$-dimensional Gaussian integral in \eqref{EG} is then
    approximated by an $N$-point quadrature rule, for example a Monte
    Carlo, sparse grid or Quasi-Monte Carlo rule, or by a multilevel
    variant (see below).
\end{enumerate}

In this paper the quadrature rules are derived from suitable Quasi-Monte
Carlo (QMC) rules (i.e. equal weight rules on the $s$-dimensional unit
cube), as we explain in the next section. The single-level variants
of these rules, as estimators for \eqref{EG}, were analysed for
  the same model problem in the
paper \cite{gknsss:2012} (see also the earlier
paper \cite{kss:2012} for the uniform case). Much emphasis was placed
there on the design of QMC rules that achieve dimension-independent error
bounds with good convergence rates and under weak assumptions.

Multilevel methods were introduced by \cite{Heinrich:2001,Giles:2008}. In
the present context \textit{multilevel Monte Carlo
\textnormal{(}MLMC\textnormal{)} estimators} for \eqref{EG} (multilevel
methods based on Monte Carlo integration) have attracted attention because
of their capacity to reduce the cost without loss of accuracy. The idea of
using such multilevel estimators for the approximation of $\mathbb{E}[F]$
was established in \cite{BarthSchwabZollinger:2011,CGST:2011} and, for the
lognormal case, analysed subsequently in \cite{CST:2011,TSGU:2012}.

The multilevel method is based on a sequence of ${L+1}$ FE approximations
of increasing accuracy as $\ell$ runs from $0$ to $L$, with mesh diameters
$h_\ell$ satisfying $h_0
> h_{1} > \ldots > h_L$.  At level $\ell$ we also truncate the KL expansion
after $s_\ell$ terms,  with $s_0 \le s_1 \le \cdots \le s_L$. With the
level $\ell$ approximation of our output functional $F$  denoted by
 $F_\ell := F_{h_\ell,s_\ell}$,  we can write $F_L$ as the
telescoping sum
\begin{equation}\label{eq:collapsing}
  F_L \,=\, F_0 + \sum_{\ell=1}^L (F_\ell - F_{\ell-1})\,.
\end{equation}
Then by linearity of the expectation operator we have
\begin{equation}\label{eq:expect}
  \mathbb{E}[F_L]
  \,=\, \mathbb{E}[F_0] + \sum_{\ell=1}^L \mathbb{E}[F_\ell - F_{\ell-1}]\,.
\end{equation}
In the MLMC scheme each term is approximated by an independent Monte
Carlo calculation, with a resulting gain in efficiency arising from the
fact that the differences $F_\ell - F_{\ell-1}$ on the higher levels,
although more expensive to compute,
have smaller variance and so require fewer Monte Carlo samples.

In this paper, each of the $L+1$ terms in \eqref{eq:expect} is instead
approximated by a different QMC rule, where the number of quadrature
points can again be chosen to decrease with~$\ell$. For sufficiently
smooth integrands, QMC quadrature rules offer the prospect of a higher
accuracy for the same computational cost compared to standard Monte Carlo
quadrature, or a lower cost for the same accuracy. Hence, the goal of this
paper is to explore the combination of multilevel estimators and QMC
methods by constructing and analysing a \textit{multilevel Quasi-Monte
Carlo \textnormal{(}MLQMC\textnormal{)} estimator} for the approximation
of \eqref{EG}. It was first observed in the context of stochastic
differential equations in \cite{GilesWaterhouse:2009} that the two gains
can be complementary.

In the context of \eqref{log-diff-primal}, single- and
multi-level QMC FE approximations were analysed also in the recent papers
\cite{kss:2015,DKLS14_1319}, but for the simpler case of \emph{uniform}
and \emph{affine} parameter dependence: in those papers the random
variables $Y_j$ appeared linearly in the differential operator, and their
values were assumed to be uniformly distributed on a bounded interval. The
lognormal case considered here is technically more involved and the
error bounds for the QMC rules developed here differ essentially from
those for the uniform case. They require, for example, so-called
``mixed regularity'' of the solution of \eqref{EG}. As shown here,
this mandates stronger assumptions on the data than those required for
MLMC or single-level QMC. The importance of this mixed regularity has
already been recognised in \cite{Harbrecht:2013}. In the present
paper, we establish for the first time $s$-independent quadrature
error bounds for MLQMC estimators and present detailed numerical
experiments indicating that MLQMC methods can outperform single-level
QMC and MLMC methods in terms of accuracy versus computational
cost. Some numerical experiments have also been reported in \cite{Robbe_etal:2015}.

The structure of this paper is as follows.  Section~\ref{sec:qmc}
explains the mechanics of QMC methods, without entering into the
question of approximation quality. Section~\ref{sec:ml-qmc} introduces
the multilevel QMC method (MLQMC), establishes an abstract convergence
theorem, compares the complexity of MLQMC to other estimators, and
discusses practical aspects and a practical
implementation. Section~\ref{sec:numerical} presents numerical results
which confirm the theoretical results. All technical parts related to
the necessary QMC convergence and construction theory are relegated to
Section~\ref{sec:theory}.

\section{Quasi-Monte Carlo Quadrature}
\label{sec:qmc}

Quasi-Monte Carlo quadrature rules are equal weight quadrature rules
for integrals over the $s$-dimensional unit cube $[0,1]^s$. For this
reason we introduce a change of variables $\bsy =
\boldsymbol{\Phi}_s^{-1}(\bszeta)$, where
$\boldsymbol{\Phi}_s^{-1}(\bszeta) :=
[\Phi^{-1}(\zeta_1),\Phi^{-1}(\zeta_2),\ldots,\Phi^{-1}(\zeta_s)]^{\tt
  T}$ denotes the inverse cumulative normal distribution applied to
each component of $\bszeta \in [0,1]^s$. We then obtain from
\eqref{EG} the expression
\begin{equation}\label{unit_EG}
 \mathbb{E}[F_{h,s}] \,=\, \int_{[0,1]^s}
 F_{h,s}(\boldsymbol{\Phi}_s^{-1}(\bszeta)) \,\rd \bszeta \,.
\end{equation}
For the approximation of $\mathbb{E}[F_{h,s}]$ in a single-level scheme,
we employ a specific kind of QMC quadrature rule, namely, the
\textit{shifted rank-$1$ lattice rule} given by
\begin{equation}\label{SL-r}
  {\cQ_{s,N}(F_{h,s};\bsDelta)} \,:=\, \frac{1}{N} \sum_{i=1}^N
 F_{h,s}\left(\boldsymbol{\Phi}_s^{-1}\left(\Frac\left(\frac{i \,
 \bsz}{N}+ \bsDelta \right)\right)\right)\,, \quad i=1,\dots,N\,,
\end{equation}
where $\bsz \in \mathbb{N}^s$ is the associated \textit{generating vector}
and $\bsDelta \in [0,1]^s$ is the \textit{shift}. The symbol
$\Frac(\cdot)$ denotes the fractional part function, which is to be
applied to every component of the $s$-dimensional input vector.  For the
general theory and fast construction of QMC lattice rules for the
$s$-dimensional cube, see e.g., \cite{DKS:2013} as well as
\cite{NC06,CKN06,DPW08}. For the particular case of integrals defined
initially over $\R^s$, see e.g., \cite{KSWWat10, NK14}.

The purely deterministic estimator \eqref{SL-r} for $\mathbb{E}[F_{h,s}]$
is biased. To remove this \textit{statistical bias} we construct the
associated \textit{randomly shifted lattice rule} where the random shift
$\bsDelta$ is uniformly distributed over $[0,1]^s$. We then use the sample
average of $\cQ_{s,N}(F_{h,s};\bsDelta)$ over a fixed, finite number $R$
of shift realizations as an estimator for $\mathbb{E}[F_{h,s}]$. We arrive
at
\begin{equation}\label{SL-rQMC}
 \cQ_{s,N,R}(F_{h,s}) \,:=\, \frac{1}{R} \sum_{k=1}^R \cQ_{s,N}(F_{h,s};\bsDelta_k) \,,
\end{equation}
where $\cQ_{s,N}(F_{h,s};\bsDelta_k)$ is defined in \eqref{SL-r},
$k=1,\dots,R$. Now, let $\mathbb{E}_\Delta[\cdot]$ denote the expected
value with respect to one or more random shifts.
Since
\begin{equation}
\begin{split}
\label{eq:unbiased}
\mathbb{E}_{\Delta}[\cQ_{s,N}(F_{h,s};\Delta)]
&\,=\,
\int_{[0,1]^s} \frac{1}{N} \sum_{i=1}^N F_{h,s}\left(\boldsymbol{\Phi}_s^{-1}
\left(\text{frac}\left (\frac{i\, \bsz}{N}+\Delta\right)\right)\right) \,\rd\Delta \\
&\,=\,
\frac{1}{N} \sum_{i=1}^N \int_{[0,1]^s} F_{h,s}(\boldsymbol{\Phi}_s^{-1}(\Delta)) \,\rd\Delta
\,=\,
\mathbb{E}[F_{h,s}] \,, \nonumber
\end{split}
\end{equation}
the quantity in \eqref{SL-rQMC} is an unbiased estimator for
$\mathbb{E}[F_{h,s}]$. However, \eqref{SL-rQMC} is not an unbiased
estimator for $\mathbb{E}[F]$, because the error arising from FE
approximation and from truncation of the KL expansion of $\log
(a-a_*)$ cannot be removed by randomisation of \eqref{SL-r}.
Specifically, the error analysis for randomly shifted lattice rules is
carried out in terms of the root mean square error (RMSE)
\begin{equation}\label{rmse}
e\big(\cQ_{s,N,R}(F_{h,s})\big) \,:=\, \sqrt{\mathbb{E}_{\Delta}
\big[\big(\cQ_{s,N,R}(F_{h,s})-\mathbb{E}[F]\big)^2\big]} \,. 
\end{equation}
Since the random diffusion coefficient $a$ in \eqref{log-diff-primal} is
statistically independent of the random shift in the QMC quadrature rule,
it is easy to see that in the single-level scheme we can split the
RMSE as follows
\begin{equation}\label{rmse:split}
e\big(\cQ_{s,N,R}(F_{h,s})\big)^2 \,=\,
\mathbb{E}_{\Delta}\big[\big(\cQ_{s,N,R}(F_{h,s})-\mathbb{E}[F_{h,s}]\big)^2\big]
+ \big(\mathbb{E}[F_{h,s}-F]\big)^2 \,.
\end{equation}

The second term in \eqref{rmse:split} is usually referred to as
\textit{bias} and can be decreased by choosing a fine enough FE mesh width
$h$ and by including a sufficiently large number $s$ of terms in the KL
expansion of $\log (a-a_*)$, as discussed in \cite{gknsss:2012}. The
first term in \eqref{rmse:split} is the (shift-averaged) QMC quadrature
error; it was analysed in detail in \cite{gknsss:2012} where the crucial
question of choosing the integer vector $\bsz$ in \eqref{SL-r} was fully
addressed.

\section{Multilevel Quasi-Monte Carlo Scheme}
\label{sec:ml-qmc}
Following the MLMC scheme, see \cite{BarthSchwabZollinger:2011,CGST:2011}
and the subsequent MLQMC scheme for the uniform case, see
\cite{kss:2015}, we construct a \textit{multilevel Quasi-Monte Carlo
estimator} for $\mathbb{E}[F]$ by combining estimators of the form
\eqref{SL-rQMC} on a hierarchy of \textit{levels}.

To define our multilevel method, let us assume that we have a nested
sequence of FE spaces $V_{h_0}, V_{h_1},\ldots, V_{h_L}$ of increasing
dimension and let $\cT_{h_0},\cT_{h_1},\ldots,\cT_{h_L}$ be the
corresponding sequence of shape-regular, conforming, simplicial
meshes (i.e., simplicial
partitions of the domain $D$ for which intersections of any two
$d$-simplices are are either empty, an entire side, or an entire face). We
assume that the mesh diameters are strictly decreasing, i.e., $h_\ell >
h_{\ell+1}$\,. Furthermore, we include only the leading $s_\ell$ terms in
the KL expansion of $\log a$ on level $\ell$, subject to the condition
$s_\ell \le  s_{\ell+1}$\,. The approximation of our output functional $F$
that we obtain on level $\ell$ is denoted by $F_\ell := F_{h_\ell,s_\ell}$
as in \eqref{eq:collapsing} and for convenience we set $F_{-1} := 0$.
We can then write \eqref{eq:expect} as
\[
\mathbb{E}[F_L] \,=\, \sum_{\ell=0}^L \mathbb{E}[F_\ell - F_{\ell-1}] \,.
\]
That is, the expected value of the output quantity of interest on the
finest mesh is equal to the expectation on the coarsest mesh, plus a
series of corrections, namely the expected value of the difference of
quantities computed on consecutive FE meshes. We estimate the expected
value $\mathbb{E}[F_\ell - F_{\ell-1}]$ on level $\ell$ by means of the
randomly shifted lattice rule estimator $\cQ_\ell :=
\cQ_{s_\ell,N_\ell,R_\ell}$ defined in \eqref{SL-rQMC} and \eqref{SL-r},
with $N_\ell$ quadrature points and $R_\ell$ random shifts from a uniform
distribution on $[0,1]^{s_\ell}$. The MLQMC estimator for $\mathbb{E}[F]$
then reads
\begin{equation}
\label{ML-rQMC}
 \cQ_{L}^{\rm ML} (F) \, := \, \sum_{\ell=0}^L \cQ_\ell(F_\ell -
 F_{\ell-1}) \,= \, \sum_{\ell=0}^L  \frac{1}{R_\ell}
 \sum_{k=1}^{R_\ell} \frac{1}{N_{\ell}} \sum_{i=1}^{N_{\ell}}
 \left(F_{\ell}(\bsy^{(i,k)}_{\ell})-F_{\ell-1}(\bsy^{(i,k)}_{ \ell } )\right)\,,
\end{equation}
where
$\bsy^{(i,k)}_{\ell}:=\boldsymbol{\Phi}^{-1}_{s_\ell}\big(\Frac\big(i
\bsz_\ell N_\ell^{-1}+\bsDelta_{\ell,k}\big)\big)$ and $\bsz_\ell$ is
the generating vector on level $\ell$ (that will in general be
different from level to level).

Let us define the variance with respect to the shifts $\bsDelta_{\ell,k}$
by
\[
\mathbb{V}_{\Delta}[\cQ_\ell(F_\ell - F_{\ell-1})] \,=\,
\mathbb{E}_{\Delta}\big[\big(\cQ_\ell(F_\ell -
F_{\ell-1})-\mathbb{E}[F_\ell - F_{\ell-1}]\big)^2\big]\,.
\]
Then, since each correction $\mathbb{E}[F_\ell - F_{\ell-1}]$,
$\ell=0,\dots,L,$ is estimated using statistically independent random
shifts, the RMSE of the MLQMC estimator satisfies
\begin{equation}
 e\left(\cQ_{L}^{\rm ML}(F)\right)^2
 \,:=\, \mathbb{E}_\Delta\big[\big(\cQ^{\rm ML}_L(F) -
   \mathbb{E}[F]\big)^2\big]
 \,= \, \sum_{\ell=0}^L \mathbb{V}_{\Delta}[\cQ_\ell(F_\ell - F_{\ell-1})] \,+\,
 \big(\mathbb{E}[F_L - F]\big)^2 \,. \label{eq:bias}
\end{equation}
The second term in \eqref{eq:bias} is the bias introduced by KL
truncation and by FE approximation. It coincides with the
second term of the single-level error in \eqref{rmse:split} for $h=h_L$
and $s=s_L$.

\subsection{Error versus cost analysis}
\label{sec:Cost}

We now extend the cost analysis in \cite[Thm.~1]{CGST:2011} to the
MLQMC estimator $\cQ_L^{\rm ML}(F)$ defined in
\eqref{ML-rQMC}.  We aim at estimating the computational cost, denoted
below by $\mathrm{cost}\left(\cQ_L^{\rm ML}(F)\right)$, necessary to
ensure that the RMSE in \eqref{eq:bias} satisfies\footnote{Throughout the
paper, the notation $A\lesssim B$ indicates that there exists a constant
$c>0$ such that $A\le c B$. The notation $A\eqsim B$ indicates that
$A\lesssim B$ and $B \lesssim A$.} $e \left(\cQ_L^{\rm ML}(F)\right)
\lesssim \eps$, as $\eps\downarrow 0$.
A similar extension of this abstract result has
recently been proved in the context of multilevel stochastic collocation
methods in \cite{MLSC}. However, our result here is tailored to MLQMC and
includes the truncation error which was ignored in \cite{MLSC}.

We assume the number of degrees of freedom $M_\ell :=
\text{dim}(V_{h_\ell})$, associated with the FE approximation
$F_{\ell}:=F_{h_\ell,s_\ell}$ on level $\ell=0,\dots,L$, satisfies
\begin{equation}\label{ass:dofs}
 M_\ell \,\eqsim\, h_\ell^{-d}\;.
\end{equation}
The assumption \eqref{ass:dofs} includes quasi-uniform families of meshes
and meshes with local refinement near corners or edges of the domain.

Apart from the negligible post-processing cost to compute the quantity of
interest, the cost of computing one sample
$F_{h_\ell,s_\ell}(\bsy^{(i,k)}_{\ell})$ on level $\ell$ is
$\calC^{\text{perm}}_\ell + \calC^{\text{solve}}_\ell$, where
$\calC^{\text{perm}}_\ell$ denotes the cost of evaluating the
$s_\ell$-term truncation $a_{s_\ell}$ of the
permeability field \eqref{kle} at all quadrature points for each of the
$\mathcal{O}(h_\ell^{-d})$ elements of the FE mesh, and
$\calC^{\text{solve}}_\ell$ denotes the cost of solving a sparse linear
equation system with $M_\ell$ unknowns. We assume that
\[
\calC^{\text{perm}}_\ell \,\lesssim\, s_\ell\, h_\ell^{-d} \quad \text{and}
\quad \calC^{\text{solve}}_\ell \,\lesssim\, h_\ell^{-\gamma}\,,
\quad \text{with} \ \ \gamma \ge d\,.
\]
In the case of a robust (algebraic) multigrid solver, we have $\gamma
= d + \delta$, for arbitrarily small $\delta > 0$. In fact, the number
of iterations for a robust multigrid solver typically grows only
logarithmically with $M_\ell$ and the cost per iteration is
$\mathcal{O}(M_\ell)$ (cf.~\cite{Vass:2008} and the references therein).

We will first state an abstract complexity theorem in which we make
only very limited assumptions. To avoid having to treat the case $\ell
= 0$ separately, in the ensuing assumptions M1 -- M3 we adopt the
convention $h_{-1} := 1$, $s_{-1} := 1$, and recall that $F_{-1} := 0$.
\begin{theorem}
\label{thm:mlqmc} Suppose that $\mathbb{E}_\bsDelta[\cQ_\ell (F_\ell -
F_{\ell-1})] = \mathbb{E}(F_\ell - F_{\ell-1})$ and that there are
nonnegative constants $\alpha,\alpha', \beta, \beta', \gamma, \lambda$
such that
\begin{itemize}
\item[\bf{M1}.] $\displaystyle \big| \mathbb{E}[F_{L} - F] \big|
    \ \lesssim \ h_L^{\alpha} + s_L^{-\alpha'} $,
\item[\bf{M2}.] $\displaystyle \mathbb{V}_\bsDelta[\cQ_\ell (F_\ell -
    F_{\ell-1})] \
  \lesssim \ R_\ell^{-1} N_\ell^{-1/\lambda} \,
   \left(h_{\ell-1}^{\beta} + (1-\delta_{s_\ell,s_{\ell-1}})
   s_{\ell-1}^{-\beta'} \right), \, $
\item[\bf{M3}.]
$\displaystyle \mathrm{cost}(\cQ_{\ell}(F_\ell - F_{\ell-1})) \ \lesssim
   \ R_\ell\, N_\ell \, \big(s_\ell h_\ell^{-d} + h_{\ell}^{-\gamma}\big)$,
\end{itemize}
for all \ $0 \le \ell \le L$, and where $\delta_{\cdot,\cdot}$ denotes the
Kronecker delta. Then\vspace{-1ex}
\begin{align*}
e\big(\cQ^\mathrm{ML}_L(F)\big)^2 &\;\lesssim\; h_L^{2\alpha} + s_L^{-2\alpha'} +
    \sum_{\ell=0}^L R_\ell^{-1} N_\ell^{-1/\lambda} \,
    \left(h_{\ell-1}^{\beta} + (1-\delta_{s_\ell,s_{\ell-1}}) s_{\ell-1}^{-\beta'} \right) \;\mbox{ and
                                      }\; \\
    \mathrm{cost}\big(\cQ^\mathrm{ML}_L(F)\big) &\;\lesssim\; \sum_{\ell=0}^L R_\ell\,
    N_\ell \left(s_\ell\, h_\ell^{-d} + h_{\ell}^{-\gamma}\right)\,.
\end{align*}
\end{theorem}
\begin{proof}
The proof follows immediately from \eqref{eq:bias} and the definition of
$\mathrm{cost}\big(\cQ^\mathrm{ML}_L(F)\big)$.
\end{proof}

We will now focus on a specific application of this theorem, with a fixed
number of terms in the KL expansion. We assume that the sampling cost is
the dominant part, which ultimately is the case with an optimal multigrid
solver in the limit as the error tolerance goes to zero. We are not
considering the case where the number of KL terms on the coarser levels is
decreased, even though this may in some cases reduce the overall
asymptotic cost of the multilevel algorithm, because it would lead to a
very complicated complexity theorem and the analysis of Assumption M2 in
Section~\ref{sec:theory} would become significantly more involved.
\begin{corollary} \label{cor:mlqmc}
Let $\gamma \le d + \alpha/\alpha'$ and let the assumptions of
Theorem~\ref{thm:mlqmc} hold. If we choose $ h_\ell \eqsim 2^{-\ell}$,
$R_\ell = R$ and $s_\ell = s_L \eqsim h_L^{-\alpha/\alpha'}$ for some $R
\in \mathbb{N}$ and for $\ell=0,\ldots, L$, then for any $\varepsilon >
0$, there exists a choice of  $L$ and of $N_0, \dots, N_L$  such that
\begin{align} \label{eq:cost}
 e\big(\cQ^\mathrm{ML}_L(F)\big)^2 \lesssim \varepsilon^2 \ \
  \text{and} \ \
 {\rm cost}\big(\cQ^\mathrm{ML}_L(F)\big)
  \lesssim
  \begin{cases}
  \varepsilon^{-2\lambda-1/\alpha'}
  & \mbox{when } \beta \lambda > d \,,
  \\
  \varepsilon^{-2\lambda-1/\alpha'}\, (\log_2 \varepsilon^{-1})^{\lambda+1}
  & \mbox{when } \beta \lambda = d\,,
  \\
  \varepsilon^{-2\lambda -1/\alpha' - (d-\beta\lambda)/\alpha}
  & \mbox{when } \beta \lambda < d \,.
  \end{cases}
\end{align}
\end{corollary}

\begin{proof}
Using the particular choices for $h_\ell$, $s_\ell$ and $R_\ell$ and the
assumption that $\gamma \le d+\alpha/\alpha'$, we obtain
\begin{equation}
\label{bounds:proofa}
e\big(Q^\mathrm{ML}_L(F)\big)^2 \lesssim  h_L^{2\alpha} +
\sum_{\ell=0}^L N_\ell^{-1/\lambda} \, h_\ell^{\beta}
\quad \mbox{and}\quad
\mathrm{cost}\big(Q^\mathrm{ML}_L(F)\big) \lesssim
h_L^{-\alpha/\alpha'} \sum_{\ell=0}^L N_\ell \, h_\ell^{-d}\;.
\end{equation}

Thus, a sufficient condition for the MSE to be bounded by a constant
times $\varepsilon^2$ is that each of the two terms in the above error
bound is $\mathcal{O}(\varepsilon^2)$, which in particular leads to
the choice $2^{-L} \eqsim h_L \eqsim \varepsilon^{1/\alpha}$ to bound
the bias error, and thus
\begin{equation} \label{eq:def_L}
L = \left\lceil \frac{1}{\alpha} \log_2(\varepsilon^{-1}) + c_1 \right\rceil
\end{equation}
for some constant $c_1 \in \mathbb{R}$ that is independent of $\varepsilon$.

We now equate sampling and bias error to within a constant factor
$c_2>0$, again independent of $\varepsilon$ and of $\ell$. To minimize
the cost subject to this constraint, we consider the functional
\[
g(N_0,\ldots,N_L,\mu) \,:=\, h_L^{-\alpha/\alpha'} \sum_{\ell=0}^L
N_\ell\, h_\ell^{-d} \;+\; \mu \left( \sum_{\ell=0}^L N_\ell^{-1/\lambda}
h_{\ell}^{\beta} \;-\ c_2 h_L^{2\alpha} \right),
\]
where $\mu$ is a Lagrange multiplier and where we treat $N_0, \ldots,
N_L$ as continuous variables. We look for its stationary point. This
leads to the first--order, necessary optimality conditions
\begin{align}
\label{opt1}
  \frac{\partial g}{\partial N_\ell}
  &\,=\, h_L^{-\alpha/\alpha'}  h_\ell^{-d}
  -  \frac{\mu}{\lambda} N_\ell^{-1/\lambda-1} h_\ell^{\beta} \,=\, 0
  \qquad\mbox{for}\quad \ell=0,\ldots,L\,.
\\
\label{opt2}
 \frac{\partial g}{\partial \mu} &\,=\, \sum_{\ell=0}^L N_\ell^{-1/\lambda}
  h_{\ell}^{\beta} - c_2 h_L^{2\alpha} = 0
\,.
\end{align}
Rearranging \eqref{opt1}, we see that $N_\ell^{{1/\lambda+1}} h_\ell^{-(d
+ \beta)}$ is independent of $\ell$. Therefore, the numbers of QMC points
should be chosen according to
\begin{equation}
\label{eq:def_N}
N_\ell \,=\, \left\lceil N_0
  \left(\frac{h_\ell}{h_0}\right)^{{(d+\beta)\lambda/(\lambda+1)}} \right\rceil
 \qquad\text{for}\quad \ell= 1,\ldots,L\,.
\end{equation}
A suitable choice for $N_0$ can then be deduced from \eqref{opt2}.
Substituting \eqref{eq:def_N} into \eqref{opt2}
 and using the fact that $h_0 \eqsim 2^0 = 1$, we obtain
$N_0^{1/\lambda} \eqsim 2^{2\alpha L} \sum_{\ell=0}^L
h_\ell^{(\beta\lambda-d)/(\lambda+1)}\,.$ Since $h_\ell\eqsim 2^{-\ell}$,
it follows from properties of geometric series that
\begin{align}
\label{geom_sum}
  \sum_{\ell=0}^L h_\ell^{(\beta\lambda-d)/(\lambda+1)}
  \,\eqsim\,
  \sum_{\ell=0}^L 2^{\ell (d-\beta\lambda)/(\lambda+1)}
  &\,\eqsim\, E_L \,:=\,
  \begin{cases}
  1
  & \mbox{when } \beta\lambda > d\,,
  \\
  L
  & \mbox{when } \beta\lambda = d\,,
  \\
  2^{L (d-\beta\lambda)/(\lambda+1)}
  & \mbox{when } \beta\lambda < d\,.
  \end{cases}
\end{align}
and hence
\begin{equation}
 \label{eq:def_N0}
 N_0 \;\eqsim\;  2^{L (2\alpha\lambda)} \, E_L^{\lambda}\,.
\end{equation}
Finally, we substitute \eqref{eq:def_N} and \eqref{eq:def_N0} into
\eqref{bounds:proofa} and use \eqref{geom_sum} to bound that cost
asymptotically, as $L\to\infty$, by
\begin{align*}
  {\rm cost}(\cQ^\mathrm{ML}_L(F))
  \,\lesssim\, h_L^{-\alpha/\alpha'} N_0 \sum_{\ell=0}^L
  h_\ell^{{(\beta\lambda-d)/(\lambda+1)}} \, \lesssim \,
  \begin{cases}
  2^{L(2\alpha\lambda+\alpha/\alpha')}
  & \mbox{when } \beta\lambda > d\,,
  \\
  2^{L(2\alpha\lambda +\alpha/\alpha')} L^{\lambda+1}
  & \mbox{when } \beta\lambda = d\,,
  \\
  2^{L(2\alpha\lambda +\alpha/\alpha' + d -\beta\lambda)}
  & \mbox{when } \beta\lambda < d \,.
  \end{cases}
\end{align*}
The bound in \eqref{eq:cost} then follows from \eqref{eq:def_L},
i.e., using the relation $2^L \eqsim \eps^{-1/\alpha}$.
\vspace{1.5ex}
\end{proof}
\subsection{Discussion and comparison with other estimators}
\label{sec:discussion}
First, let us check the assumptions in Theorem~\ref{thm:mlqmc}
for the lognormal model problem \eqref{log-diff-primal}.
\begin{itemize}
\item We observe that Assumption M1 relates only to the FE error and
  the KL truncation error, and is not specific to MLQMC. It has been
  studied extensively in \cite{CST:2011, aretha_paper, TSGU:2012,
    gknsss:2012}. The assumptions on the data in Section
  \ref{sec:theory}, in particular on the regularity of the input
  random field $a(\cdot,\omega)$ and of the functional $\mathcal{G}$,
  imply $\alpha = 2$. For non-convex domains $D$, this requires
  special sequences of meshes and an analysis in weighted spaces (see
  Proposition~\ref{prop:BilRegPrp} in Section~\ref{sec:Prel} which can
  also be used to bound the FE bias error). The value for $\alpha'$
  depends on the rate of decay of the KL eigenvalues. Under suitable
  regularity assumptions on the data, it was shown in
  \cite{Charrier:2013} that, for Gaussian fields with Mat\'ern
  covariance and smoothness parameter~$\nu$ (for a precise definition
  see Section 4), any $\alpha' < 2\nu/d$ can be chosen.
\item As shown in Section~\ref{sec:qmc}, the assumption that
  $\mathbb{E}_\bsDelta[\cQ_\ell (F_\ell - F_{\ell-1})] =
  \mathbb{E}(F_\ell - F_{\ell-1})$ is satisfied for our randomised QMC rules.
\item The main theoretical result of this paper, postponed  to Section
  \ref{sec:theory}, is to provide a proof of Assumption M2 for
  appropriate QMC rules. We will see there that this assumption can
  usually be satisfied for linear functionals, with $\beta = 2\alpha$
  and with $\lambda\in(1/2,1)$, for the case where $s_\ell =
  s_{\ell-1}$. The value of $\lambda$, for a sufficiently good choice
  of the QMC rules, depends on the parametric regularity of
  $a(\cdot,\omega)$. In particular, $\lambda$ can be chosen
  arbitrarily close to $1/2$ in the case of lognormal fields with
  Mat\'ern covariance and large enough smoothness parameter $\nu$ (as
  we discuss below). 
\item Finally, if we use an optimal deterministic PDE solver, such as multigrid,
Assumption M3 is also satisfied with $\gamma = d + \delta$, for some
$\delta > 0$, but typically $\delta \ll \alpha/\alpha'$
and thus $\gamma \le d + \alpha/\alpha'$, as in Corollary \ref{cor:mlqmc}.
\end{itemize}

In practice, however, for the choices of parameters in Corollary
\ref{cor:mlqmc} and assuming $\gamma \approx d$, there is typically a
critical tolerance $\varepsilon_*>0$ such that $\calC^{\text{perm}}_\ell
\le \calC^{\text{solve}}_\ell$ for all $\varepsilon\ge\varepsilon_*$. In
that situation, we can drop the exponent $-1/\alpha'$ in \eqref{eq:cost}
for $\varepsilon\ge\varepsilon_*$. Especially for $d>1$, most practical
choices for the
tolerance $\eps$ in applications lie above this critical tolerance
$\varepsilon_*>0$. We shall call the quantity obtained by dropping the
$-1/\alpha'$ exponent the \textit{pre-asymptotic cost}. Note however, that
as seen in \cite{gknsss:2012}, the QMC quadrature error also exhibits a
pre-asymptotic behaviour. To obtain sharp bounds, the $\lambda$ in the
pre-asymptotic cost should be replaced by the numerically observed
effective rates $1/\lambda_{\text{eff}} \le 1/\lambda$ of the employed QMC
rules. Note that the same is true for the single-level QMC estimator. There the
cost is $\mathcal{O}(\eps^{-2\lambda - 1/\alpha' - d/\alpha})$ as $\eps
\to 0$, and $\mathcal{O}(\eps^{-2\lambda_{\text{eff}} - d/\alpha})$ for
$\eps \ge \eps_*$.

The analysis in \cite{CGST:2011,TSGU:2012} of standard multilevel Monte
Carlo (MLMC) methods for the lognormal case does not rely on the use of
truncated KL-expansions. Isotropic input random fields $a(\cdot,\omega)$,
such as those studied in Section \ref{sec:numerical}, can be sampled in
$\mathcal{O}(h^{-d} \log(h^{-d}))$ operations via circulant embedding
techniques (see, e.g., \cite{Graham_etal:2011}). In that case,
$\calC^{\text{perm}}_\ell \lesssim \calC^{\text{solve}}_\ell$ and so, with
an optimal multigrid solver, the total cost on level $\ell$ is
$\mathcal{O}(N_\ell h_\ell^{-\gamma})$, for any $\gamma
> d$ (for more details see Section \ref{sec:numerical}). Hence, assuming
$\beta \not=\gamma$, the cost of an optimal implementation of MLMC grows with $\mathcal{O}(\eps^{-2 -
\max(0,(\gamma-\beta)/\alpha)})$ and $\gamma > d$ arbitrarily close to $d$.

Nevertheless, for sufficiently large values of $\alpha'$ --
typical for lognormal fields with Mat\'ern covariance and sufficiently large
smoothness parameter $\nu$ -- we see that the presently proposed
MLQMC estimator has significantly lower cost than, for example, MLMC
estimators when $\lambda < 1$. We will see in
Section~\ref{sec:numerical}  that this holds in
practice, even for values of the Mat\'ern parameter $\nu$ below the
minimum required in the present convergence analysis.
\subsection{Practical aspects}
\label{sec:practical}
The formula \eqref{eq:def_L} for $L$ requires knowledge of the constant
$c_1$. When the error estimates are sharp, this can be computed a
priori, as we do in our numerical experiments below.  However,
 the FE discretization error, and thus the value of $L$,
can also be estimated dynamically (i.e., without computing
additional samples) from the estimates $\cQ_\ell(F_\ell - F_{\ell-1})$, as for
standard MLMC (see
\cite{Giles:2008,CGST:2011}).

Like standard Monte Carlo estimators, randomised lattice rules also
come with a simple variance estimator,
namely the sample variance with respect to the random shifts, i.e.,
\begin{equation}
\label{varestimate}
\mathbb{V}_\bsDelta[\cQ_\ell(F_\ell - F_{\ell-1})] \, \approx \, \frac{1}{R_\ell(R_\ell-1)}
\sum_{k=1}^{R_\ell} \left[\cQ_{s_\ell,N_\ell}(F_\ell - F_{\ell-1};\bsDelta_{\ell,k})
-
\cQ_{s_\ell,N_\ell,R_\ell}(F_\ell - F_{\ell-1})\right]^2
\,.
\end{equation}
However, (on-the-fly) estimates for the rate of convergence $1/\lambda$ of
the lattice rule (or for its effective rate $1/\lambda_{\text{eff}}$) are
very unreliable, and thus the formulae \eqref{eq:def_N} and
\eqref{eq:def_N0} for the optimal values of $N_\ell$ and $N_0$ in the
proof of Corollary \ref{cor:mlqmc} are of limited practical use.

From a computational point of view, \emph{extensible} lattice sequences or
\emph{embedded} lattice rules are useful, as they allow the results
already calculated to be ``recycled'' when adaptively choosing the number
of samples, see e.g., \cite{HN03,CKN06,DPW08}. To explore this
``nestedness'' property in practice, it is most convenient for the number
of points $N_\ell$ to be only powers of 2 (since then we always obtain
complete lattice rules and do not need to be concerned about how the
individual lattice points are ordered). A simple and effective algorithm
that ensures this and does not require knowledge of $\lambda$ is presented
in \cite{GilesWaterhouse:2009}. For completeness, let us recall the
algorithm. To simplify notation, we define for $\ell =0,\ldots,L$,
$\cV_\ell := \mathbb{V}_\Delta(\cQ_\ell(F_\ell - F_{\ell-1}))$ and
$\cC_\ell := \mathrm{cost}(\cQ_\ell(F_\ell - F_{\ell-1}))$.

\begin{algorithm}\label{alg1}\em
Let $L=0$.
\begin{enumerate}
\item
Set $N_L=1$ and estimate $\cV_L$ using \eqref{varestimate}.
\vspace{-1ex}
\item While $\sum\limits_{\ell=0}^L \cV_\ell  > \varepsilon^2$, double
    $N_\ell$ on the level $\ell$ for which the ratio
    $\cV_\ell/\cC_\ell$ is largest. \vspace{-1ex}
\item
If the bias estimate is greater than
$\varepsilon$ or $L <2$, set $L \to L+1$ and go to Step 1.
\end{enumerate}
\end{algorithm}
Note that this is a greedy algorithm that strives to equilibrate the
\emph{profit}, that is, the ratio of variance and cost, across levels.
Thus, in the limit as $\eps \to 0$, the numbers of samples $N_\ell$ on the
levels will be such that $\cV_0/\cC_0 \approx \cV_1/\cC_1 \approx \cdots
\approx \cV_L/\cC_L$. To show that this choice of $N_\ell$ leads to the
same overall cost for MLQMC as the theoretical algorithm in the proof of
Corollary~\ref{cor:mlqmc}, let us assume that $\cV_\ell =  v_\ell\,
N_\ell^{-1/\lambda}$ (+ higher order terms), for some $\lambda > 0$ and
for some $0 < v_\ell \lesssim h_\ell^{\beta}$ that is independent of
$N_\ell$. This is a stronger assumption than M2, but asymptotically it is
satisfied for our QMC rules. Crucially, we do not require values of
$\lambda$, $v_\ell$ or $\beta$ in the algorithm.

We may also assume $\cC_\ell = \kappa_\ell N_\ell$
+ lower order terms,
where, at leading order, the ``cost-per-sample''
$\kappa_\ell \eqsim h_L^{-\alpha/\alpha'} \, h_\ell^{-d}$ is
independent of $N_\ell$.
With these assumptions, we may
set up a constrained optimisation problem, as
in the proof of Corollary~\ref{cor:mlqmc}, minimising the total cost
subject to the constraint in Step 2 of the algorithm on the total variance
being less than $\eps^2$. However, here we write more abstractly
\[
\tilde{g}(N_0,\ldots,N_L,\tilde{\mu}) \,:=\, \sum_{\ell=0}^L
  \cC_\ell + \tilde{\mu} \left( \sum_{\ell=0}^L \cV_\ell - \eps^2 \right).
\]
We ignore the higher and lower order terms in $\cV_\ell$ and in $\cC_\ell$,
respectively, treat
the $N_0, \ldots, N_L$ as continuous variables again and differentiate
$\tilde{g}$ with respect to $N_\ell$ and $\tilde{\mu}$ to get
\begin{align}
\label{opt1a}
\frac{\partial \tilde{g}}{\partial N_\ell}
  &\,=\, \kappa_\ell  -  \frac{\tilde{\mu}}{\lambda} v_\ell\, N_\ell^{-1/\lambda-1} \,=\,
    \left(\cC_\ell - \frac{\tilde{\mu}}{\lambda} \cV_\ell\right)N_\ell^{-1} \,=\,  0
  \qquad\mbox{for}\quad \ell=0,\ldots,L\,,
\\
\label{opt2a}
 \frac{\partial \tilde{g}}{\partial \tilde{\mu}} &\,=\, \sum_{\ell=0}^L
                                               \cV_\ell - \eps^2 \,= \,0\,.
\end{align}
It follows from \eqref{opt1a} that $\cV_\ell/\cC_\ell =
\lambda/\tilde{\mu}$, which is independent of $\ell$, and so the profit is
indeed equilibrated across the levels for the optimal values of $N_\ell$.
The fact that the asymptotic cost scales as in
\eqref{eq:cost} can then be deduced as in the
proof of Corollary~\ref{cor:mlqmc}, choosing
\[
N_\ell \, = \, \left\lceil N_0 \bigg(\frac{\kappa_0 \, v_\ell}{v_0 \,
    \kappa_\ell}\bigg)^{\lambda/(\lambda+1)} \right\rceil_2\,, \quad \text{where} \ \
\lceil x \rceil_2 := 2^{\lceil \log_2(x) \rceil}\,,
\]
that is, we round $N_\ell$ up to the nearest power of $2$. Substituting
this into \eqref{opt2a}, using \eqref{eq:def_L} and the assumptions on
$v_\ell$ and $\kappa_\ell$,
we can deduce that the expression for the optimal value for
$N_0$ is as in \eqref{eq:def_N0} (but rounded to the nearest
power of 2). The bound on the cost follows as before.

For standard multilevel Monte Carlo it is possible to compare this
algorithm with the original algorithm in \cite{Giles:2008} that adaptively
approximates the optimal choices of samples $N_\ell$, and we will see in
Section~\ref{sec:numerical} that Algorithm~\ref{alg1} achieves almost
the same cost effectiveness as the original algorithm, even for fairly
large $\eps$.

\section{Numerical results}
\label{sec:numerical}

For all our numerical experiments we assume that the log-permeability
$\log a(x,\omega)$ in \eqref{kle} is a mean-zero Gaussian
field with Mat\'ern  covariance, that is, $a_* \equiv 0$, $a_0 \equiv 1$ and
$(\mu_j,\xi_j)$ are the eigenpairs of the integral operator $\int_D
\rho_\nu(|\vx-\vx'|)\, v(\vx')\,\rd \vx'$, with
\begin{equation}
\label{matern}
  \rho_\nu (r) \,:=\, \sigma^2 \frac{2^{1-\nu}}{\Gamma(\nu)}
  \left(2\sqrt{\nu} \frac{r}{\lambda_\mathcal{C}}\right)^\nu
  K_\nu\left(2\sqrt{\nu} \frac{r}{\lambda_\mathcal{C}} \right)\;,
\end{equation}
where $\Gamma$ is the gamma function and $K_\nu$ is the modified
Bessel function of the second kind. The parameter $\nu$ is a smoothness
parameter, $\sigma^2$ is the variance and $\lambda_\mathcal{C}$ is the
correlation length scale. In practice, we will
always truncate the sum in \eqref{kle} after a finite number of
$s$ terms.

To compute the eigenpairs
$(\mu_j,\xi_j)$, $1\leq j \leq s$, we discretize the integral operator above
using the Nystr\"{o}m method based on Gauss-Legendre quadrature on $[0,1]^d$
and then solve the resulting algebraic eigenvalue problem.

The numerical results were obtained on a 2.4GHz Intel Core i7
processor in Matlab R2014b.

\subsection{Results in space dimension one}\label{exa:1D}

We first consider problem \eqref{log-diff-primal} in one
dimension on $D=(0,1)$ with homogeneous Dirichlet boundary conditions
$u(0,\omega) = u(1,\omega) = 0$ and source term $f\equiv 1$.
This problem is identical to the one studied in \cite[Sect.~6]{gknsss:2012}.
For the discretization of the associated variational formulation
on level $\ell=0,\dots,L$ we use
piecewise linear, continuous FEs on a uniform
simplicial mesh of width $h_\ell = h_02^{-\ell}$, where
$h_0 = 2^{-\ell_0}$ for some $\ell_0 \in \mathbb{N}$, such that
$M_\ell = 2^{\ell+\ell_0} -1$. We generate samples of $\log a$ (and thus of
$a$) at the midpoints of the intervals constituting the FE mesh
using the KL expansion of $\log a$ with $s$ terms, and approximate the
entries of the stiffness matrix via the midpoint quadrature rule.
The output quantity of interest is chosen to be $F := u(1/3,\omega)$, i.e., the
solution evaluated at $x=1/3$.

In order to have a nondimensional error measure for
$\cQ^\mathrm{ML}_L(F)$, our MLQMC estimator for $\mathbb{E}[F]$ with
randomly shifted lattice rules, we define what is usually called the {\em
relative standard error} in the statistical literature, that is
\begin{equation}
\label{def:relstderr}
e_{\text{rel}}\big(\cQ^\mathrm{ML}_L(F)\big) := \left|\frac{e\big(\cQ^\mathrm{ML}_L(F)\big)}{\mathbb{E}[F]}\right|
\end{equation}
We then study, for different tolerances $\eps>0$, the computational cost
to achieve a relative standard error
$e_{\text{rel}}\big(\cQ^\mathrm{ML}_L(F)\big) \le \eps$ and compare it to
the cost to achieve the same relative standard error with standard MLMC,
as well as with the single-level versions of both algorithms. In all the
QMC estimators, for simplicity we use $R=16$ random shifts and an
embedded lattice rule with generating vector taken from the file
\cite[{\tt lattice-39102-1024-1048576.3600.txt}]{frances_web}. (We
remark that there is no theoretical justification to use this lattice rule
for our problem here, however, numerical experiments from
\cite{gknsss:2012} indicated that such generic lattice rules do perform
just as well as those specifically tuned to the problem.)

We restrict ourselves to smoothness parameters $\nu \ge 1$, where the numerically observed
FE error is $\mathcal{O}(h^2)$ (independent of $\nu$).\footnote{Note that
theoretically the FE error for point evaluations in one space dimension is
$\mathcal{O}(h^2\log|h|)$ (cf.~\cite{aretha_paper}), but we do not observe the
log-factor in practice.} To estimate
the bias error on the finest level $L$, we then assume the following
upper bound (with uniform constants $C_{\text{FE}}$ and $C_{\text{trunc}}$):
\begin{equation}
\label{upper_bound}
|\mathbb{E}[F_{h,s} -  F^*]| \; \le \;
|\mathbb{E}[F_{h,s} -  F_{h^*,s}]| + |\mathbb{E}[F_{h^*,s} -
F^*]| \; \le \; C_{\text{FE}} \, h^2 \; + \; C_{\text{trunc}} \, s^{-2\nu/d}\,,
\end{equation}
where $F^*$ is a reference
solution computed with $h^*\ll h$ and $s^* \gg s$
 (see \cite[Sect.~2.4]{gknsss:2012} for a justification).
In Figure~\ref{Fig:bias}
\begin{figure}[t]
\centering
\includegraphics[width=0.495\textwidth]{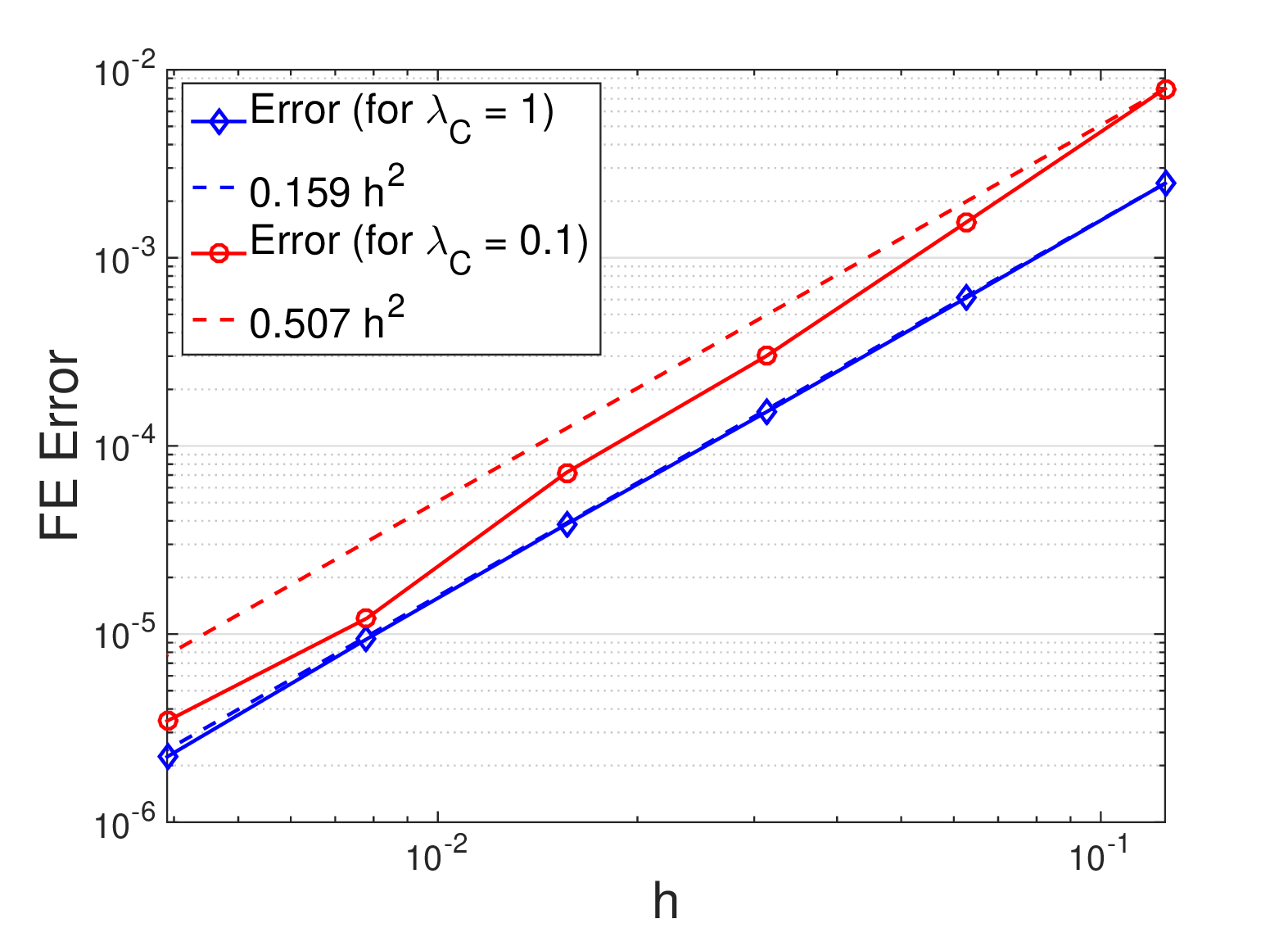}
\includegraphics[width=0.495\textwidth]{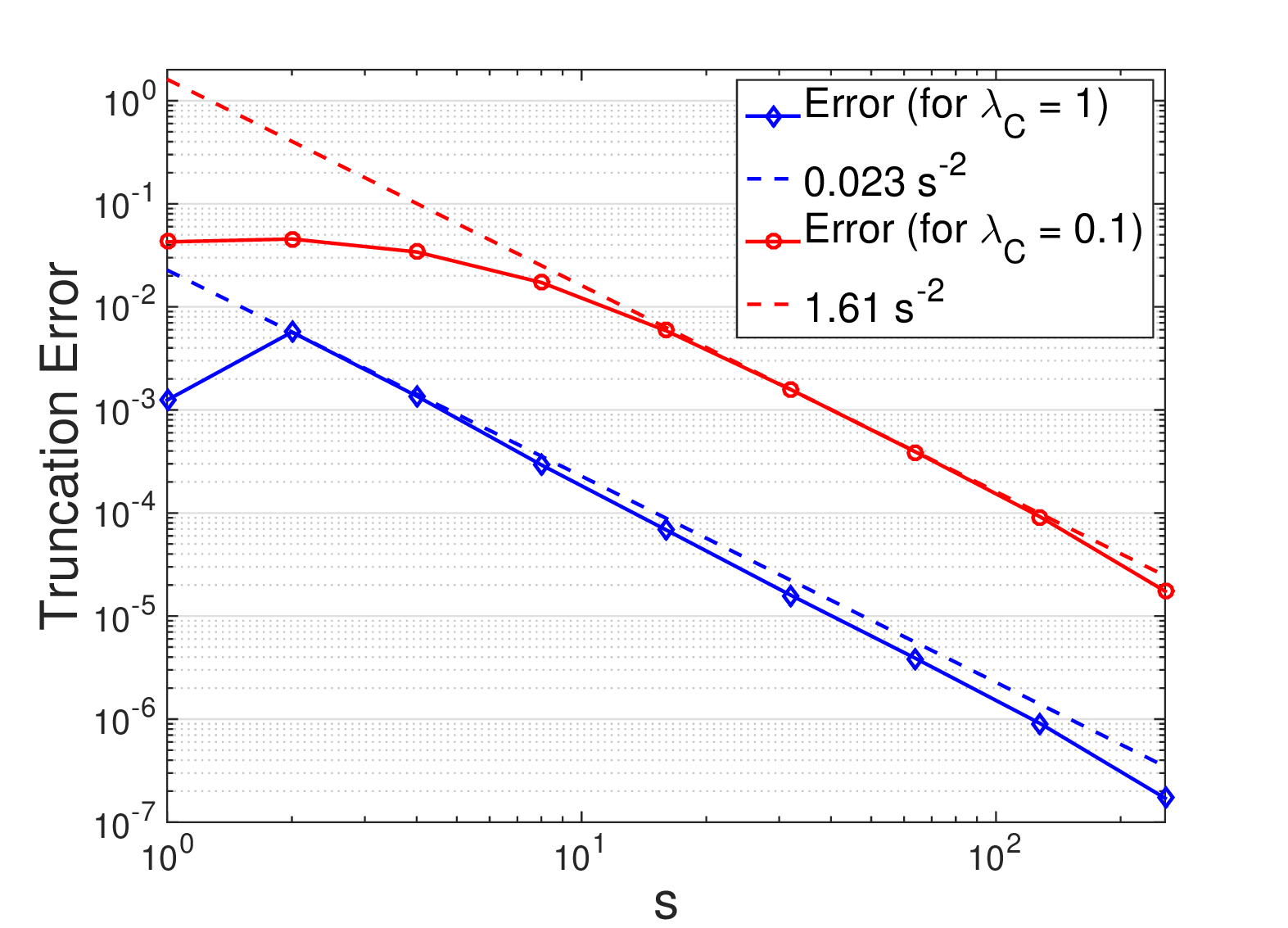}
\caption{\label{Fig:bias} Estimates of the FE bias error $|\mathbb{E}[F_{h,s^*} -
  F^*]|$ (solid lines, left) and of the truncation error $|\mathbb{E}[F_{h^*,s} -
  F^*]|$ (solid lines, right), as well as the bounds in
\eqref{upper_bound} for each case (dashed lines), for $\nu=1$ and
  $\sigma^2=1$.}
\end{figure}
 we plot estimates of $|\mathbb{E}[F_{h,s^*} -  F^*]|$
and of $|\mathbb{E}[F_{h^*,s} - F^*]|$, for the case of $d=1$, $\nu=1$,
$\sigma^2=1$ and for two different values of
$\lambda_\mathcal{C}$. We also show bounds over the plotted range of
$h$ and $s$, for each of the two terms
in \eqref{upper_bound} with the smallest possible values of
$C_{\text{FE}}$ and of $C_{\text{trunc}}$.
The expectations of these constants were estimated with $10^5$ MC samples and with $h^*=1/1024$ and
$s^*=500$. We see that the rates of $\alpha = 2 = \alpha'$
(for $\nu=1$) in \eqref{upper_bound} are sharp.

In our experiments, we then choose a particular sequence
$\eps_L := 2\sqrt{2}\,C_{\text{FE}} h_L^2$, where $L \ge 1$ and
$C_{\text{FE}}$ is the constant in \eqref{upper_bound} which we
estimate as shown above for each problem. We choose a corresponding
truncation dimension $s_L$ such that $C_\text{trunc} s_L^{-2\nu/d} \le
C_{\text{FE}} h_L^2$, which implies
\begin{equation}
\label{eq:sL}
s_L  \, := \, \left\lceil C_{\text{bal}} \, h_L^{-d/\nu} \right\rceil\, \quad \text{with} \ \ C_{\text{bal}} \, :=\, (C_\text{trunc}/C_{\text{FE}})^{d/(2\nu)} \,,
\end{equation}
and ensures that the total bound on the bias error in
\eqref{upper_bound} is less than $\eps_L/\sqrt{2}$. We then run each of the
estimators until the variance error is less than $\eps_L^2/2$, thus
ensuring a MSE (as defined in \eqref{eq:bias}) of less than $\eps_L^2$ and a relative standard error (as defined in \eqref{def:relstderr}) of less than $\eps_L/|\mathbb{E}[F]|$.

The numbers $N_\ell$ of lattice points for the MLQMC estimator on
each of the levels are chosen adaptively using the algorithm by Giles
and Waterhouse \cite{GilesWaterhouse:2009}, given in Algorithm~\ref{alg1} in
Section~\ref{sec:practical}. To estimate
the variance $\cV_\ell := \mathbb{V}_\Delta(\cQ_\ell(F_\ell -
F_{\ell-1}))$ on each level, we use \eqref{varestimate}. As in
  Corollary~\ref{cor:mlqmc}, we choose $s_\ell = s_L$ on all coarser levels. For the cost
on level $\ell$, we assume
\begin{equation}
\label{cost:QMC}
\cC_\ell := \mathrm{cost}(\cQ_\ell(F_\ell - F_{\ell-1})) \approx  (2s_L +
13) \, h_\ell^{-1} \, N_\ell \, R \, .
\end{equation}
This estimate is based on the fact (i) that the evaluation at the mid
points of the mesh intervals of the coefficient in \eqref{kle}, with $a_* \equiv
0$, $a_0 \equiv 1$ and with the sum truncated after $s_L$ terms,
requires about $\cC_\ell^{\text{perm}} = (2s_L + 1)\, h_\ell^{-1}$ operations; and (ii) that
there are direct solvers for diagonally dominant tridiagonal systems (e.g., the Thomas algorithm)
that achieve a complexity of $8$ operations per unknown, leading to the cost estimate
$\cC_\ell^{\text{solve}} = 8 (h_\ell^{-1} +
h_{\ell-1}^{-1}) = 12 h_\ell^{-1}$.

For the standard MLMC estimator we choose the same mesh and truncation parameters, $h_\ell$ and $s_\ell$, as for our new MLQMC estimator. The optimal numbers of
samples $N^{\text{MC}}_\ell$ are chosen according to the formula in the
original paper \cite{Giles:2008}. This requires variance estimates for the
differences $F_\ell - F_{\ell-1}$ on each of the levels, which are obtained
via the usual sample variance estimate with $10^2$ initial
samples, updating the estimates as $N^{\text{MC}}_\ell \to \infty$ on each level. The one-level variants are defined accordingly.

In Figures~\ref{fig1}--\ref{fig2},  we plot the cost to achieve a relative
standard error less than $\eps$ with MLQMC and  MLMC, as well as with the
one-level variants QMC and MC, for $\sigma^2=1$, $\nu=1, 2$, and
$\lambda_\cC = 1, 0.1$. Red lines with circles correspond to the MC-based
variants, while blue lines with diamonds correspond to the QMC-based
estimators. The points on each graph correspond to the choices $\ell_0 =
3$ and $L=1,\ldots,4$. The values of $s_L$ are chosen according to
\eqref{eq:sL} in each case. The exception is the hardest test case ($\nu=1$,
$\lambda_\cC=0.1$), where we used $L=2,\ldots,5$ and a variable
number $s_\ell = \lceil C_{\text{bal}} h_\ell^{-1} \rceil$ of KL terms
on level $\ell$ in MLQMC and in MLMC. The maximum number of
KL terms included in that case is $s_5 = 456$. In all test cases, we
consistently see substantial gains for the MLQMC estimator, with respect
to MLMC and QMC, even though the value of $\nu$ is substantially smaller
than our theory supports (see Remark~\ref{rem:support} ahead).
\begin{figure}[t]
\centering
\includegraphics[width=0.495\textwidth]{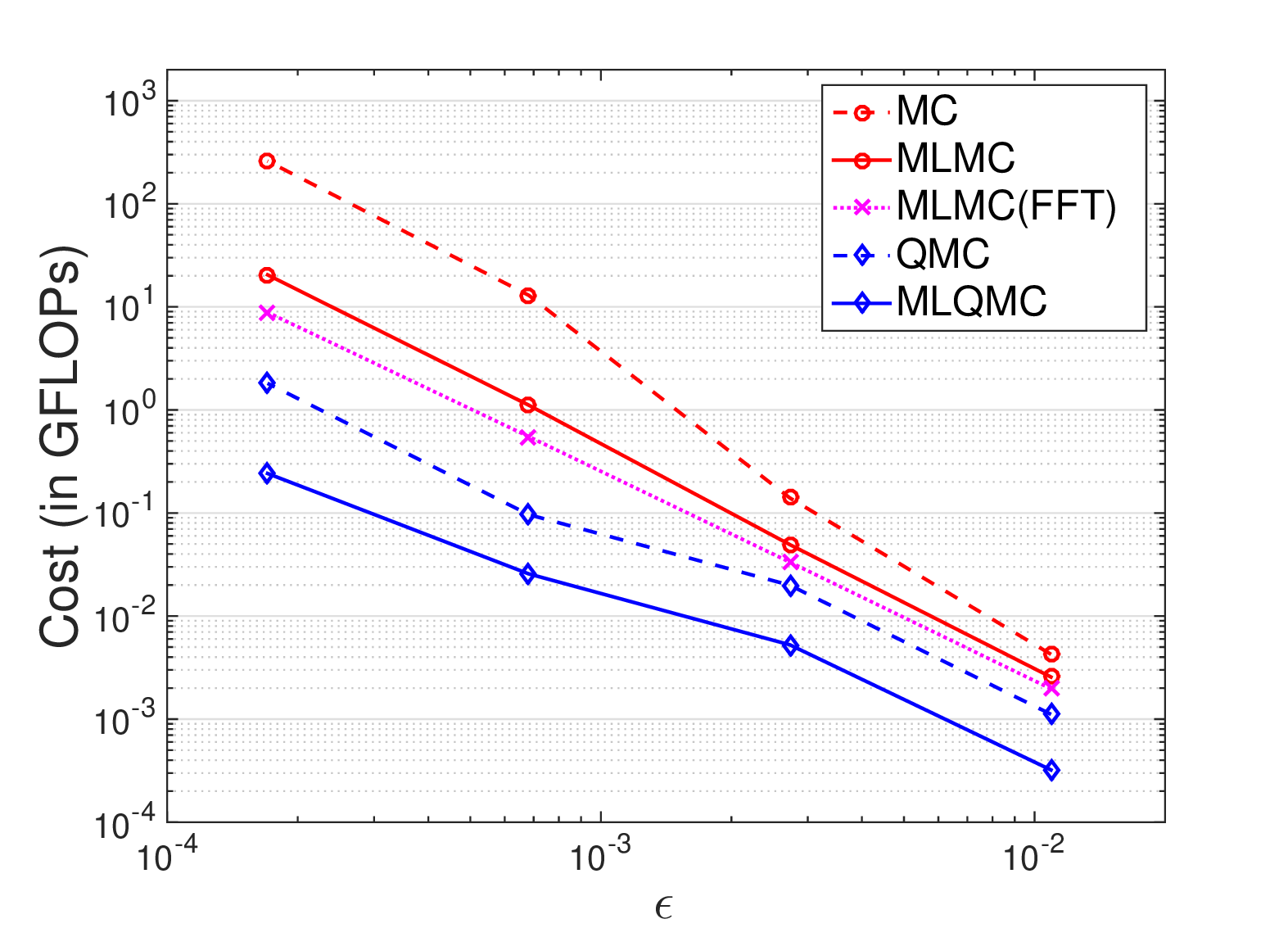}
\includegraphics[width=0.495\textwidth]{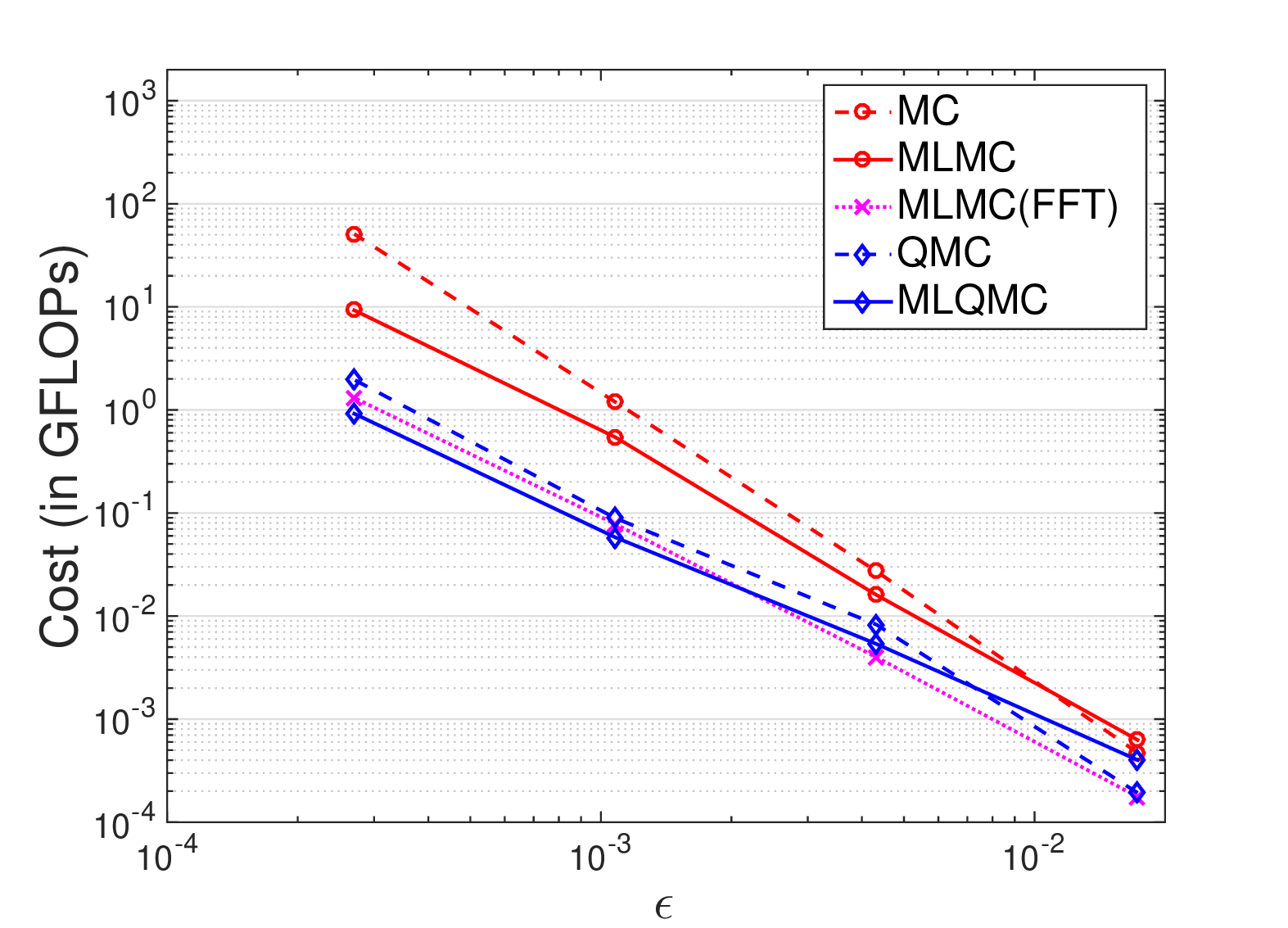}
\caption{Cost to obtain a relative standard error less than $\eps$ in the 1D example. The covariance
  parameters are $\nu=2$ and $\sigma^2=1$, as well as $\lambda_\cC=1.0$ (left) and
  $\lambda_\cC=0.1$ (right), respectively.
The estimates for $C_{\text{bal}}$ are 0.76 (left) and 2.38 (right), respectively.}
\label{fig1}
\end{figure}
\begin{figure}[t]
\centering
\includegraphics[width=0.495\textwidth]{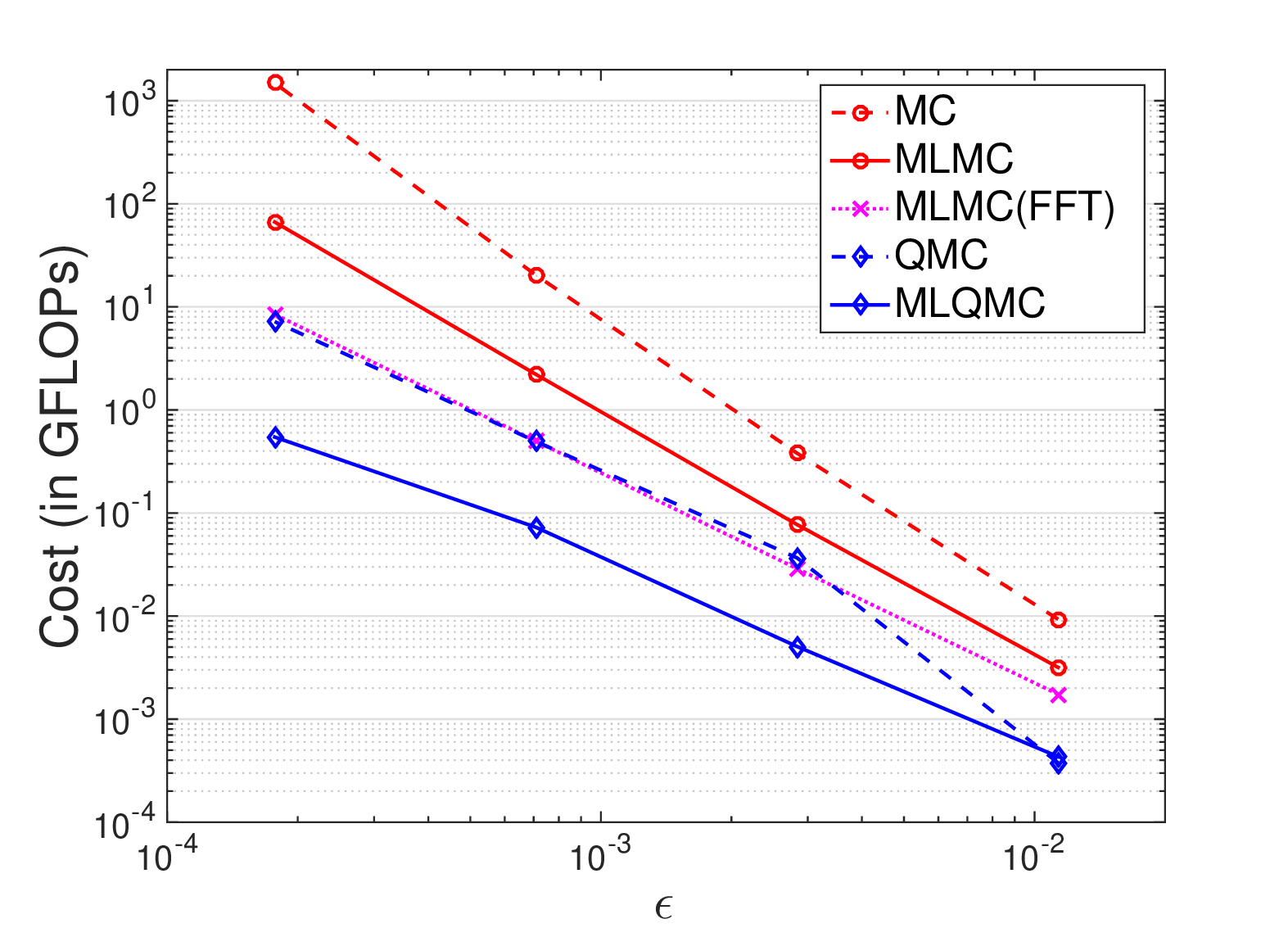}
\includegraphics[width=0.495\textwidth]{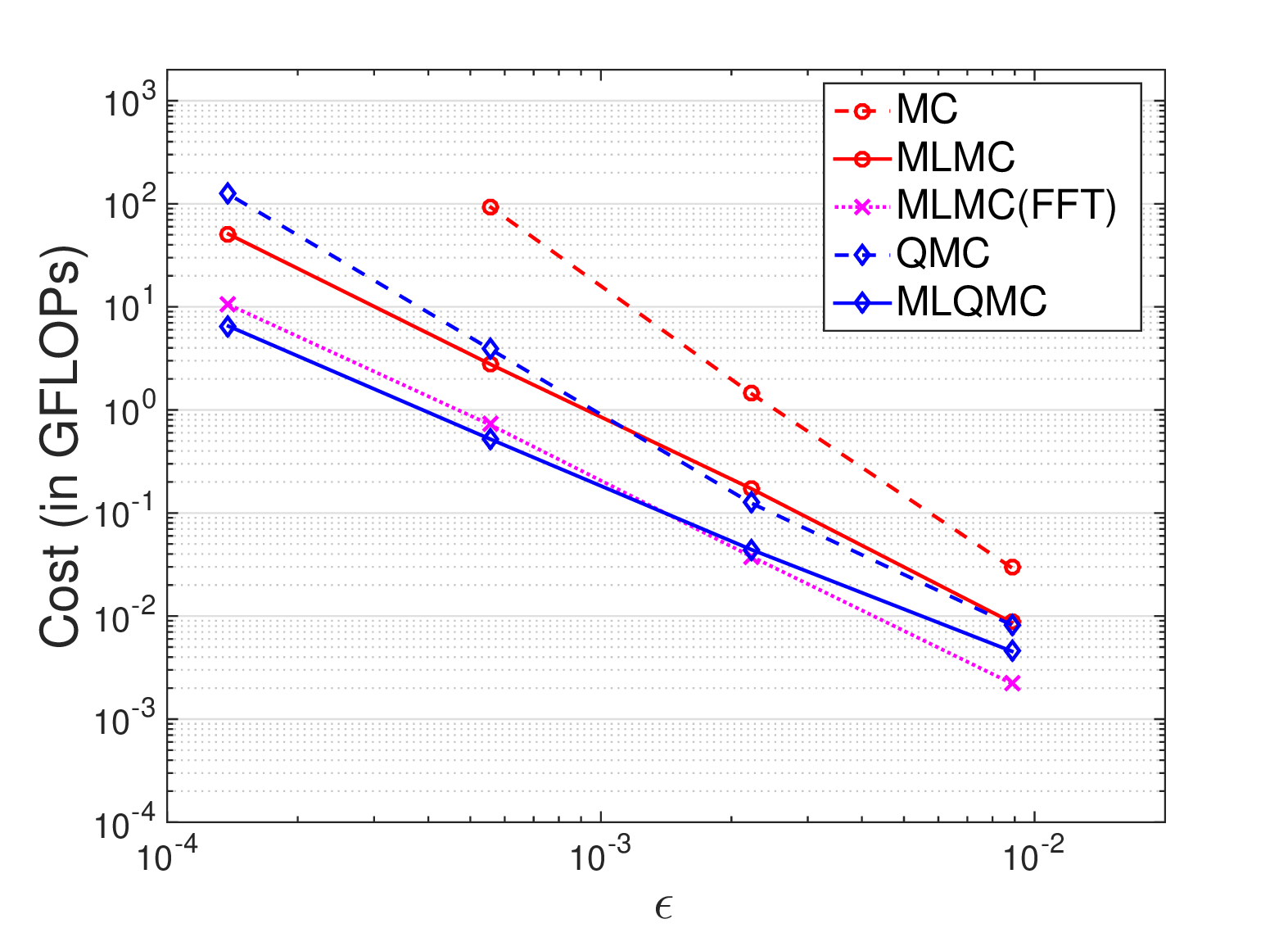}
\caption{Cost to obtain a relative standard error less than $\eps$ in the 1D example.
The covariance parameters are $\nu=1$ and $\sigma^2=1$, as well as
$\lambda_\cC=1.0$ (left) and $\lambda_\cC=0.1$ (right), respectively.
The estimates for $C_{\text{bal}}$ are 0.38 (left) and 1.78 (right), respectively.}
\label{fig2}
\end{figure}

For comparison, we show in
Figures~\ref{fig1}--\ref{fig2} also cost estimates
for MLMC using circulant embedding which makes use of the Fast
Fourier Transform (FFT)
(magenta line, labelled `MLMC(FFT)'). Circulant
embedding allows for efficient sampling at the quadrature points
from isotropic random fields, such as
the one studied here, without any truncation error (see
e.g.~\cite{Graham_etal:2011}) and with a cost independent of $s_L$. We assume
that for the MLMC estimator with circulant embedding, the cost on
level~$\ell$ is
\begin{equation}
\label{cost:MC}
\cC^{\text{FFT}}_\ell \approx \big(5(\ell+\ell_0) + 2\big) \, h_\ell^{-1} \,
N^{\text{MC}}_\ell \, < \, (68/9 (\ell+\ell_0) - 248/27 + 12) \, (\sqrt{2} h_\ell)^{-1} \,
N^{\text{MC}}_\ell\,.
\end{equation}
The factor $\sqrt{2}$ in front of $h_\ell$ appears because there is no
truncation error and thus the FE bias error can be increased by a
factor 2 to still achieve a MSE of $\eps_L^2$ for the MLMC estimator.
For the sampling of the coefficient we then assume the use of
circulant embedding without padding~\cite{Graham_etal:2011} -- which doubles
the number of unknowns in 1D -- and a split-radix FFT algorithm
that requires
$\frac{34}{9} \,n \log_2(n) - \frac{124}{27}\,n + \mathcal{O}\big(
\log_2(n) \big)$ operations for vectors of length $n$~\cite{JohnsonFrigo:2007}.
This is almost certainly underestimating the cost for circulant embedding,
but, as we can see in Figures~\ref{fig1}--\ref{fig2}, the cost is still higher than that of our MLQMC estimator asymptotically.

\begin{figure}[t]
\centering
\psfrag{l0}{\scriptsize ${\scriptstyle \bigcirc}$ $\ell = 0$}
\psfrag{l1}{\scriptsize $\Diamond$ $\ell = 1$}
\psfrag{l2}{\scriptsize ${\scriptstyle \nabla}$ $\ell = 2$}
\psfrag{l3}{\scriptsize ${\scriptstyle \triangle}$ $\ell = 3$}
\psfrag{l4}{\scriptsize $\Box$ $\ell = 4$}
\includegraphics[width=0.495\textwidth]{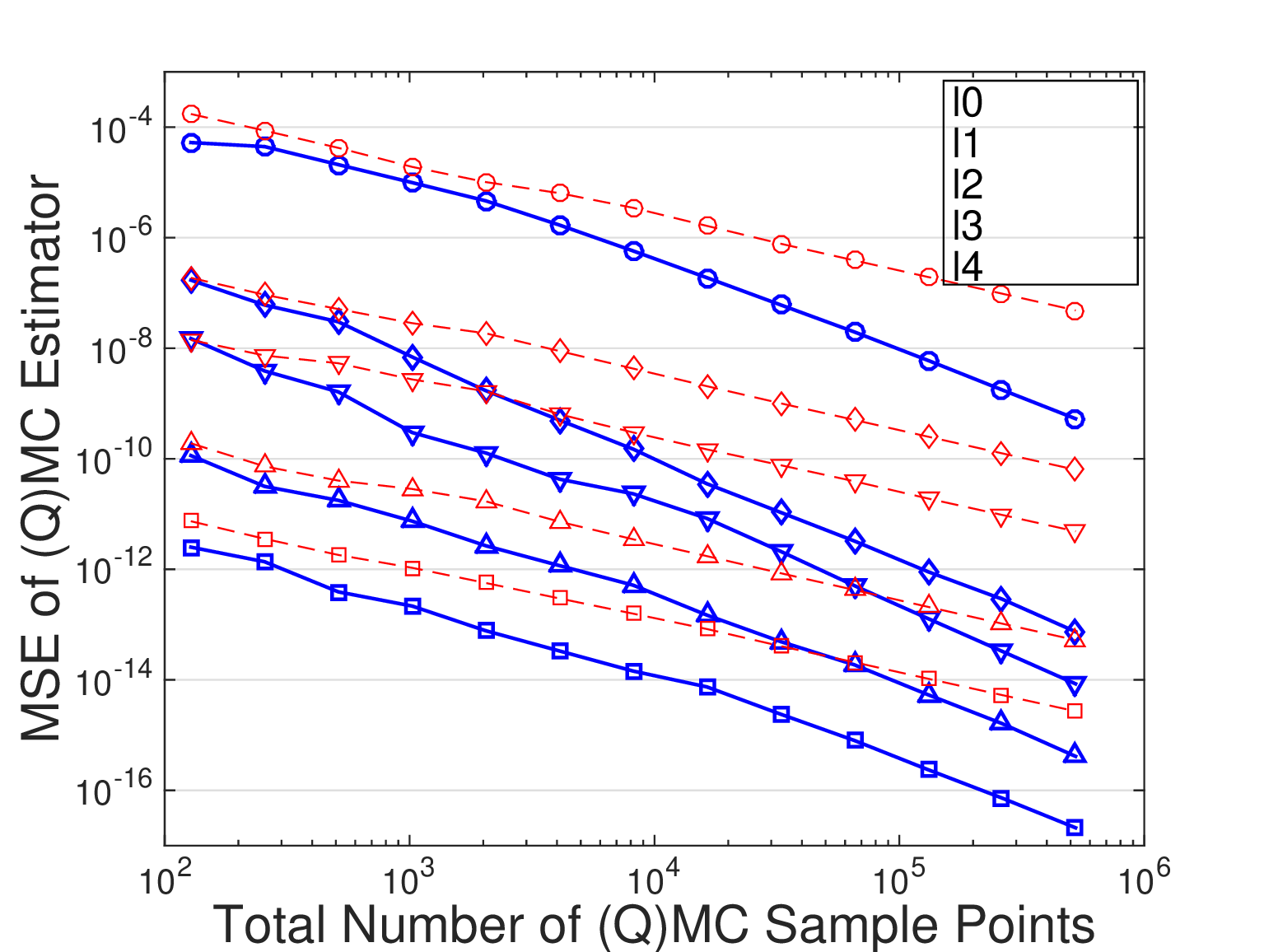}
\includegraphics[width=0.495\textwidth]{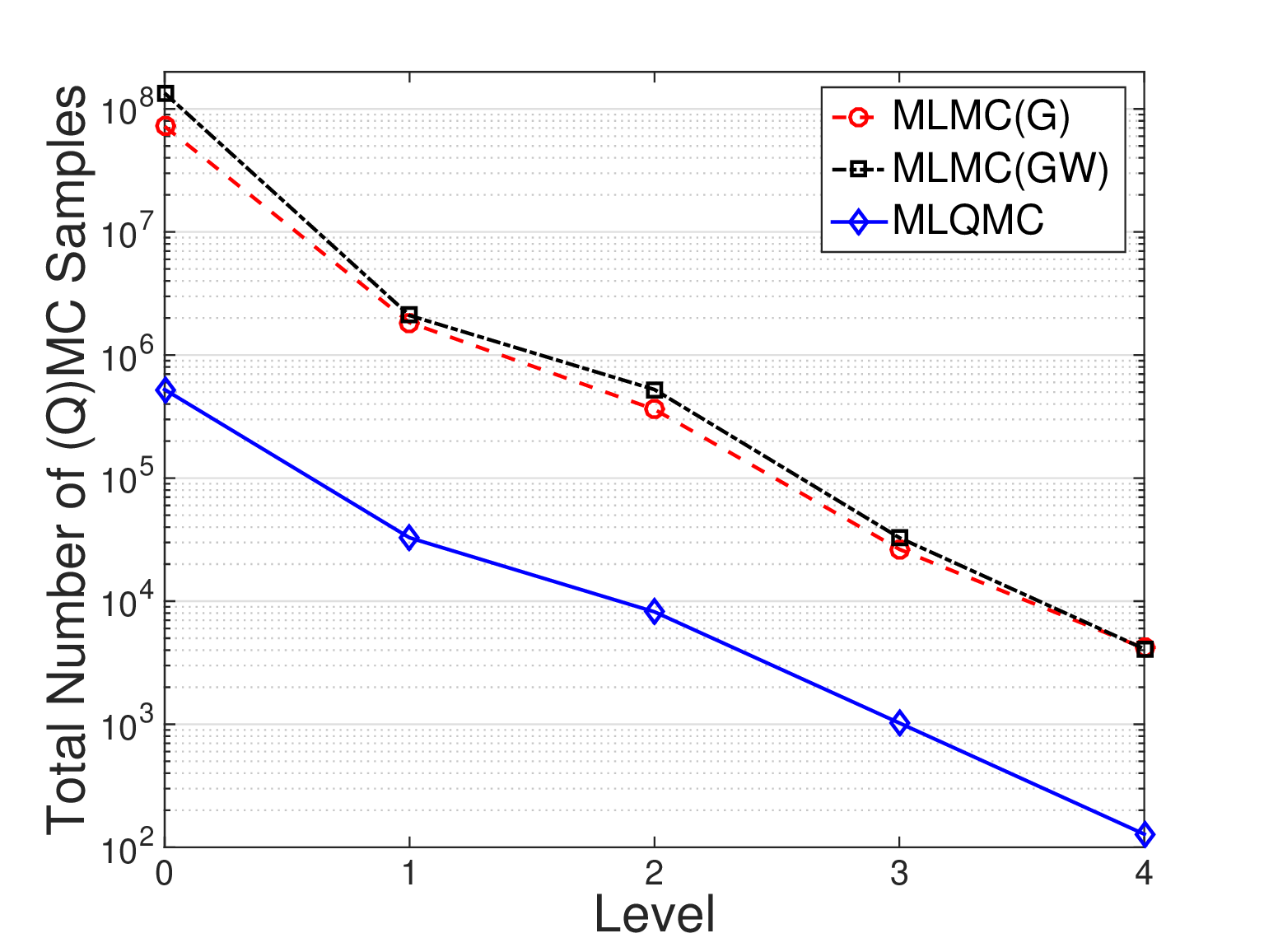}
\caption{Left: MSE of the QMC/MC estimators for $F_\ell -
  F_{\ell - 1}$ as functions of the total number of sample points, i.e.
$R N_\ell$ for QMC (solid blue lines) and $N_\ell^\text{MC}$ for MC (dashed red lines),
for $\nu=\sigma^2=\lambda_\cC=1$. Right: Total number of samples on each level
  to achieve a relative standard error less than $\eps = 1.8 \times 10^{-4}$ for the same example.}
\label{fig3}
\end{figure}
 In Figure~\ref{fig3}, we look at the particular case $\nu=\sigma^2=\lambda_\cC=1$, $h_L = 1/128$
and $s_L = 49$, and plot in the left figure the MSE of the QMC and the MC estimators for the expected values of the differences $F_\ell - F_{\ell-1}$ as the total number of sample points is increased
(i.e., $R N_\ell$ and $N_\ell^\text{MC}$, respectively). We clearly see the faster rate of convergence
with $N_\ell \to \infty$ for the QMC estimators, which is almost optimal
(i.e. the MSE is nearly $\mathcal{O}(N_\ell^{-2})$) even though $\nu=1$ is not
sufficiently big for our theory in Section~\ref{sec:theory} to apply and even though in the construction of the QMC rules we did not use the weights derived there. We also clearly see
the variance reduction from level to level (i.e., the offset between the lines), which does behave as theoretically shown in Section 5 (i.e. roughly like $\mathcal{O}(h_\ell^{4})$).

In Figure~\ref{fig3} (right) we plot for the same example the numbers
of sample points on each of the levels. For MLQMC they were produced
by Algorithm \ref{alg1}, showing $R N_\ell$, i.e. number of lattice points times
number of shifts. For standard MLMC we show two sequences of numbers:
those produced by the formula in the original MLMC paper \cite{Giles:2008},
labelled `MLMC(G)', and those produced by Algorithm \ref{alg1} with
standard MC estimators on each level, labelled `MLMC(GW)'.
We note that there are only very small differences in these final two
sequences, confirming our discussion in Section \ref{sec:practical}
that Algorithm~\ref{alg1} proposed in
\cite{GilesWaterhouse:2009} can be used instead of the original
algorithm to find the optimal sample distributions over the levels.
The behaviour is the same for all other parameter values.

\subsection{Results in space dimension two}\label{exa:2D}
We consider the problem \eqref{log-diff-primal}, \eqref{kle} with
  Mat\'ern covariance $\rho_\nu$ in \eqref{matern} on $D=(0,1)^2\subset
\mathbb{R}^2$. At first we use again homogeneous
Dirichlet conditions, i.e. $\Gamma = \Gamma_\cD$ and $u(\cdot,\omega)|_\Gamma \equiv
0$, and the source term $f \equiv 1$. The output quantity of interest is the average of the solution $u$
over the region $D^* = \big(\frac{3}{4},\frac{7}{8}\big)\times\big(\frac{7}{8},1\big)$, i.e.,
\[
F(\omega) := \frac{1}{|D^*|} \, \int_{D^*} u(\vx,\omega) \,
\text{d}\vx \, .
\]

We discretise the associated
variational formulation (spatially) using standard piecewise
linear, continuous FEs on a sequence of triangular meshes obtained by
taking a tensor product of each of the meshes in Section \ref{exa:1D} with
itself and by subdividing each of the squares of the resulting mesh
into two triangles, thus leading to $2^{2(\ell+\ell_0)+1}$ triangular elements of
size $h_\ell := h_0 2^{-\ell}$ with $h_0 := 2^{-\ell_0+1/2}$ and $M_\ell =
(2^{\ell+\ell_0}-1)^2$ degrees of freedom on level $\ell=0,\dots,L$.

The finite element bias error and the truncation error are estimated
as in 1D. The choice of domain and functional
guarantee that $u(\cdot,\omega) \in H^2(D)$ (almost surely) and the FE
and truncation errors converge as stated in \eqref{upper_bound}, for
$\nu >1$.
Then, the number of KL terms $s_L$ is again chosen according to
\eqref{eq:sL} and $s_\ell = s_L$ on the coarser levels of the
  multilevel methods in all cases. For the
average cost to compute one sample on each level, we use actual
CPU-timings here (instead of FLOP counts). These were obtained
using {\tt FreeFEM++} \cite{freefem} and the sparse direct solver {\tt
  UMFPACK} \cite{suitesparse}.
The measured times to evaluate the KL expansion (with $s$ terms) at
the quadrature points ($\mathcal{C}_\ell^{\text{perm}}$) and to assemble and solve the
sparse linear equation system ($\mathcal{C}_\ell^{\text{solve}}$) are
shown in Figure \ref{Fig:cost} (left) together with the total time to
compute one sample, for the case $\ell_0 = 3$, $\ell =4$ and $s = 50$.
Finite element methods for \eqref{log-diff-primal} in two space
dimensions allow, in the practical range of $M_\ell$ considered here,
for superior performance of sparse direct solvers as compared to, e.g.,
multigrid methods. Since we do not exploit the
uniform grid structure in {\tt FreeFEM++} the cost  in Figure
\ref{Fig:cost} (left) is actually
dominated by the FE system assembly,
which scales like $\mathcal{O}(h_\ell^{-2})$. We also note
that $\mathcal{C}_\ell^{\text{perm}} \ll
\mathcal{C}_\ell^{\text{solve}}$ for all our choices of $s_L$ below.
\begin{figure}[t]
\centering
\psfrag{l0}{\scriptsize ${\scriptstyle \bigcirc}$ $\ell = 0$}
\psfrag{l1}{\scriptsize $\Diamond$ $\ell = 1$}
\psfrag{l2}{\scriptsize ${\scriptstyle \nabla}$ $\ell = 2$}
\psfrag{l3}{\scriptsize ${\scriptstyle \triangle}$ $\ell = 3$}
\includegraphics[width=0.495\textwidth]{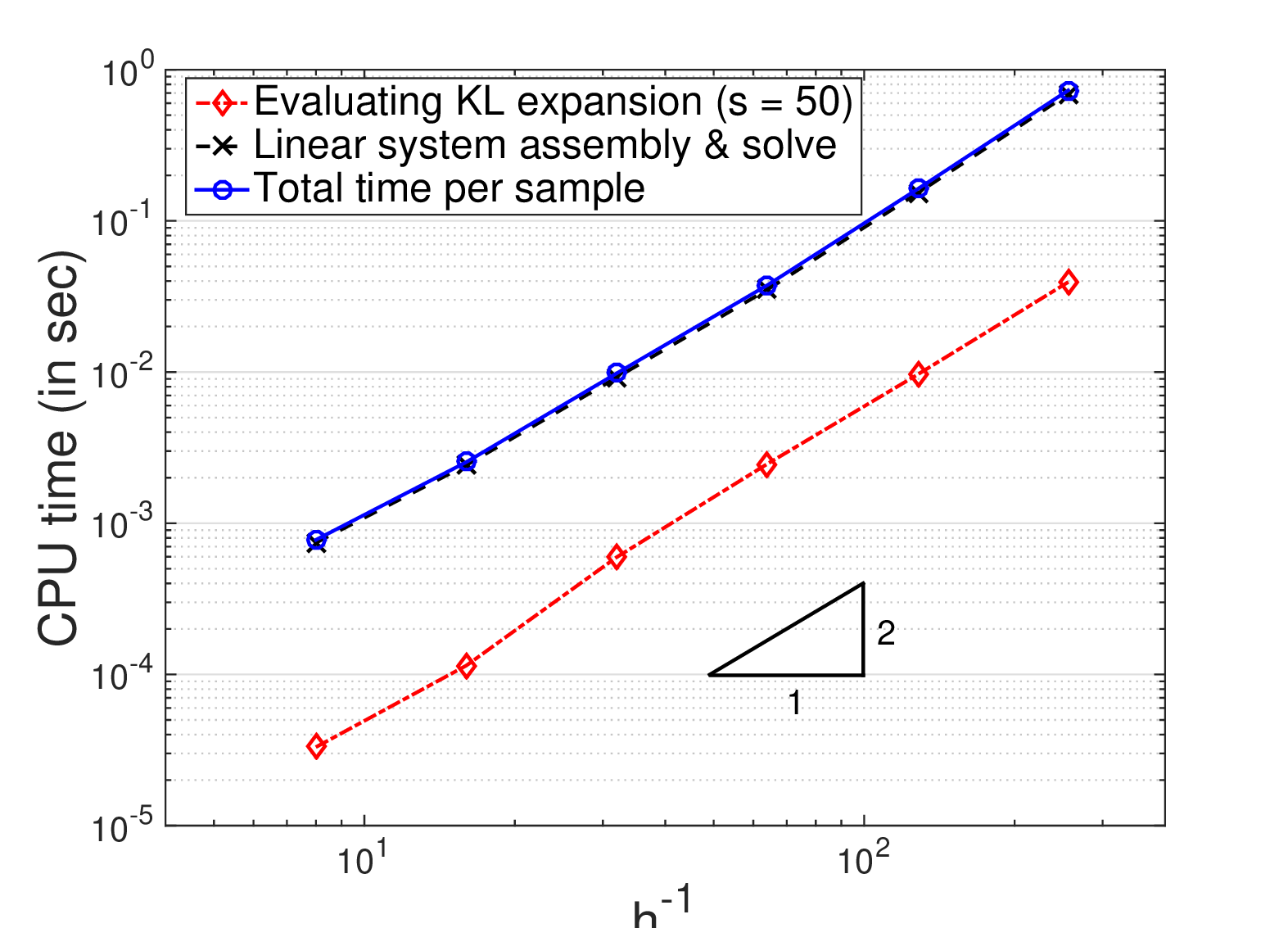}
\includegraphics[width=0.495\textwidth]{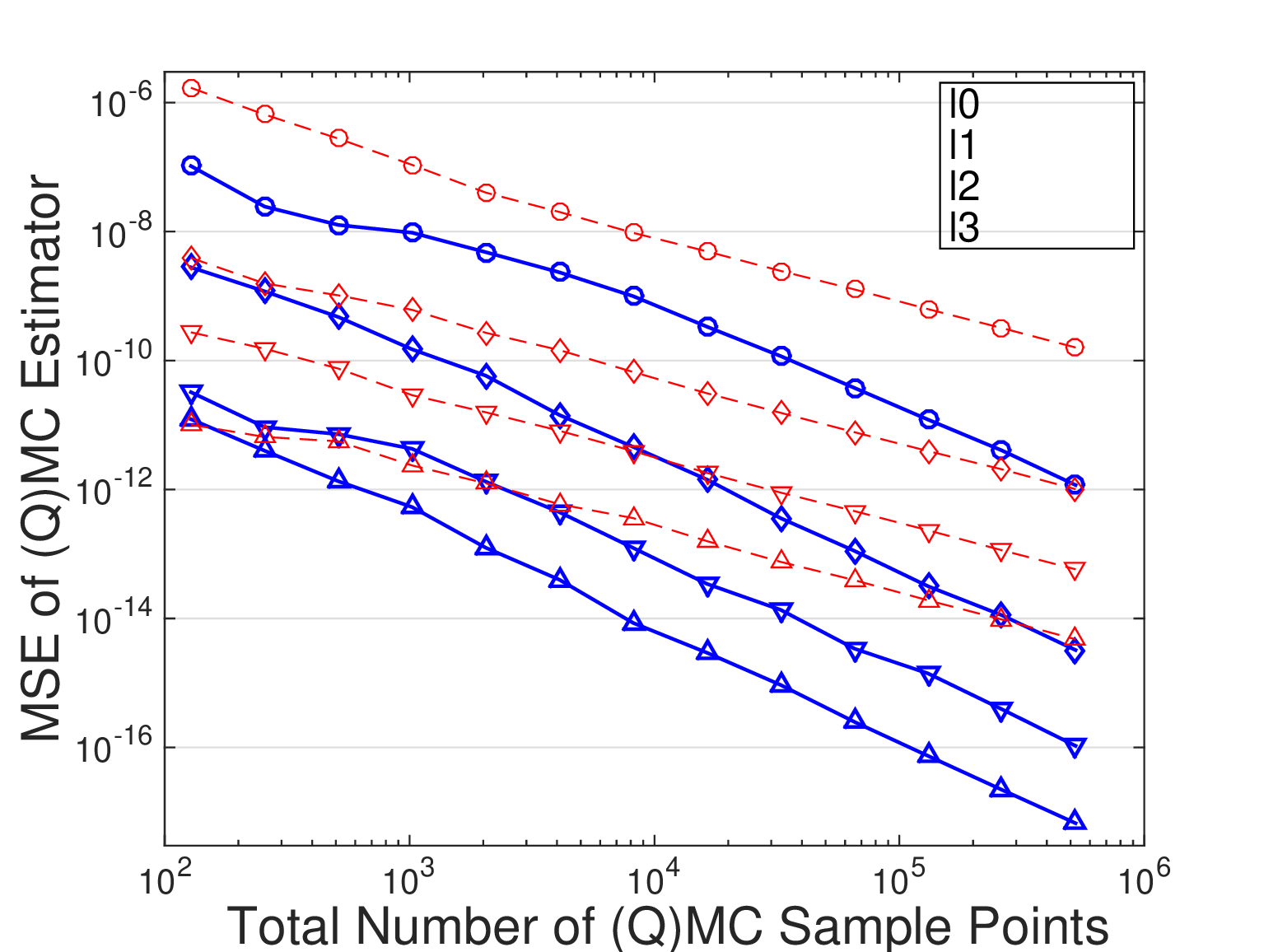}
\caption{\label{Fig:cost} Left: Measured CPU times to calculate
    one sample of $F_\ell$ on level $\ell$ for the 2D problem with
    $\ell_0 = 3$ and $s = 50$.
Right: MSE of the QMC/MC estimators for $\mathbb{E}[F_\ell-F_{\ell-1}]$,
as functions of the number of sample points, i.e. $R N_\ell$ for QMC (solid blue lines) and $N^{\text{MC}}_\ell$ for MC
(dashed red lines), for $\nu=1.5$, $\sigma^2=1$, $\lambda_\cC = 1$,
$\ell_0 = 3$, $L=3$ and $s_L = 27$.}
\end{figure}

In Figure~\ref{Fig:cost} (right), we plot the MSE of the QMC and of
the MC estimators for $\mathbb{E}[F_\ell - F_{\ell-1}]$ as a function of the total number
of sample points for the covariance parameters $\nu=1.5$,
$\sigma^2=1$, $\lambda_\cC = 1$, and for $\ell_0 = 3$, $L=3$ and $s_L = 27$. Again, we see the significantly faster and almost optimal convergence rate for the
QMC estimators as $N_\ell \to \infty$.

\begin{figure}[t]
\centering
\includegraphics[width=0.495\textwidth]{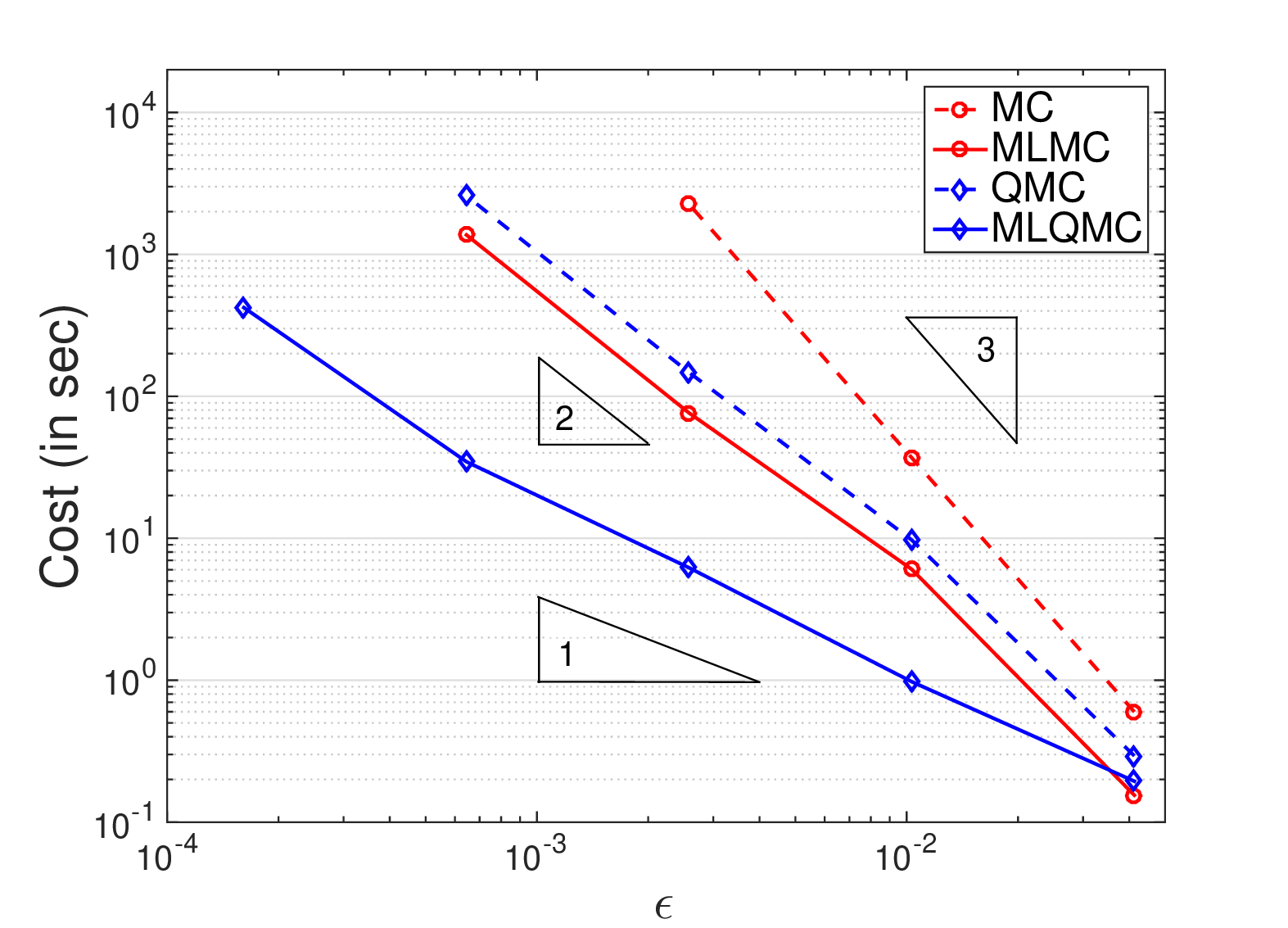}
\includegraphics[width=0.495\textwidth]{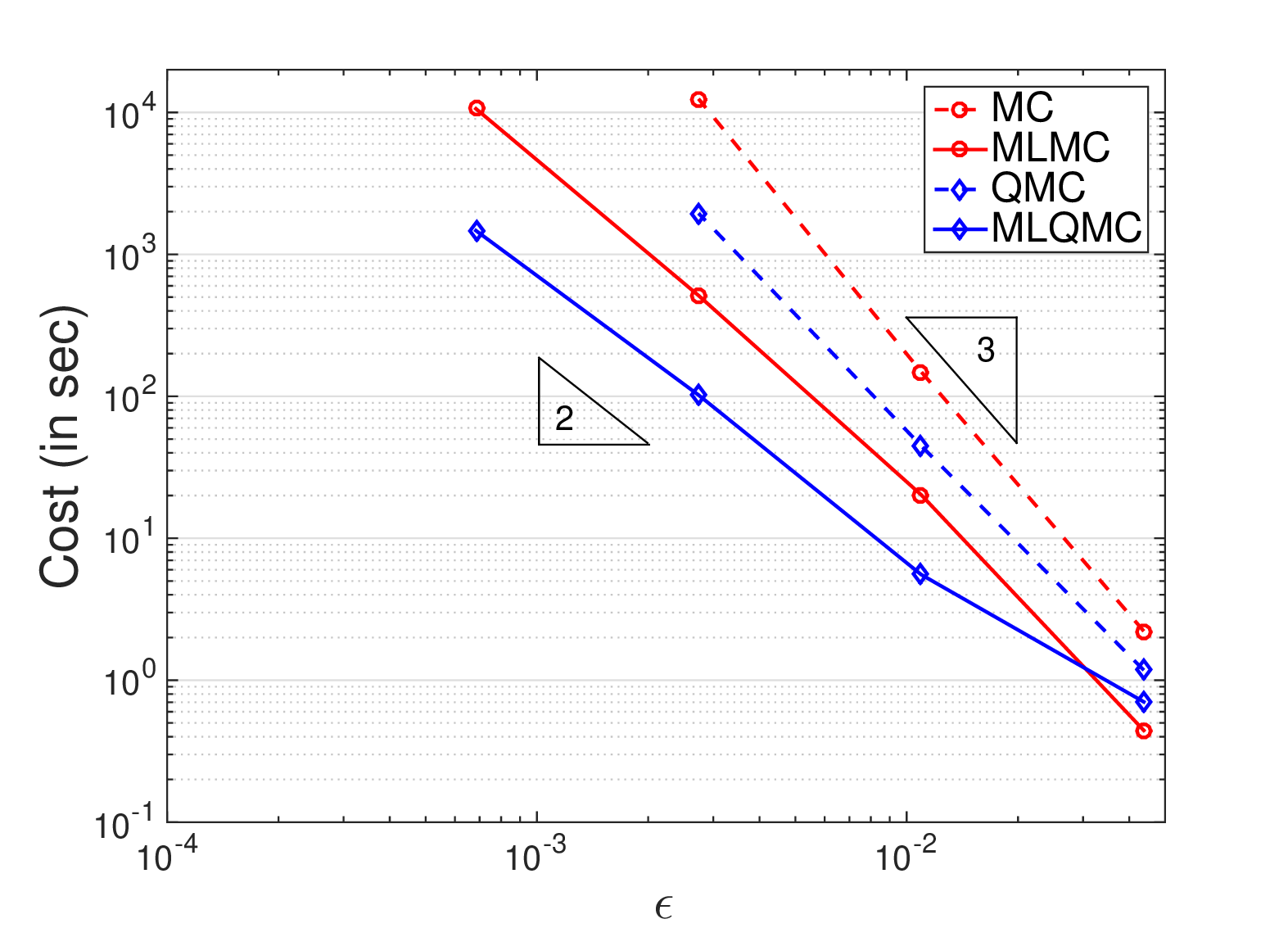}
\caption{\label{fig4}
 Cost to obtain a relative standard error less than $\eps$ in the
2D example with homogeneous
 Dirichlet conditions, for $\nu=2.5$, $\sigma^2=0.25$,
   $\lambda_\cC=1$ and $\ell_0 =3$ (left), as well as for $\nu=1.5$,
   $\sigma^2=1$, $\lambda_\cC=0.1$ and $\ell_0 = 4$  (right).
   The estimates for $C_{\text{bal}}$ are 0.55 (left) and 0.68 (right), leading to a maximum of $s_5=47$ and $s_4=1106$ KL terms on the finest mesh, respectively.}
\end{figure}
In Figure~\ref{fig4}, we plot again the cost to achieve a
relative standard error less than $\eps$ with all four estimators for
two sets of covariance parameters. The points on each of the graphs
correspond to the choices $L=1,\ldots,5$ with $\ell_0=3$ (left) and
$L=1,\ldots,4$ with $\ell_0=4$ (right).  We see similarly impressive gains
with respect to MLMC and QMC in two dimensions, but we also see more
clearly the influence of the smoothness parameter $\nu$.
For the test case in the left figure, the numerically observed growth of
the MLQMC cost is about $\mathcal{O}(\eps^{-1.25})$ over the range $L=1$
to $4$. For comparison,
the costs for MLMC and QMC both show growths of
$\mathcal{O}(\eps^{-2.2})$ over the same range, while MC shows the expected
$\mathcal{O}(\eps^{-3})$ growth.

As a final example, we consider the practically more interesting
case of a 2D ``flow cell'', that is, we solve the PDE
\eqref{log-diff-primal} in $D=(0,1)^2$ with mixed
Dirichlet--Neumann conditions. The horizontal boundaries are
assumed to be impermeable, that is, $(a \nabla u) \cdot \vec{n} = 0$ for $x_2 = 0$ and $x_2=1$.
Along the vertical boundaries
we specify Dirichlet boundary conditions and set $u \equiv 1$, for $x_1 = 0$, and $u
\equiv 0$, for $x_1=1$. We discretise this problem using the same
sequence of meshes as above. Due to the Neumann conditions on the
horizontal boundaries, the number of degrees of freedom in this problem
is $M_\ell = 2^{2(\ell+\ell_0)}-1$ on level $\ell=0,\dots,L$.

Here, the quantity of interest is the outflow through the right
vertical boundary, i.e.
\[
F(\omega)
:=
-\int_0^1 a(\vx,\omega)  \frac{\partial u(\vx,\omega)}{\partial x_1}
\bigg|_{x_1 = 1} \ \rd x_2\;.
\]
As an approximation of this functional we use
\[
F_{h,s}(\omega)
:= -\int_D a_s(\vx,\omega) \nabla u_{h,s}(\vx,\omega) \cdot \nabla \varphi(\vx)\ \rd\vx,
\]
where $\varphi$
denotes the FE function which is equal to one at all of the
vertices of the right vertical boundary and is equal to zero at all
other vertices (see \cite[section 3.4]{TSGU:2012} for details).

\begin{figure}[t]
\centering
\includegraphics[width=0.495\textwidth]{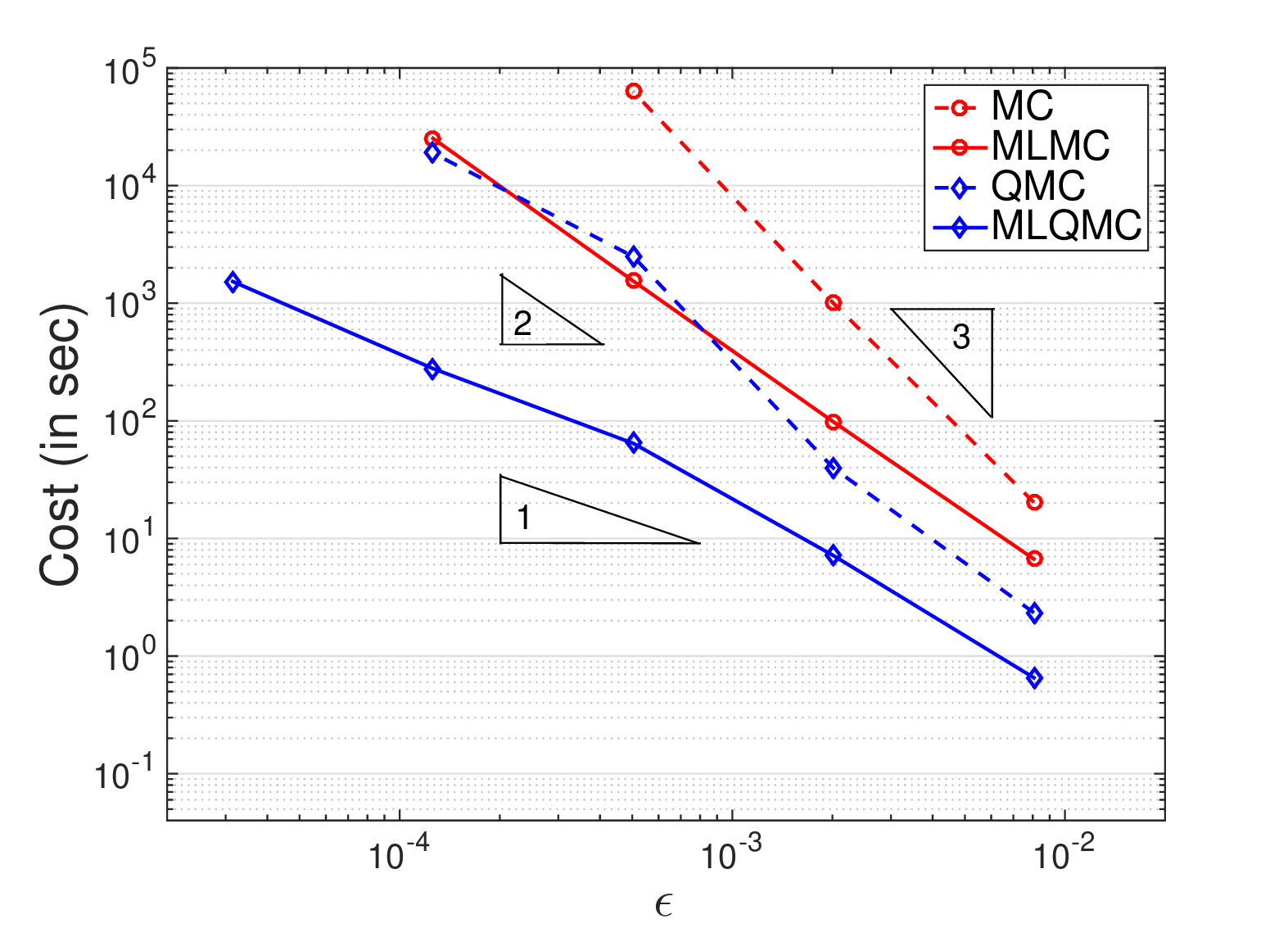}
\includegraphics[width=0.495\textwidth]{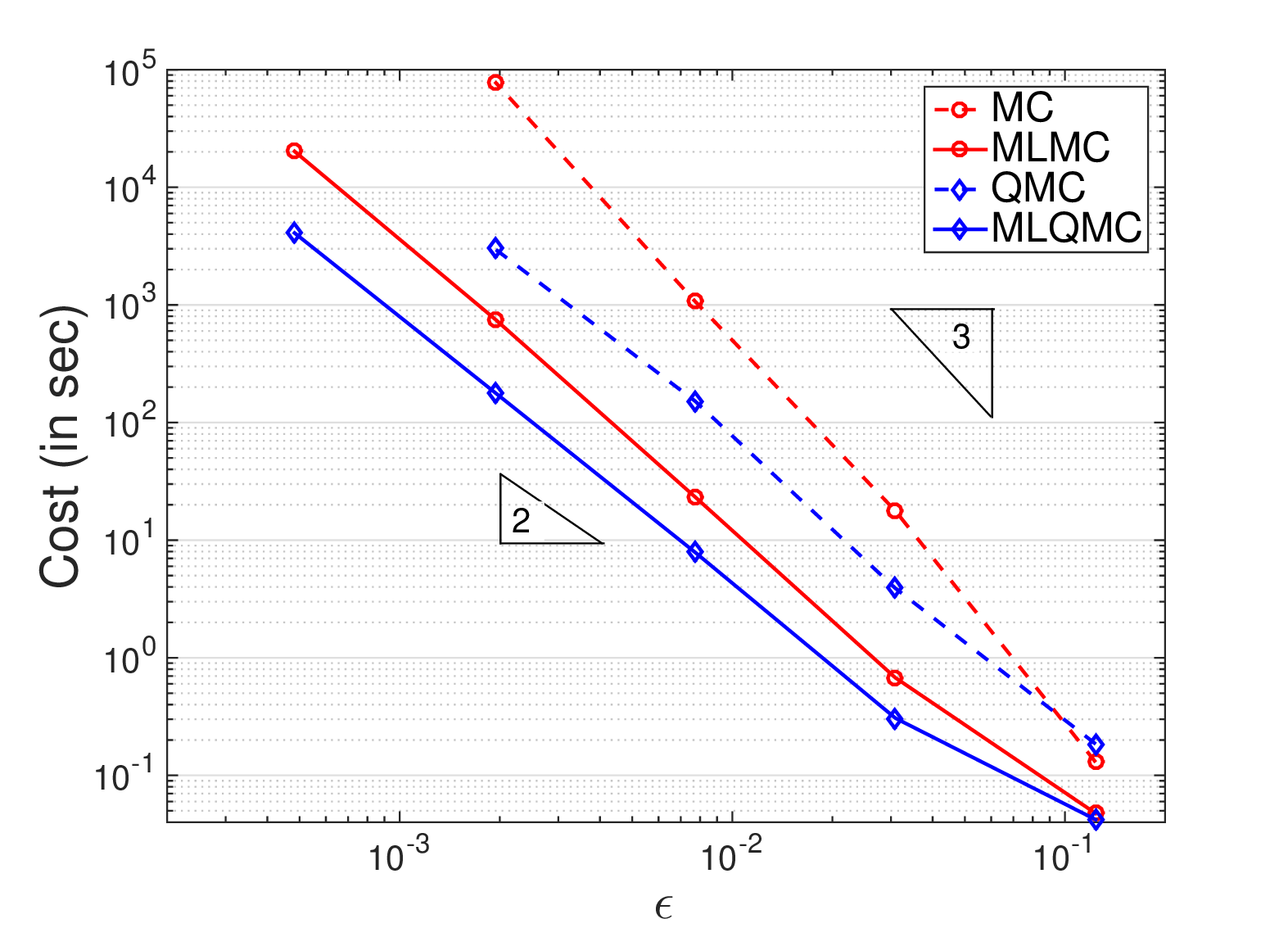}
\caption{\label{fig6}
Cost to obtain a relative standard error less than $\eps$ in the 2D flow cell example. The covariance
  parameters are $\nu=2.5$, $\sigma^2=1$, $\lambda_\cC = 1$ (left) and
$\nu=1$, $\sigma^2=3$, $\lambda_\cC = 0.3$ (right).
The estimates for $C_{\text{bal}}$ are 0.61 (left) and 0.0097 (right),
respectively. The problem on the right is significantly more
  challenging. Please note the different range for $\varepsilon$ in
the two figures.}
\end{figure}
The numerical results for this problem are shown
in Figure \ref{fig6}. In the left figure, we choose $\nu=2.5$,
$\sigma^2=1$, $\lambda_\cC =1$. In the right figure, we choose a set
of parameters closer to the ones used in actual subsurface flow studies, namely
$\nu=1$, $\sigma^2=3$ and $\lambda_\cC=0.3$. In both cases $\ell_0=2$. The points on the graphs correspond to the choices $L = 1,\ldots,5$ (left) and $L = 2,\ldots,5$ (right), respectively.
The gains are again of the same order as above in both cases. In the smoother test case (left), the growth of the MLQMC cost is as low as
$\mathcal{O}(\eps^{-1.15})$ between $L=3$ and $5$.

\section{Mathematical analysis and construction of suitable QMC rules}
\label{sec:theory}
In the remainder of the paper, we present sufficient conditions on the
data and on the FE spaces to verify Assumption M2 in the general MLQMC
convergence result in Theorem~\ref{thm:mlqmc}, as well as constructible
QMC rules that achieve this. We will start in Section~\ref{sec:Prel} by
addressing the spatial regularity and approximation orders for the FE
function $u_{h,s}(\cdot,\bsy)$ in \eqref{EG}, making explicit the
dependence on the parameter $\bsy$ in any constants that appear. Then we
turn to the key estimates required for the MLQMC theory: bounds on the
derivatives of the FE error with respect to the stochastic variables in
certain weighted function spaces $\calW_s$ which appear in the QMC
convergence theory (see \cite{gknsss:2012} and the references there)
with constants that are independent of the truncation dimension $s$.
These bounds correspond to ``mixed derivative bounds'', appearing also in
hyperbolic cross and other high-dimensional approximation methods
\cite{BG:2004,SG11_518,Harbrecht:2013}, in that they require joint
regularity of the random solution $u(\vx,\omega)$ with respect to the
spatial as well as with respect to the stochastic argument. These
estimates are proved in Section~\ref{sec:KeyRes}, and are used in
Section~\ref{sec:QMCConv} to establish the MLQMC convergence rate
estimates.

\subsection{Parametric formulation, spatial regularity and FE approximation}
\label{sec:Prel}

As in \cite{gknsss:2012}, we
assume that, for $d=2,\,3$, $D$ is a bounded,
Lipschitz polygonal/polyhedral domain.
For simplicity, we restrict ourselves to homogeneous Dirichlet data
$\phi_\cD = 0$ and to deterministic
Neumann data $\phi_\cN \in H^{1/2}(\Gamma_\cN)$
in \eqref{log-diff-primal}.
Then, the stochastic PDE \eqref{log-diff-primal}
is (upon a measure-zero modification of the
lognormal random field $a$ in \eqref{log-diff-primal}),
equivalent to the infinite-dimensional, parametric, deterministic PDE
\begin{equation} \label{eq:pde}
  - \nabla\cdot (a(\cdot,\bsy)\, \nabla u(\cdot,\bsy)) \,=\, f
 \quad\mbox{in}\quad D,
 \quad u|_{\Gamma_\cD}  \,=\, 0 ,
 \quad \vec{n}\cdot a(\cdot,\bsy) \nabla u(\cdot,\bsy) |_{\Gamma_\cN} \,=\, \phi_\cN
\,,
\end{equation}
with parametric, deterministic coefficient
\begin{equation} \label{eq:axy}
a(\vx,\bsy)
\,=\,
a_*(\vx) + a_0(\vx)\exp\bigg(\sum_{j\ge 1} \sqrt{\mu_j}\,\xi_j(\vx)\,y_j\bigg)\,,
\end{equation}
where $\vx\in D\subset \bbR^d$ and where the parameter sequence
$\bsy = (y_j)_{j\ge1} \in \mathbb{R}^\mathbb{N}$ is
distributed according to the product Gaussian measure
$\overline{\boldsymbol{\mu}}_G = \bigotimes_{j=1}^\infty \mathcal{N}(0,1)$.

If $\bsy$ belongs to the set
\begin{equation}\label{eq:DefUb}
 U_{\bsb} \,:=\,
 \bigg\{ \bsy\in\bbR^\bbN \;:\; \sum_{j\ge1} b_j\,|y_j| <\infty \bigg\} \subset\bbR^\bbN\,,
\end{equation}
where the sequence $\bsb = (b_j)_{j\ge 1}$ is defined by $b_j \,:=\,
\sqrt{\mu_j}\,\|\xi_j\|_{L^\infty(D)}$ and is assumed to be in
  $\ell^2(\bbN)$, then the equivalence of \eqref{eq:pde}--\eqref{eq:axy}
  and \eqref{log-diff-primal}--\eqref{kle}
holds up to
$\overline{\boldsymbol{\mu}}_G$-measure zero modifications of the input
random field. Due to the continuous
dependence on the input random field, the
parametric, deterministic coefficient $a(\cdot,\bsy)$ in \eqref{eq:axy}
and the parametric, deterministic solution $u(\cdot,\bsy)$ of \eqref{eq:pde}
will also differ only on a $\overline{\boldsymbol{\mu}}_G$-nullset.
If, moreover, $\bsb\in \ell^1(\bbN)$, then the series \eqref{eq:axy}
converges in
$L^\infty(D)$ for every $\bsy\in U_\bsb$, which we assume in what follows.

To simplify the presentation, we assume $a_* = 0$ and $|\Gamma_{\cD}|
> 0$. We need the weak form of \eqref{eq:pde} on the Hilbert
space $V = H^1_{\Gamma_\cD}(D) := \{ v \in H^1(D) : v|_{\Gamma_\cD} =
0\}$, with norm
\[
  \|v\|_{V} \,:=\, \|\nabla v\|_{L^2(D)}\,.
\]
As in \cite{gknsss:2012},
we define for $\bsy \in U_\bsb$ the
\emph{parametric, deterministic bilinear form} in $V$ by
\begin{equation}\label{eq:DefParBil}
  \scA(\bsy;w,v) \,:=\, \int_D a(\vx,\bsy)\, \nabla w(\vx)\cdot\nabla v(\vx)\,\rd\vx\,,
  \qquad\mbox{for all}\quad w,v\in V\,.
\end{equation}
We list properties of the parametric
bilinear form $ \scA(\bsy;\cdot,\cdot)$ and of the weak solution
  $u(\cdot,\bsy)$, as well as its FE approximation $u_h(\cdot,\bsy)$,
from \cite{gknsss:2012}. For a proof see~\cite[Thm.~13]{gknsss:2012}.
\begin{proposition}\label{prop:ParmBilPrp}
Assume that $\bsb\in \ell^1(\bbN)$.
\begin{enumerate}
\item[\textnormal{(a)}] The expressions
\begin{equation}\label{eq:hatchka}
\hat{a}(\bsy)   \,:=\, \max_{\vx\in \overline{D}} a(\vx,\bsy)
\quad\mbox{and}\quad
\check{a}(\bsy) \,:=\, \min_{\vx\in \overline{D}} a(\vx,\bsy)
\end{equation}
are well-defined, $\overline{\boldsymbol{\mu}}_G$-measurable mappings
from $U_\bsb$ to $\bbR$ which satisfy
\begin{equation}\label{eq:bounda}
 0< \check{a}(\bsy) \leq \hat{a}(\bsy) < \infty \qquad \text{for all} \quad
 \bsy\in U_\bsb\,.
\end{equation}
\item[\textnormal{(b)}] For every $\bsy\in U_\bsb$, the parametric
    bilinear form $ \scA(\bsy;\cdot,\cdot):V\times V\rightarrow \bbR$
    defined in \eqref{eq:DefParBil} is continuous and coercive in the
    following sense:
\begin{align}\label{eq:ParmBilCoer}
  \scA(\bsy;w,w) &\,\ge\, \check{a}(\bsy) \| w \|_V^2
  &&\text{for all}\quad w\in V\,, \quad \mbox{and} \\[1ex]
  \label{eq:ParmBilCont}
  \scA(\bsy;v,w) &\,\le\, \hat{a}(\bsy) \| v \|_V \| w \|_V
  &&\text{for all} \quad v, w\in V\,.
\end{align}
\item[\textnormal{(c)}] For every $f\in L^2(D)$ and every $\phi_\cN
    \in H^{1/2}(\Gamma_\cN)$ and with the
    \textnormal{(}$\bsy$-independent\textnormal{)} linear functional
$$
\scL(v) \,:=\, \int_D f(\vx)\,v(\vx) \,\rd\vx + \int_{\Gamma_\cN} \phi_\cN(\vx)\, v(\vx)\,\rd s\,,
$$
the parametric, deterministic variational problem, to find $u(\cdot,\bsy)\in V$ such that
\begin{equation} \label{eq:var}
  \scA(\bsy;u(\cdot,\bsy),v) \,=\,
  \scL(v)
  \qquad\mbox{for all}\quad
  v\in V\,,
\end{equation}
admits a unique solution $u(\cdot,\bsy)\in V$, for every $\bsy
  \in U_\bsb$.

\item[\textnormal{(d)}] This parametric solution $U_\bsb \ni \bsy
    \mapsto u(\cdot,\bsy)\in V$ is a strongly measurable mapping
    \textnormal{(}with respect to a suitable $\sigma$-algebra,
    cf.~\cite{gknsss:2012}\textnormal{)}
which satisfies the bound
\begin{equation}\label{eq:aprioriu}
\| u(\cdot,\bsy) \|_V
\,\lesssim\,
\frac{1}{\check{a}(\bsy)}
\left(\| f \|_{L^2(D)} + \| \phi_\cN \|_{H^{1/2}(\Gamma_\cN)}\right)
\qquad \text{for all}\quad \bsy \in U_\bsb\,,
\end{equation}
\textnormal{(}pointwise with respect to $\bsy$\textnormal{)}. The
implied constant depends on $D$, $\Gamma_\cD$ and $\Gamma_\cN$, but is
independent of $\bsy$. In particular, for any $s\in \bbN$, $u(\cdot,
(\bsy_{\{1:s\}};\boldsymbol{0})) \in V$ is well-defined,\footnote{As
    in \cite{gknsss:2012}, for any finite subset $\setu \subset
    \mathbb{N}$, we denote by $(\bsy_{\setu};\boldsymbol{0})$ the
    vector $\bsy \in U_\bsb$ with the constraint that $y_j = 0$ if $j
    \not\in \setu$, and we use the
  shorthand notation $\{1:s\}$ for $\{1,2,\ldots,s\}$.} measurable,
  and satisfies the above bounds uniformly with respect to $s$.

\item[\textnormal{(e)}] Restricting in \eqref{eq:var} to functions in
    the FE space $V_h\subset V$, there exists a unique, parametric FE
    solution $u_h(\cdot,\bsy)\in V_h$, for every $\bsy \in U_\bsb$ and
    $0<h<1$, that also satisfies the bound \eqref{eq:aprioriu}.
    Replacing, in addition, the coefficient $a(\vx,\bsy)$ in
    \eqref{eq:var} by the $s$-term truncated KL expansion
    $a^s(\vx,\bsy)$, the corresponding $s$-parametric FE solutions
    $u_{h,s}(\cdot,(\bsy_{\{1:s\}};\boldsymbol{0}))\in V_h$ are
    uniquely defined for {\em any} $\bsy_{\{1:s\}} \in \bbR^s$, and
    satisfy, for $(\bsy_{\{1:s\}};\boldsymbol{0})\in U_\bsb$, the
    bound \eqref{eq:aprioriu} uniformly with respect to $h$ and to
    $s$.
\end{enumerate}
\end{proposition}
For the derivation of FE convergence rate bounds,
we require additional spatial regularity: we assume in \eqref{eq:axy}
\begin{equation}
\label{ass:regular}
a_* \equiv 0\,, \quad a_0 \in  W^{1,\infty}(D) \quad \text{and} \quad \xi_j \in
W^{1,\infty}(D)\,.
\end{equation}
With \eqref{ass:regular}, we may define the sequence
\begin{equation} \label{eq:barbj}
\overline{b}_j
\,:=\, \sqrt{\mu_j}\, \max(\|\xi_j\|_{L^\infty(D)},\||\nabla\xi_j|\|_{L^\infty(D)})
\,,
\quad j=1,2,...
\,.
\end{equation}
Evidently, $\overline{b}_j \geq b_j$ so that
$U_{\overline{\bsb}} \subset U_{\bsb} \subset \bbR^\bbN$.
We assume in what follows that
\begin{equation}\label{eq:barbjl1}
\overline{\bsb} = (\overline{b}_j )_{j\geq 1} \in \ell^1(\bbN) \;.
\end{equation}
These conditions are satisfied under the provision of appropriate
regularity of the covariance function of the Gaussian random field $\log(a
- a_*)$ in \eqref{log-diff-primal} (we refer to the discussion in
\cite[Rem.~4]{gknsss:2012}). Also, for any $\overline{\bsb} \in
\ell^1(\bbN) \subset \ell^2(\bbN)$ we have
$\overline{\boldsymbol{\mu}}_G( U_{\overline{\bsb}} ) = 1$.
\begin{proposition}\label{prop:BilRegPrp}
Let us assume \eqref{ass:regular} and \eqref{eq:barbjl1}, and suppose we
are given deterministic functions $f\in L^2(D)$ and $\phi_\cN \in
H^{1/2}(\Gamma_\cN)$. Then the following results hold.
\begin{enumerate}
\item[\textnormal{(a)}] For every $\bsy\in U_{\overline{\bsb}}$\,, the
    series \eqref{eq:axy} converges in $W^{1,\infty}(D)$ and
\begin{equation}\label{eq:aLip}
a(\cdot,\bsy) \in W^{1,\infty}(D)\,.
\end{equation}
\item[\textnormal{(b)}] The parametric solution map $\bsy\mapsto
    u(\cdot,\bsy)$ is strongly
    $\overline{\boldsymbol{\mu}}_G$--measurable as a map from
    $U_{\overline{\bsb}}$ to the space
\begin{equation}\label{eq:defW}
 W \,:=\, \{ v\in V: \Delta u\in L^2(D) \}\,,
\end{equation}
and we have the a priori estimate
\begin{equation} \label{eq:Est1}
 \|\Delta u(\cdot,\bsy)\|_{L^2(D)} \,\lesssim\,  T_1(\bsy)
 \left(\|f\|_{L^2(D)} + \| \phi_\cN \|_{H^{1/2}(\Gamma_\cN)} \right),
\end{equation}
where
\begin{equation}\label{eq:defT1}
 T_1(\bsy)
 \,:=\, \frac{1}{ \check{a}(\bsy) }
 + \frac{\|\nabla a(\cdot,\bsy)\|_{L^\infty(D)}}{\check{a}(\bsy)^2}
 \,<\, \infty
 \qquad \mbox{for all} \quad \bsy \in U_{\overline{\bsb}}
\,.
\end{equation}
\item[\textnormal{(c)}] There exists a sequence $\{ V_{h_\ell}
    \}_{\ell \ge 0}$ of nested FE spaces of continuous, piecewise
    linear functions on conforming, simplicial meshes $\{ \cT_{h_\ell}
\}_{\ell \ge 0}$ that satisfies the assumptions of
Section~\ref{sec:ml-qmc}: in particular, $M_\ell =
\mathrm{dim}(V_{h_\ell}) \eqsim h_\ell^{-d}$ and $h_\ell \eqsim
2^{-\ell}$.
The solutions $u_{h_\ell}(\cdot,\bsy) \in V_{h_\ell}$ defined in
Proposition~\ref{prop:ParmBilPrp}\,\textnormal{(e)} satisfy the
asymptotic error bound
\begin{equation} \label{eq:Est2}
  \|a^{1/2}(\cdot,\bsy)\,\nabla(u- u_{h_\ell})(\cdot,\bsy)\|_{L^2(D)}
  \,\lesssim\, h_\ell \,  T_2(\bsy)\,\|\Delta u(\cdot,\bsy)\|_{L^2(D)}\,,
\end{equation}
where
\begin{equation}\label{eq:defT2}
 T_2(\bsy) \,:=\, \hat{a}(\bsy)^{1/2}
 \qquad \mbox{for all} \quad
\bsy \in U_{\overline{\bsb}}\,.
\end{equation}
The result holds verbatim
for the FE solution $u_{h_\ell,s_\ell}(\cdot,(\bsy_{\{1:s_\ell\}},\boldsymbol{0}))
\in V_{h_\ell}$ of the $s_\ell$-term truncated problem.
\end{enumerate}
\end{proposition}
\begin{proof}
(a) \ This is a consequence of \eqref{eq:barbjl1} and
$\overline{\boldsymbol{\mu}}_G(U_{\overline{\bsb}})=1$ (see, eg.,
\cite[Lem.~2.28]{SG11_518}).

(b) \ Since $a(\cdot,\bsy) \in W^{1,\infty}(D)$, $u|_{\Gamma_{\cD}} = 0$
and \eqref{eq:bounda} holds, for every $\bsy\in U_{\overline{\bsb}}$, the
solution $u(\cdot,\bsy)$ of \eqref{eq:pde} also satisfies the
following Poisson problem
\begin{equation}\label{eq:MixBVP}
-\Delta u(\cdot,\bsy)
\,=\, \tilde{f}(\cdot,\bsy)
\,:=\,
\frac{1}{a(\cdot,\bsy)}\left[ f+\nabla a(\cdot,\bsy)\cdot \nabla u(\cdot,\bsy)\right]
\quad \mbox{in} \ L^2(D)\,,
\end{equation}
The bound \eqref{eq:Est1} with $T_1(\bsy)$ defined in \eqref{eq:defT1} then follows from
\eqref{eq:aprioriu}.

(c) \ The bound on $\| \Delta u(\cdot,\bsy)\|_{L^2(D)}$ in (b)
together with the classical regularity theory for the Laplace operator on
Lipschitz polygonal/polyhedral domains (see, e.g., \cite{Hack:2010})
implies \emph{weighted $H^2(D)$-regularity of $u(\cdot,\bsy)$} (with
suitable weights near reentrant corners and edges) for non-convex $D$ and
\emph{full $H^2(D)$-regularity} for convex $D$. The existence of a
sequence $\{ V_{h_\ell} \}_{\ell \ge 0}$ that satisfies the assumptions of
Section~\ref{sec:ml-qmc}, together with
\begin{equation}\label{eq:GalErr}
\inf_{v \in V_{h_\ell}} \| w - v \|_V
\,\lesssim\, h_\ell \, \| \Delta w \|_{L^2(D)} \qquad \text{for all} \quad w \in W\,,
\end{equation}
then follows from classical FE results (for convergence bounds in
weighted spaces see, e.g., \cite{Apel:1999}) and from
the norm equivalence $\|
\Delta w \|_{L^2(D)} \eqsim \|w\|_W$, for all $w \in W$. The error bound in
\eqref{eq:Est2} follows from an application of Cea's Lemma
\cite{Hack:2010} (in the energy norm) together with
\eqref{eq:ParmBilCont} and \eqref{eq:GalErr}.
\end{proof}
Note that, for convex $D$ and for homogeneous Dirichlet boundary
conditions on all of $\partial D$, $W = H^2(D)$ and the family $\{
\cT_{h_\ell} \}_{\ell \ge 0}$ can be constructed by uniform mesh
refinement of an arbitrary conforming triangulation $\cT_{h_0}$ of $D$.

\subsection{Parametric Regularity}
\label{sec:KeyRes}
%
As first observed in \cite{kss:2015}, in the uniform case, the
analysis of MLQMC methods for FE discretisations of PDEs requires
estimates of the parametric solution map $\bsy \mapsto u(\cdot,\bsy)$
\emph{in the regularity space $W$ in \eqref{eq:defW}}.
Here, we establish corresponding results for the lognormal parametric
problem \eqref{eq:pde}, \eqref{eq:axy}.
In order to be able to draw upon our results in \cite{gknsss:2012},
we restrict the analysis to the particular case
\begin{equation}\label{eq:Dirichlet}
\Gamma_\cN = \emptyset ,\quad \Gamma_\cD = \Gamma ,
\quad
V = H^1_0(D),\quad \text{and} \ \ V^* = H^{-1}(D)
\;.
\end{equation}
We denote by $\indx \,:=\, \{ \bsnu \in \bbN_0^\bbN \; : \;
|\bsnu|<\infty\}$, where $|\bsnu|:=\sum_{j=1}\nu_j$, the (countable) set
of all ``finitely supported'' multi-indices (i.e., sequences of
nonnegative integers for which only finitely many entries are nonzero).
For $\bsm,\bsnu\in\indx$, we write $\bsm\le\bsnu$ if $m_j\le \nu_j$ for
all $j$, we denote by $\bsnu-\bsm$ a multi-index with the elements
$\nu_j-m_j$, and we define $\binom{\bsnu}{\bsm}:= \prod_{j\ge 1}
\binom{\nu_j}{m_j}$. For a sequence of non-negative real numbers
$(\beta_j)_{j\in\bbN}$ we write $\bsbeta^\bsnu := \prod_{j\ge 1}
\beta_j^{\nu_j}$.  The following result is abstracted from the proof of
\cite[Thm.~14]{gknsss:2012}.

\begin{lemma} \label{lem:recur}
Given non-negative numbers $(\beta_j)_{j\in\bbN}$, let
$(\bbA_\bsnu)_{\bsnu\in\indx}$ and $(\bbB_\bsnu)_{\bsnu\in\indx}$ be
non-negative numbers satisfying the inequality
\[
  \bbA_\bsnu
  \,\le\, \sum_{\satop{\bsm\le\bsnu}{\bsm\ne\bsnu}} \binom{\bsnu}{\bsm} \bsbeta^{\bsnu-\bsm} \bbA_\bsm + \bbB_\bsnu\,,
  \quad\mbox{for any $\bsnu\in\indx$ \textnormal{(}including $\bsnu=\bszero$\textnormal{)}}.
\]
Then
\[
  \bbA_\bsnu
  \,\le\, \sum_{\bsk\le\bsnu} \binom{\bsnu}{\bsk} \Lambda_{|\bsk|}\, \bsbeta^\bsk\, \bbB_{\bsnu-\bsk}\,,
  \quad\mbox{for all $\bsnu\in\indx$},
\]
where the sequence $(\Lambda_n)_{n\ge 0}$ is defined recursively by
\begin{align} \label{eq:lambda}
  \Lambda_0 \,:=\, 1 \quad\mbox{and}\quad \Lambda_n \,:=\, \sum_{i=0}^{n-1} \binom{n}{i} \Lambda_i\,,
  \quad\mbox{for all $n\ge 1$}.
\end{align}
The result holds also
when both inequalities are replaced by equalities.
Moreover, we have
\begin{align} \label{eq:lambda-bound}
  \Lambda_n \,\le\, \frac{n!}{(\log 2)^n}\,, \quad\mbox{for all \ $n\ge
  0$}.
 \end{align}
\end{lemma}

\begin{proof}
We prove this result by induction. The case $\bsnu = \bszero$ holds
trivially. Suppose that the result holds for all $|\bsnu| < n$ with some
$n\ge 1$. Then for $|\bsnu| = n$, we substitute $\bsm' = \bsnu-\bsm$ in
the recursion and use the induction hypothesis to write
\begin{align*}
  \bbA_\bsnu
  &\,\le\, \sum_{\bszero\ne\bsm'\le\bsnu} \binom{\bsnu}{\bsnu-\bsm'} \beta^{\bsm'} \bbA_{\bsnu-\bsm'} + \bbB_\bsnu \\
  &\,\le\, \sum_{\bszero\ne\bsm'\le\bsnu} \binom{\bsnu}{\bsnu-\bsm'} \beta^{\bsm'}
  \sum_{\bsk\le\bsnu-\bsm'} \binom{\bsnu-\bsm'}{\bsk} \Lambda_{|\bsk|}\, \beta^\bsk\, \bbB_{\bsnu-\bsm'-\bsk} + \bbB_\bsnu .
\end{align*}
Substituting $\bsk' = \bsk + \bsm'$,
exchanging the order of summation,
and regrouping the binomial coefficients,
we obtain
\begin{align*}
  \bbA_\bsnu
  &\,\le\, \sum_{\bszero\ne\bsm'\le\bsnu} \binom{\bsnu}{\bsnu-\bsm'}
   \sum_{\bsm'\le\bsk'\le\bsnu} \binom{\bsnu-\bsm'}{\bsk'-\bsm'} \Lambda_{|\bsk'-\bsm'|}\, \beta^{\bsk'}\, \bbB_{\bsnu-\bsk'} + \bbB_\bsnu \\
  &\,=\, \sum_{\bszero\ne\bsk'\le\bsnu} \binom{\bsnu}{\bsk'}
   \sum_{\bszero\ne\bsm'\le\bsk'} \binom{\bsk'}{\bsk'-\bsm'} \Lambda_{|\bsk'-\bsm'|}\, \beta^{\bsk'}\, \bbB_{\bsnu-\bsk'} + \bbB_\bsnu,
\end{align*}
where
\begin{align*}
  \sum_{\bszero\ne\bsm\le\bsk'} \binom{\bsk'}{\bsk'-\bsm} \Lambda_{|\bsk'-\bsm|}
  &\,=\, \sum_{\satop{\bsm\le\bsk'}{\bsm\ne\bsk'}} \binom{\bsk'}{\bsm} \Lambda_{|\bsm|}
  \,=\, \sum_{i=0}^{|\bsk'|-1} \sum_{\satop{\bsm\le\bsk'}{|\bsm|=i}} \binom{\bsk'}{\bsm} \Lambda_i
  \,=\, \sum_{i=0}^{|\bsk'|-1} \binom{|\bsk'|}{i} \Lambda_i\,,
\end{align*}
which is equal to $\Lambda_{|\bsk'|}$ as required.
In the last step we used a simple counting identity (consider the
number of ways to select $i$ distinct balls from some baskets containing a
total number of $|\bsk'|$ distinct balls)
\begin{equation}\label{ident}
\sum_{\satop{\bsm\le\bsk'}{|\bsm|=i}} \binom{\bsk'}{\bsm} \,=\, \binom{|\bsk'|}{i}\,.
\end{equation}
The proof of \eqref{eq:lambda-bound} then follows as in the proof of
\cite[Thm.~14]{gknsss:2012}.
\end{proof}

\begin{theorem} \label{thm:key0}
For every $f\in L^2(D)$, $\bsnu\in \indx$ and $ \bsy \in
U_{\overline{\bsb}}$, with $T_1(\bsy)$ as in \eqref{eq:defT1}, we have
\[
  \|\Delta (\partial^\bsnu u(\cdot,\bsy) ) \|_{L^2(D)}
  \;\lesssim\; 
   \|f\|_{L^2(D)}
   \,T_1(\bsy)\,
 \overline{\bsb}^{\bsnu}\,2^{|\bsnu|}\,(|\bsnu|+1)!
\,.
\]
\end{theorem}

\begin{proof}
Throughout this proof, all estimates are for arbitrary
$\bsy\in U_{\bar{\bsb}}\subset U_{\bsb}$
with the understanding that constants implied in $\simeq$ and $\lesssim$
do not depend on $\bsy$.
For any multi-index $\bsnu\in\indx$, we define (formally, at this stage) the expression
\[
  g_\bsnu(\cdot,\bsy) \,:=\, \nabla\cdot (a(\cdot,\bsy)\, \nabla (\partial^\bsnu u(\cdot,\bsy)) )
\,.
\]

Differentiation of order $|\bsnu| > 0$ of
the parametric, deterministic variational formulation \eqref{eq:var} with respect to $\bsy$
reveals that
$$
0 \,=\, \partial^\bsnu \, \scA(\bsy;u(\cdot,\bsy),v)
  \,=\, \int_D \nabla v(\vx) \cdot \partial^\bsnu\big(a(\vx,\bsy) \nabla u(\vx,
  \bsy)\big) \,\rd\vx \qquad \text{for all} \quad v \in V\,.
$$
The Leibniz rule
$\partial^\bsnu (PQ) = \sum_{\bsm\le\bsnu}
\binom{\bsnu}{\bsm} (\partial^{\bsnu-\bsm}P)(\partial^\bsm Q)$
and integration by parts imply
\begin{align*}
 \nabla\cdot \partial^\bsnu (a\nabla u) \,=\,
\nabla \cdot
  \bigg(\sum_{\bsm \le\bsnu} \binom{\bsnu}{\bsm}
  (\partial^{\bsnu-\bsm} a)\, \nabla (\partial^\bsm u) \bigg)
 \,=\, 0
\quad \mbox{in}\ \ V^*\,.
\end{align*}
Separating out the $\bsm = \bsnu$ term yields the following identity in
$V^*$
\begin{align*}
\underbrace{\nabla\cdot(a\nabla(\partial^\bsnu u))}_{= g_\bsnu}
  &\,=\, - \nabla \cdot \bigg(\sum_{\satop{\bsm\le\bsnu}{\bsm\ne\bsnu}} \binom{\bsnu}{\bsm}
  (\partial^{\bsnu-\bsm} a)\, \nabla (\partial^\bsm u) \bigg) \\
  &\,=\, - \sum_{\satop{\bsm\le\bsnu}{\bsm\ne\bsnu}} \binom{\bsnu}{\bsm}
  \nabla\cdot \bigg(\frac{\partial^{\bsnu-\bsm} a}{a}\, (a \nabla (\partial^\bsm u)) \bigg) \\
  &\,=\, - \sum_{\satop{\bsm\le\bsnu}{\bsm\ne\bsnu}} \binom{\bsnu}{\bsm}
  \bigg(\frac{\partial^{\bsnu-\bsm} a}{a}\, g_\bsm \, + \, \nabla
    \bigg(\frac{\partial^{\bsnu-\bsm} a}{a}\bigg)\cdot (a \nabla
    (\partial^\bsm u)) \bigg)\,.
\end{align*}
In the last step we used the identity $\nabla\cdot (p\,\bsq) =
p\,\nabla\cdot\bsq + \nabla p\cdot\bsq$.
Due to \eqref{eq:bounda} we may multiply $g_\bsnu$ by $a^{-1/2}$
and obtain, for any $|\bsnu| > 0$, the recursive bound
\begin{align} \label{eq:step1}
  \|a^{-1/2}\,g_\bsnu\|_{L^2(D)} \,\le\, \sum_{\satop{\bsm\le\bsnu}{\bsm\ne\bsnu}} \binom{\bsnu}{\bsm}
  \bigg(&\;\bigg\|\frac{\partial^{\bsnu-\bsm}
          a}{a}\bigg\|_{L^\infty(D)}\, \|a^{-1/2} g_\bsm \|_{L^2(D)}
  \; + \nonumber
\\
  &+ \; \bigg\|\nabla \bigg(\frac{\partial^{\bsnu-\bsm}
    a}{a}\bigg)\bigg\|_{L^\infty(D)}\, \|a^{1/2} \nabla (\partial^\bsm u)\|_{L^2(D)}
\bigg)
\;.
\end{align}
By assumption, $g_0 = \nabla\cdot(a(\cdot,\bsy)\nabla
  u(\cdot,\bsy)) = -f \in L^2(D)$,
so that we obtain (by induction with respect to $|\bsnu|$ and
using \eqref{ass:regular} and \eqref{eq:barbjl1})
from \eqref{eq:step1}  that $a^{-1/2}(\cdot,\bsy)g_\bsnu(\cdot,\bsy)\in
L^2(D)$, and hence from Proposition~\ref{prop:ParmBilPrp}(a) that
$g_\bsnu(\cdot,\bsy)\in L^2(D)$ for every $\bsnu\in
\indx$. Thus, the above
formal identities indeed hold in $L^2(D)$.

To obtain a bound on \eqref{eq:step1}, we observe that it follows
from the particular form of $a$ in \eqref{eq:axy} that
\begin{equation}\label{eq:dady}
  \partial^{\bsnu-\bsm} a \,=\, (a-a_*) \prod_{j\ge 1}
  (\sqrt{\mu_j}\,\xi_j)^{\nu_j-m_j} \quad \text{for all} \quad
  \bsnu\ne\bsm\,.
\end{equation}
Since we assumed $a_* = 0$ in \eqref{ass:regular}, we then have
\begin{equation} \label{eq:est1}
  \bigg\|\frac{\partial^{\bsnu-\bsm} a}{a}\bigg\|_{L^\infty(D)}
  \,=\, \bigg\|\prod_{j\ge 1} (\sqrt{\mu_j}\,\xi_j)^{\nu_j-m_j}
  \bigg\|_{L^\infty(D)} \,\le\, \prod_{j\ge 1} \Big(\sqrt{\mu_j}\,\|\xi_j\|_{L^\infty(D)}\Big)^{\nu_j-m_j}
  \,=\, \bsb^{\bsnu-\bsm}
\;.
\end{equation}
Moreover, using the product rule, we have
\[
  \nabla\bigg(\frac{\partial^{\bsnu-\bsm} a}{a}\bigg)
  \,=\, \sum_{k\ge 1} (\nu_k-m_k)(\sqrt{\mu_k}\,\xi_k)^{\nu_k-m_k-1} (\sqrt{\mu_k}\,\nabla\xi_k)
        \prod_{\satop{j\ge 1}{j\ne k}} (\sqrt{\mu_j}\,\xi_j)^{\nu_j-m_j}\,.
\]
Due to the definition of $\overline{b}_j$ in \eqref{eq:barbj}, this
implies, in a similar manner to \eqref{eq:est1}, that
\begin{equation} \label{eq:est2}
\bigg\|\nabla\bigg(\frac{\partial^{\bsnu-\bsm} a}{a}\bigg) \bigg\|_{L^\infty(D)}
 \,\le\, |\bsnu-\bsm| \,\overline\bsb^{\bsnu-\bsm}.
\end{equation}
Substituting \eqref{eq:est1} and \eqref{eq:est2} into \eqref{eq:step1}, we
conclude that
\begin{align*}
 \underbrace{\|a^{-1/2}g_\bsnu\|_{L^2(D)}}_{\bbA_\bsnu}  \,\le\,
 \sum_{\satop{\bsm\le\bsnu}{\bsm\ne\bsnu}} \binom{\bsnu}{\bsm}\,\bsb^{\bsnu-\bsm}\,
 \underbrace{\|a^{-1/2}g_\bsm \|_{L^2(D)}}_{\bbA_\bsm}
 \,+\, B_\bsnu\,,
\end{align*}
where
\begin{align} \label{eq:Bnu}
  B_\bsnu
  \,:=&\,
  \sum_{\satop{\bsm\le\bsnu}{\bsm\ne\bsnu}} \binom{\bsnu}{\bsm}\,
  |\bsnu-\bsm| \,\overline{\bsb}^{\bsnu-\bsm}\,
  \|a^{1/2} \nabla (\partial^\bsm u)\|_{L^2(D)} \nonumber\\
 \,\le&\,
  \sum_{\satop{\bsm\le\bsnu}{\bsm\ne\bsnu}} \binom{\bsnu}{\bsm}\,
  |\bsnu-\bsm| \,\overline{\bsb}^{\bsnu-\bsm}\,
  \Lambda_{|\bsm|}\,\bsb^\bsm\,
 \frac{\|f\|_{V^*}}{\sqrt{\check{a}(\bsy)}}
 \,\le\,
  \overline\Lambda_{|\bsnu|}\,\overline{\bsb}^\bsnu
  \frac{\|f\|_{V^*}}{\sqrt{\check{a}(\bsy)}}
  \,.
\end{align}
In the first inequality in \eqref{eq:Bnu} we used
\begin{equation} \label{eq:induct3}
  \|a^{1/2} \nabla (\partial^\bsm u)\|_{{L^2(D)}}
  \,\le\, \Lambda_{|\bsm|}\,\bsb^\bsm\, \frac{\| f\|_{V^*}}{\sqrt{\check{a}(\bsy)}}\,,
\end{equation}
which was established in the proof of \cite[Thm.~14]{gknsss:2012}.
In the second inequality in \eqref{eq:Bnu} we used the bound $b_j\le
\overline{b}_j$, for all $j\ge 1$, and the identity \eqref{ident} to
write,  with $n = |\bsnu|$,
\[
  \sum_{\satop{\bsm\le\bsnu}{\bsm\ne\bsnu}}
  \binom{\bsnu}{\bsm}\,
  |\bsnu-\bsm|\, \Lambda_{|\bsm|}
  \,=\, \sum_{i=0}^{n-1} \sum_{\satop{\bsm\le\bsnu}{|\bsm|=i}}
  \binom{\bsnu}{\bsm}\, (n-i)\, \Lambda_i
  \,=\, \sum_{i=0}^{n-1} \binom{n}{i}\, (n-i)\,
    \Lambda_i
  \,=:\, \overline\Lambda_n\,.
\]

Since
$\bbA_\bszero = \|a^{-1/2} f\|_{L^2(D)} \le \|f\|_{L^2(D)}/\sqrt{\check{a}(\bsy)}$,
we now define
\[
 \bbB_\bsnu \,:=\, C_{\rm emb}\,\overline\Lambda_{|\bsnu|}\,\overline{\bsb}^\bsnu
  \frac{\|f\|_{L^2(D)}}{\sqrt{\check{a}(\bsy)}}\,,
\quad
\mbox{where}\;\;
C_{\rm emb} := \sup_{f\in L^2(D)} \frac{\| f \|_{V^*}}{\| f \|_{L^2(D)}}
\;.
\]
Then $\bbA_\bszero\le \bbB_\bszero$ and $B_\bsnu\le \bbB_\bsnu$ for all $\bsnu$.
We may now apply Lemma~\ref{lem:recur} to obtain
\begin{align}\label{eq:step2}
 \|a^{-1/2}g_\bsnu\|_{L^2(D)}
 &\,\le\,
 \sum_{\bsk\le\bsnu} \binom{\bsnu}{\bsk}\Lambda_{|\bsk|}\,\bsb^\bsk\,
 C_{\rm emb}\,\overline\Lambda_{|\bsnu-\bsk|}\,\overline{\bsb}^{\bsnu-\bsk}
  \frac{\|f\|_{L^2(D)}}{\sqrt{\check{a}(\bsy)}}\,.
\end{align}%
Note the extra factor $n-i$ in the definition of $\overline\Lambda_n$
compared to $\Lambda_n$ in \eqref{eq:lambda} so that
$\Lambda_n\le\overline\Lambda_n$. Using the bound in
\eqref{eq:lambda-bound}, we have for all $\alpha\le\log 2 =
0.69\cdots$,
\begin{align*}
  \overline\Lambda_n
  \,\le\, \sum_{i=0}^{n-1} \binom{n}{i} (n-i)\,\frac{i!}{\alpha^i}
  \,=\, \frac{n!}{\alpha^n}\,\alpha\,\sum_{i=0}^{n-1} \frac{\alpha^{n-i-1}}{(n-i-1)!}
  \,=\, \frac{n!}{\alpha^n}\,\alpha\,\sum_{k=0}^{n-1} \frac{\alpha^k}{k!}
  \,\le\, \frac{n!}{\alpha^n}\,\alpha\, e^\alpha
  \,\le\, \frac{n!}{\alpha^n},
\end{align*}
where the final step is valid provided that $\alpha\,e^\alpha \le 1$. Thus
it suffices to choose $\alpha \le 0.567\cdots$. For convenience we take
$\alpha = 0.5$ to bound \eqref{eq:step2}. Using again the identity
\eqref{ident}, we obtain
\begin{equation} \label{eq:identThm6}
  \sum_{\bsk\le\bsnu} \binom{\bsnu}{\bsk} |\bsk|!\, |\bsnu-\bsk|!
  \,=\, \sum_{i=0}^{|\bsnu|} i!\,(|\bsnu|-i)! \sum_{\satop{\bsk\le\bsnu}{|\bsk|=i}} \binom{\bsnu}{\bsk}
  \,=\, \sum_{i=0}^{|\bsnu|} i!\,(|\bsnu|-i)! \binom{|\bsnu|}{i}
  \,=\, (|\bsnu|+1)!\;.
\end{equation}
Applying these estimates to \eqref{eq:step2} gives
\begin{equation} \label{eq:step3}
 \|a^{-1/2}g_\bsnu\|_{L^2(D)}
 \,\le\, C_{\rm emb}\,\frac{\|f\|_{L^2(D)}}{\sqrt{\check{a}(\bsy)}}\,\overline{\bsb}^\bsnu\,2^{|\bsnu|}\,(|\bsnu|+1)!\,.
\end{equation}
Since $a^{-1/2}g_\bsnu = a^{-1/2}\nabla\cdot(a\nabla(\partial^\bsnu
u)) = a^{1/2}\Delta(\partial^\bsnu u) + a^{-1/2}\,(\nabla a\cdot
\nabla(\partial^\bsnu u))$, we have
\begin{align*}
 \|a^{1/2}\Delta(\partial^\bsnu u)\|_{L^2(D)}
 \,\le\, \|a^{-1/2}g_\bsnu\|_{L^2(D)} \,+\, \|a^{-1/2}\,(\nabla a\cdot\nabla(\partial^\bsnu u))\|_{L^2(D)}\,,
\end{align*}
which yields
\begin{align*}
 \sqrt{\check{a}(\bsy)}\,\|\Delta(\partial^\bsnu u)\|_{L^2(D)}
 \,\le\, \|a^{-1/2}g_\bsnu\|_{L^2(D)}
 \,+\, \frac{\|\nabla a(\cdot,\bsy)\|_{L^\infty(D)} }{\check{a}(\bsy)} \|a^{1/2}\nabla(\partial^\bsnu u)\|_{L^2(D)}\,,
\end{align*}
and in turn
\begin{align} \label{eq:step4}
 \|\Delta(\partial^\bsnu u)\|_{L^2(D)}
 \,\le\, \frac{\|a^{-1/2}g_\bsnu\|_{L^2(D)}}{\sqrt{\check{a}(\bsy)}}
 \,+\, \frac{\|\nabla a(\cdot,\bsy)\|_{L^\infty(D)} }{\check{a}(\bsy)}\,
 \frac{\|a^{1/2}\nabla(\partial^\bsnu u)\|_{L^2(D)}}{\sqrt{\check{a}(\bsy)}}\,.
\end{align}
Substituting \eqref{eq:step3} and \eqref{eq:induct3} into
\eqref{eq:step4}, and using $\Lambda_{|\bsnu|}\le
2^{|\bsnu|}|\bsnu|!$, as well as
$\bsb^\bsnu\le \overline{\bsb}^\bsnu$,
we arrive at
\begin{align*}
 \|\Delta(\partial^\bsnu u)\|_{L^2(D)}
 &\,\le\, C_{\rm emb}\,\|f\|_{L^2(D)}\bigg( \frac{1}{\check{a}(\bsy)} + \frac{\|\nabla a(\cdot,\bsy)\|_{L^\infty}}{\check{a}(\bsy)^2}\bigg)
 \overline{\bsb}^\bsnu\,\,2^{|\bsnu|}\,(|\bsnu|+1)!\,.
\end{align*}
This completes the proof.
\end{proof}

\subsection{QMC convergence and design}
\label{sec:QMCConv}
We first review the quasi-Monte Carlo theory that is essential for the QMC
convergence rate estimates. We follow the setting and analysis in
\cite[Section~4]{gknsss:2012} which, in turn, uses results from
\cite{NK14}, see also the earlier references
\cite{WW00,WW02,KWWat06,KSWWat10}.

In our multilevel algorithm \eqref{ML-rQMC}, for every level we apply
a randomly shifted lattice rule $Q_\ell$ to the integrand $F_\ell -
F_{\ell-1}$ which is multiplied by the product of univariate normal
densities. Replacing $F_\ell - F_{\ell-1}$ by a general function $\calF$
in $s$ variables, we have the general integration problem $\int_{\bbR^s}
\calF(\bsy)\,\prod_{j=1}^s \phi(y_j)\,\rd\bsy$, with $\phi(y) =
\exp(-y^2/2)/\sqrt{2\pi}$. The strategy in \cite{gknsss:2012} is to
consider a \emph{weighted} function space with norm defined by
\begin{align} \label{eq:Wnorm}
  &\| \mathcal{F}\|_{\calW_s}^2 \\
  &\,:=\,
  \sum_{\setu\subseteq\{1:s\}} \frac{1}{\gamma_\setu}
  \int_{\bbR^{|\setu|}}
  \Bigg(
  \int_{\bbR^{s-|\setu|}}
  \frac{\partial^{|\setu|} \mathcal{F}}{\partial \bsy_\setu}(\bsy_\setu;\bsy_{\{1:s\}\setminus\setu})
  \prod_{j\in\{1:s\}\setminus\setu} \phi(y_j)\,\rd\bsy_{\{1:s\}\setminus\setu}
  \Bigg)^2
  \prod_{j\in\setu} \psi_j^2(y_j)
  \,\rd\bsy_\setu\;, \nonumber
\end{align}
where $\{1:s\}$ is shorthand notation for the set of indices
$\{1,2,\ldots,s\}$, and
$\frac{\partial^{|\setu|}\mathcal{F}}{\partial\bsy_\setu}$ denotes the
mixed first derivative with respect to the ``active'' variables
$\bsy_\setu = (y_j)_{j\in\setu}$ while $\bsy_{\{1:s\}\setminus\setu} =
(y_j)_{j\in\{1:s\}\setminus\setu}$ denotes the ``inactive'' variables. To
ensure that the norm is finite for our particular integrand $\mathcal{F} =
F_\ell - F_{\ell-1}$, we follow \cite{gknsss:2012} and choose the
\emph{weight functions}
\begin{equation} \label{eq:psi}
  \psi_j^2(y_j) \,=\, \exp (-2\,\alpha_j\, |y|)\;,
  \quad\mbox{with}\quad
  0 < \alpha_{\min} \le \alpha_j \le \alpha_{\max} < \infty\;.
\end{equation}
In Corollary~\ref{cor:psi} below, we will further impose the condition
that $\alpha_j> 9 \overline{b}_j$, with $\overline{b}_j$ defined
by~\eqref{eq:barbj}.

A key ingredient in the analysis of \cite{gknsss:2012}, see also
\cite{kss:2012,kss:2015}, is to choose \emph{weight parameters}
$\gamma_\setu>0$, for every set $\setu\subset\bbN$ of finite cardinality
$|\setu|<\infty$, such that the overall error bound does not grow with
increasing dimension $s$. Such analysis makes use of the fact that the
generating vector $\bsz$ for a randomly shifted lattice rule (see
\eqref{SL-r}) can be constructed using a component-by-component algorithm
to achieve a certain error bound, see \cite[Thm.~15]{gknsss:2012} or more
generally \cite[Thm.~8]{NK14}. In particular, for $\mathcal{F} = F_\ell -
F_{\ell-1}$ the result is that the variance (or the mean square error) of
$\mathcal{Q}_\ell$ is bounded by
\begin{equation} \label{eq:cbc}
  \mathbb{V}_\bsDelta [\mathcal{Q}_\ell(F_{\ell} - F_{\ell-1})]
  \,\le\, R_\ell^{-1}\,
  \Bigg(\sum_{\emptyset\subseteq \setu\subseteq\{1:s_{\ell}\}} \gamma_\setu^\lambda\, \prod_{j\in\setu} \varrho_j(\lambda)\Bigg)^{1/\lambda}\,
   [\varphi_{\rm tot}(N_\ell)]^{-1/\lambda} \,
   \|F_{\ell} - F_{\ell-1}\|_{\mathcal{W}_{s_{\ell}}}^2
\end{equation}
for all $\lambda\in (1/2,1]$, with
\begin{equation} \label{eq:rho}
  \varrho_j(\lambda) \,:=\, 2\left(\frac{\sqrt{2\pi}\exp(\alpha_j^2/\eta_*)}{\pi^{2-2\eta_*}(1-\eta_*)\eta_*}\right)^\lambda\,
  \zeta\big(\lambda + \tfrac{1}{2}\big),
  \quad\mbox{and}\quad
  \eta_* \,:=\, \frac{2\lambda-1}{4\lambda}\,,
\end{equation}
where $\varphi_{\rm tot}(N) := |\{1\le z\le N-1 : \gcd(z,N)=1\}|$ denotes
the Euler totient function, and $\zeta(x) := \sum_{k=1}^\infty k^{-x}$
denotes the Riemann zeta function. Note that $\varphi_{\rm tot}(N)
  = N-1$ for $N$ prime and it has been verified that $1/\varphi_{\rm
    tot}(N) < 9/N$ for all $N \le 10^{30}$. Hence, from a practical
  point of view we can use
\begin{equation}\label{totient}
 \varphi_{\rm tot}(N) \,\eqsim\, N\,.
\end{equation}
The best rate of convergence clearly comes from choosing $\lambda$ close
to $1/2$, but the advantage is offset by the fact that $\zeta\big(\lambda
+ \tfrac{1}{2})\to\infty$ as $\lambda\to \tfrac{1}{2}_+$.

To verify Assumption M2 in Theorem~\ref{thm:mlqmc}, it remains to bound
$\|F_{\ell} - F_{\ell-1}\|_{\mathcal{W}_{s_{\ell}}}$ in \eqref{eq:cbc}.
Due to the triangle inequality,
\[
\|F_{\ell} - F_{\ell-1}\|_{\mathcal{W}_{s_{\ell}}} \, \le \,
\|\mathcal{G}(u_{h_\ell,s_\ell} - u_{s_\ell} )
\|_{\mathcal{W}_{s_{\ell}}} \, + \,
\|\mathcal{G}(u_{s_\ell} - u_{h_{\ell-1},s_{\ell-1}} )
\|_{\mathcal{W}_{s_{\ell}}} \,,
\]
it follows from the next theorem and the subsequent corollary that M2
holds with $\beta = 4$, in the case $s_\ell = s_{\ell-1}$. The remainder
of the paper is then devoted to giving a choice of weights
$\gamma_\setu$ that guarantees that the implied constant in M2 is
independent of $s_\ell$.
\begin{theorem} \label{thm:key1}
Let $s \in \mathbb{N}$, $h>0$ and $a_*=0$, and let $\bsnu\in \indx$ be a
general multi-index. Assume \eqref{ass:regular} and \eqref{eq:barbjl1}.
Then, for every $f\in L^2(D)$ and for every $\calG\in
L^2(D)^*$ with representer $g \in L^2(D)$, we have for all $\bsy\in
U_{\overline{\bsb}}$\,,
\[
  | \partial^\nu \calG(u(\cdot,\bsy) - u_h(\cdot,\bsy)) | \,\lesssim\,
   h^2\,\|f\|_{L^2(D)}\,\|g\|_{L^2(D)} \,
   H(\bsy) \,\overline{\bsb}^{\bsnu} \,2^{|\bsnu|} (|\bsnu|+5)!\,,
\]
with an implied constant that is independent of
$h$, $f$ and $g$, as well as of $\bsy \in U_{\overline{\bsb}}$,
and with
\begin{equation} \label{eq:defH}
  H(\bsy) \,:=\,
  T_1^2(\bsy)\,T_2^2(\bsy)
  \,<\, \infty \,,
\end{equation}
where $T_1(\bsy)$ and $T_2(\bsy)$ are as defined in \eqref{eq:defT1} and
\eqref{eq:defT2}, respectively.
\end{theorem}

\begin{proof}
Let $\bsy\in U_{\overline{\bsb}}$. We define $v^\mathcal{G}(\cdot,\bsy)\in
V$ and $v_h^\mathcal{G}(\cdot,\bsy)\in V_h$ via the adjoint problems
\begin{align}
  \scA(\bsy; w, v^\mathcal{G}(\cdot,\bsy)) &\,=\, \mathcal{G}( w )
  &&\text{for all}\quad w\in V\,,
  \label{eq:nitsche_1}
\\
  \scA(\bsy; w_h, v_h^\mathcal{G}(\cdot,\bsy)) &\,=\, \mathcal{G} ( w_h )
  &&\text{for all}\quad w_h\in V_h\,.
\label{eq:nitsche_2}
\end{align}
Due to Galerkin orthogonality for the original problem, i.e.,
\begin{equation}
\label{eq:GalerkinOrth}
\scA(\bsy;u(\cdot,\bsy)-u_h(\cdot,\bsy),
z_h)\,= \,0 \qquad \text{for all} \quad z_h \in V_h\,,
\end{equation}
on choosing the test function $w = u(\cdot,\bsy) - u_h(\cdot,\bsy)$ in
\eqref{eq:nitsche_1}, we obtain
\[
  \calG(u(\cdot,\bsy)-u_h(\cdot,\bsy))
 \,=\, \scA(\bsy; u(\cdot,\bsy)-u_h(\cdot,\bsy), v^\mathcal{G}(\cdot,\bsy) - v_h^\mathcal{G}(\cdot,\bsy)).
\]

From the Leibniz rule we have
\begin{align*}
  &\partial^\bsnu\mathcal{G}(u - u_{h})
  \,=\,
  \int_D \partial^\bsnu \left(a \,\nabla (u
  - u_{h}) \cdot \nabla (v^\mathcal{G} - v^\mathcal{G}_h) \right)\, \rd\vx \\
  &\,=\,
  \int_D \sum_{\bsm\le\bsnu} \binom{\bsnu}{\bsm} (\partial^{\bsnu-\bsm} a)\,
  \partial^\bsm \left( \nabla (u - u_{h}) \cdot \nabla (v^\mathcal{G} - v^\mathcal{G}_h) \right)\, \rd\vx \\
  &\,=\,
  \int_D \sum_{\bsm\le\bsnu} \binom{\bsnu}{\bsm} (\partial^{\bsnu-\bsm} a)\,
  \sum_{\bsk\le\bsm} \binom{\bsm}{\bsk} \nabla \partial^\bsk (u- u_{h}) \cdot
  \nabla \partial^{\bsm-\bsk} (v^\mathcal{G} - v^\mathcal{G}_h)\, \rd\vx \\
  &\,=\,
  \int_D \sum_{\bsm\le\bsnu} \binom{\bsnu}{\bsm} \frac{\partial^{\bsnu-\bsm} a}{a}
  \sum_{\bsk\le\bsm} \binom{\bsm}{\bsk}
  \left(a^{1/2}\nabla \partial^\bsk (u - u_{h})\right) \cdot
  \left(a^{1/2}\nabla \partial^{\bsm-\bsk} (v^\mathcal{G} - v^\mathcal{G}_h)\right)\, \rd\vx.
\end{align*}
Using the Cauchy-Schwarz inequality and \eqref{eq:est1}, we obtain
\begin{align} \label{eq:split}
 &|\partial^\bsnu\mathcal{G}(u - u_{h}) |
 \,\le\,
 \sum_{\bsm\le\bsnu} \binom{\bsnu}{\bsm}
 \bigg\|\frac{\partial^{\bsnu-\bsm} a}{a}\bigg\|_{L^\infty(D)}
 \sum_{\bsk\le\bsm} \binom{\bsm}{\bsk}
 \left(\int_D a\, |\nabla \partial^\bsk (u - u_{h})|^2\,\rd \vx\right)^{1/2} \nonumber\\
 &\qquad\qquad\qquad\qquad\qquad\qquad\qquad\qquad\qquad\qquad\qquad\times
 \left(\int_D a\, |\nabla \partial^{\bsm-\bsk} (v^\mathcal{G} - v^\mathcal{G}_h)|^2\,\rd \vx\right)^{1/2}
 \nonumber\\
 &\,\le\,
 \sum_{\bsm\le\bsnu} \binom{\bsnu}{\bsm}
 \bsb^{\bsnu-\bsm}
 \sum_{\bsk\le\bsm} \binom{\bsm}{\bsk}
 \| a^{1/2} \nabla \partial^\bsk (u - u_{h})\|_{L^2(D)}\,
 \| a^{1/2} \nabla \partial^{\bsm-\bsk} (v^\mathcal{G} - v^\mathcal{G}_h)\|_{L^2(D)}\,.
\end{align}

To continue, we need to obtain an estimate for $\| a^{1/2} \nabla
\partial^\bsk (u - u_{h})\|_{L^2(D)}$. Let $\calI:V\to V$ denote the
identity operator and let $\calP_h = \calP_h(\bsy) :V\to V_h$ denote the
parametric FE projection onto $V_h$ which is defined, for arbitrary $w\in
V$, by
\begin{equation} \label{eq:Ph}
\scA(\bsy; \calP_h(\bsy) w - w, z_h) \,=\, 0
\qquad\text{for all} \quad z_h\in V_h \,.
\end{equation}
In particular, we have $u_h = \calP_h  u \in V_h$ and
\begin{equation}\label{eq:IdemP}
\calP_h^2(\bsy) \,\equiv\, \calP_h(\bsy) \quad\mbox{on}\quad V_h \,.
\end{equation}
Moreover, since $\partial^\bsk u_h \in V_h$ for every $\bsk\in
\indx$, it follows from \eqref{eq:IdemP} that
\begin{equation}\label{eq:I-Ph}
(\calI - \calP_h(\bsy))(\partial^\bsk u_h(\bsy)) \,\equiv\, 0
\,.
\end{equation}
Thus
\begin{align} \label{eq:new_1}
  &\| a^{1/2} \nabla \partial^\bsk (u - u_{h})\|_{L^2(D)}
  \,=\, \| a^{1/2} \nabla \calP_h \partial^\bsk (u - u_{h}) + a^{1/2} \nabla (\calI - \calP_h) \partial^\bsk (u - u_{h}) \|_{L^2(D)} \nonumber
\\
  &\,\le\, \| a^{1/2} \nabla \calP_h \partial^\bsk (u - u_{h})\|_{L^2(D)}
  \,+\, \|a^{1/2} \nabla (\calI - \calP_h) \partial^\bsk u \|_{L^2(D)}
\,.
\end{align}
Now, applying $\partial^\bsk$ to \eqref{eq:GalerkinOrth} and separating
out the $\bsell = \bsk$ term, we get for all $z_h\in V_h$,
\begin{align} \label{eq:new_2}
 \int_D a\, \nabla \partial^{\bsk} (u-u_h)\cdot\nabla z_h \,\rd \vx
 \,=\, - \sum_{\satop{\bsell\le\bsk}{\bsell\ne\bsk}} \binom{\bsk}{\bsell}
 \int_D (\partial^{\bsk-\bsell} a)
  \nabla\partial^{\bsell}(u-u_h)\cdot\nabla z_h \,\rd \vx\,.
\end{align}
Choosing $z_h = \calP_h \partial^\bsk (u-u_h)$ and using the definition
\eqref{eq:Ph} of $\calP_h$, the left-hand side of \eqref{eq:new_2} is
equal to $\int_D a\, |\nabla \calP_h
\partial^{\bsk} (u-u_h)|^2 \,\rd \vx$. Dividing and multiplying the
right-hand side of \eqref{eq:new_2} by $a$, and using the Cauchy-Schwarz
inequality, then leads to the bound
\begin{align*}
 &\int_D a\, |\nabla \calP_h \partial^{\bsk} (u-u_h)|^2 \,\rd \vx \\
 &\ \ \le\, \sum_{\satop{\bsell\le\bsk}{\bsell\ne\bsk}} \binom{\bsk}{\bsell}
 \left\|\frac{\partial^{\bsk-\bsell} a}{a}\right\|_{L^\infty(D)}
 \left(\int_D a\, |\nabla\partial^{\bsell}(u-u_h)|^2\,\rd\vx \right)^{\frac{1}{2}}
 \left(\int_D a\, |\nabla \calP_h \partial^{\bsk} (u-u_h)|^2\,\rd\vx \right)^{\frac{1}{2}}.
\end{align*}
Canceling one common factor from both sides and using \eqref{eq:est1}, we
arrive at
\begin{align} \label{eq:new_3}
 \|a^{1/2} \nabla \calP_h \partial^{\bsk} (u-u_h)\|_{L^2(D)}
 &\,\le\, \sum_{\satop{\bsell\le\bsk}{\bsell\ne\bsk}} \binom{\bsk}{\bsell}\bsb^{\bsk-\bsell}\,
 \|a^{1/2} \nabla\partial^{\bsell}(u-u_h)\|_{L^2(D)}.
\end{align}
Substituting \eqref{eq:new_3} into \eqref{eq:new_1}, we then obtain
\begin{align*}
  &\| a^{1/2} \nabla \partial^\bsk (u - u_{h})\|_{L^2(D)} \\[1ex]
  &\,\le\, \sum_{\satop{\bsell\le\bsk}{\bsell\ne\bsk}} \binom{\bsk}{\bsell}\bsb^{\bsk-\bsell}
  \underbrace{\|a^{1/2} \nabla\partial^{\bsell}(u-u_h)\|_{L^2(D)}}_{\bbA_\bsell}
  \,+\, \underbrace{\|a^{1/2} \nabla (\calI - \calP_h) \partial^\bsk u
    \|_{L^2(D)}}_{\bbB_\bsk}.
\end{align*}
Note that we have $\bbA_\bszero = \bbB_\bszero$.
Now applying Lemma~\ref{lem:recur} with $\alpha=0.5$,
Proposition~\ref{prop:BilRegPrp}(c) and Theorem~\ref{thm:key0}, we
conclude that
\begin{align} \label{eq:nice_1}
  &\| a^{1/2} \nabla \partial^\bsk (u - u_{h})\|_{L^2(D)}
  \,\le\, \sum_{\bsell\le\bsk} \binom{\bsk}{\bsell} \Lambda_{|\bsell|}\,\bsb^\bsell\,
  \|a^{1/2} \nabla (\calI - \calP_h) \partial^{\bsk-\bsell} u
  \|_{L^2(D)}\,.\nonumber\\
&\,\lesssim\, h\, T_2(\bsy)\, \sum_{\bsell\le\bsk} \binom{\bsk}{\bsell} \Lambda_{|\bsell|}\, \bsb^\bsell\,
  \|\Delta (\partial^{\bsk-\bsell} u)\|_{L^2(D)}
  \nonumber \\
  &\,\lesssim\, h\, \|f\|_{L^2(D)} T_1(\bsy)\,T_2(\bsy)\,
  \sum_{\bsell\le\bsk} \binom{\bsk}{\bsell} \, 2^{|\bsell|} |\bsell|!\,\overline\bsb^\bsell\,
  \overline\bsb^{\bsk-\bsell}\, 2^{|\bsk-\bsell|}\,(|\bsk-\bsell|+1)! \nonumber\\
  &\,=\, h \, \|f\|_{L^2(D)} T_1(\bsy)\,T_2(\bsy)\,
  \overline\bsb^{\bsk}\, 2^{|\bsk|}\, \frac{(|\bsk|+2)!}{2}\,,
\end{align}
where $T_1(\bsy)$ and $T_2(\bsy)$ are defined in \eqref{eq:defT1} and
\eqref{eq:defT2}, respectively, and where in the last step we used the
identity
\[
  \sum_{\bsell\le\bsk} \binom{\bsk}{\bsell} |\bsell|! \,(|\bsk-\bsell|+1)!
  \,=\, \frac{(|\bsk|+2)!}{2}\,,
\]
which can be derived in the same way as \eqref{eq:identThm6}.

Since the bilinear form $\scA(\bsy;\cdot,\cdot)$ is symmetric and since
the representer $g$ for the linear functional $\mathcal{G}(\cdot)$ is in
$L^2(D)$, all the results in Section \ref{sec:Prel} hold verbatim also for
the adjoint problem \eqref{eq:nitsche_1} and for its FE discretisation
\eqref{eq:nitsche_2}. Hence, as in \eqref{eq:nice_1}, we obtain
\begin{align} \label{eq:nice_2}
  \| a^{1/2} \nabla \partial^{\bsm-\bsk} (v^\mathcal{G} - v^\mathcal{G}_{h})\|_{L^2(D)}
  &\lesssim h \, \|g\|_{L^2(D)} T_1(\bsy)\,T_2(\bsy) \,
  \overline\bsb^{\bsm-\bsk}\, 2^{|\bsm-\bsk|} \frac{(|\bsm-\bsk|+2)!}{2}\,.
\end{align}
Substituting \eqref{eq:nice_1} and \eqref{eq:nice_2} into \eqref{eq:split}
yields
\begin{align*}
 |\partial^\bsnu\mathcal{G}(u - u_{h}) |
 &\,\lesssim\,
 h^2\, \|f\|_{L^2(D)} \,\|g\|_{L^2(D)}\,
 T_1^2(\bsy)\,T_2^2(\bsy) \nonumber\\
 &\,\times
 \sum_{\bsm\le\bsnu} \binom{\bsnu}{\bsm}
 \bsb^{\bsnu-\bsm}
 \sum_{\bsk\le\bsm} \binom{\bsm}{\bsk}
 \overline\bsb^{\bsk}\, 2^{|\bsk|}\, \frac{(|\bsk|+2)!}{2}\,
 \overline\bsb^{\bsm-\bsk}\, 2^{|\bsm-\bsk|}\, \frac{(|\bsm-\bsk|+2)!}{2}\,.
\end{align*}
Using \eqref{ident} we can obtain a similar identity to
\eqref{eq:identThm6},
\begin{align*}
 \sum_{\bsk\le\bsm} \binom{\bsm}{\bsk} \frac{(|\bsk|+2)!}{2}\,\frac{(|\bsm-\bsk|+2)!}{2}
 \,=\, \frac{(|\bsm|+5)!}{120}\,.
\end{align*}
Using again \eqref{ident}, with $n=|\bsnu|$ we have
\begin{align*}
 &\sum_{\bsm\le\bsnu} \binom{\bsnu}{\bsm} 2^{|\bsm|}\,\frac{(|\bsm|+5)!}{120} \\
 &\,=\, \sum_{i=0}^{n} \binom{n}{i} 2^{i}\,\frac{(i+5)!}{120}
\,=\, n!\,\sum_{i=0}^{n} \frac{(i+1)(i+2)(i+3)(i+4)(i+5) 2^{i}}{120(n-i)!}
 \,\le\, \frac{(n+5)!}{120} 2^n e\,.
\end{align*}
These, together with $b_j\le\overline{b}_j$ for all $j\ge 1$, yield the
required bound in the theorem.
\end{proof}

Now, to estimate the $\calW_s$-norm of $\calG(u-u_h)$, we need to bound
its mixed first partial derivatives with respect to $\bsy =
(y_1,\ldots,y_s, 0,0,\ldots)$. The result in Theorem~\ref{thm:key1} was
more general. In the following, we will only consider multi-indices
$\bsnu$ where each $\nu_j \le 1$. As in the definition of the norm on
$\cW_s$, we will use subsets $\setu \subseteq \{1:s\}$ of active indices
instead of multi-indices.

\begin{corollary} \label{cor:psi}
Let $\hat{a}_0 := \max_{\vx\in \overline{D}} a_0(\vx)$ and  $\check{a}_0
:= \min_{\vx\in \overline{D}} a_0(\vx)$. For the weight functions $\psi_j$
defined by \eqref{eq:psi} with parameters $\alpha_j>
9\overline{b}_j$, we have
\begin{align*}
  &\| \mathcal{G}(u_s-u_{h,s})\|_{\calW_s}^2
  \,\lesssim \,
  K^2\, h^4
  \sum_{\setu\subseteq\{1:s\}}
  \frac{[(|\setu|+5)!]^2}{\gamma_\setu}
  \prod_{j\in\setu} \widetilde{b}_j^2,
\ \ \text{with} \ \
\widetilde{b}_j := 
  \frac{\overline{b}_j}{\exp(81\overline{b}_j^2/2)\,\Phi(9\overline{b}_j)\,\sqrt{\alpha_j-9\overline{b}_j}
}
\end{align*}
and
\[
  K \,:=\, \|f\|_{L^2(D)}\,\|g\|_{L^2(D)}\,
  \left(1 + \frac{\|\nabla a_0\|_{L^\infty(D)}}{\check{a}_0}\right)^2\,
  \frac{\hat{a}_0^{3}}{\check{a}_0^{4}}
  \exp\bigg(\frac{81}{2} \sum_{j\ge 1} \overline{b}_j^2 + \frac{18}{\sqrt{2\pi}} \sum_{j\ge 1}\overline{b}_j\bigg).
\]
\end{corollary}

\begin{proof}
We begin by estimating $H(\bsy)$ defined in \eqref{eq:defH}. It
follows from \eqref{eq:axy} with $a_*=0$ that
\[
  \nabla a(\vec{x},\bsy)
  \,=\, a(\vec{x},\bsy) \bigg(\frac{\nabla a_0(\vec{x})}{a_0(\vec{x})} + \sum_{j\ge1} \sqrt{\mu_j}\,\nabla\xi_j(\vec{x})\, y_j\bigg)\;,
\]
leading to
\[
  \|\nabla a(\cdot,\bsy)\|_{L^\infty(D)}
  \,\le\, \hat{a}(\bsy) \bigg(\frac{\|\nabla a_0\|_{L^\infty(D)}}{\check{a}_0}
  + \sum_{j\ge1} \overline{b}_j\, |y_j|\bigg)\;.
\]
Since $1+x \le \exp(x)$ for $x \ge 0$, we have
\begin{align*}
 H(\bsy)
 &\,\le\,
  \bigg(1 + \frac{\|\nabla a_0\|_{L^\infty(D)}}{\check{a}_0} +
   \sum_{j\ge1} \overline{b}_j\, |y_j|\bigg)^2 \,
  \frac{\hat{a}(\bsy)^{3}}{\check{a}(\bsy)^{4}} \\
  &\,\le\,
  \left(1 + \frac{\|\nabla a_0\|_{L^\infty(D)}}{\check{a}_0}\right)^2
  \,\exp \bigg(2\sum_{j\ge1} \overline{b}_j\, |y_j|\bigg)
  \frac{\hat{a}_0^{3}}{\check{a}_0^{4}}\exp \bigg(7 \sum_{j\ge1} b_j\, |y_j|\bigg) \\
  &\,\le\,
  \left(1 + \frac{\|\nabla a_0\|_{L^\infty(D)}}{\check{a}_0}\right)^2
  \frac{\hat{a}_0^{3}}{\check{a}_0^{4}} \prod_{j\ge1} \exp (9\overline{b}_j\, |y_j| )\;.
\end{align*}
Therefore, with $K_* := \|f\|_{L^2(D)}\,\|g\|_{L^2(D)}\, (1+\frac{\|\nabla
 a_0\|_{L^\infty(D)}}{\check{a}_0})^2 \hat{a}_0^{3}/\check{a}_0^{4}$, it follows from
Theorem~\ref{thm:key1} and the definition of the $\calW_s$--norm
  in \eqref{eq:Wnorm} that
\begin{align}\label{Wsnorm:bound}
  &\| \calG(u_s-u_{h,s})\|_{\calW_s}^2
  \,\lesssim\,  h^4\, K_*^2
  \sum_{\setu\subseteq\{1:s\}} \Bigg[\frac{[(|\setu|+5)!]^2\prod_{j\in\setu} (4\bar{b}_j^2)}{\gamma_\setu} \\
  &\qquad\times
  \int_{\bbR^{|\setu|}}
  \left(
  \int_{\bbR^{s-|\setu|}}
  \prod_{j\in\{1:s\}\setminus\setu} \exp (9\overline{b}_j\, |y_j| )\phi(y_j)\,\rd\bsy_{\{1:s\}\setminus\setu}
  \right)^2
  \prod_{j\in\setu} \exp (18\overline{b}_j\, |y_j| )\psi_j^2(y_j)
  \,\rd\bsy_\setu \Bigg], \nonumber
\end{align}
leading to the univariate integrals
\[
  1\ \,\le\,
  \int_{-\infty}^\infty \exp (9\overline{b}_j\, |y| ) \phi(y)\,\rd y
  \,=\, 2\exp\left(\frac{81}{2}\overline{b}_j^2\right) \Phi\left(9\overline{b}_j\right)
  \,\le\, \exp\left(\frac{81}{2}\overline{b}_j^2 + \frac{18}{\sqrt{2\pi}}\overline{b}_j\right)
\]
and
\begin{equation} \label{eq:change}
  \int_{-\infty}^\infty \exp (18\, \overline{b}_j\, |y| )\psi_j^2(y) \,\rd y
  \,=\, \frac{1}{\alpha_j - 9\overline{b}_j}
\,  .
\end{equation}
These, together with \eqref{Wsnorm:bound}, then yield the estimate on the
$\calW_s$-norm of $\calG(u-u_h)$.
\end{proof}

\begin{theorem} \label{thm:final}
For every $f\in L^2(D)$ and for every $\calG\in L^2(D)^*$ with representer
$g \in L^2(D)$, consider the multilevel QMC algorithm defined by
\eqref{ML-rQMC} with $s_\ell = s$ and $R_\ell = R$ for all $\ell =
0,\ldots,L$. Suppose that the sequence $\overline{b}_j$ defined by
\eqref{eq:barbj} satisfies
\[
  \sum_{j\ge1} \overline{b}_j^q < \infty
  \qquad\mbox{for some}\qquad 0<q<1\;.
\]
For each $\ell=1,\ldots,L$, let the generating vector for the randomly
shifted lattice rule $\mathcal{Q}_\ell$ be constructed using a
component-by-component algorithm \cite{NK14}, with weight parameters
\[
  \gamma_\setu \,:=\,
  \Bigg(
  \frac{(|\setu|+5)!}{120} \prod_{j\in\setu}
  \frac{\widetilde{b}_j}{\sqrt{\varrho_j(\lambda)}}
  \Bigg)^{2/(1+\lambda)}
 \quad \text{and} \quad \alpha_j \,:=\,
  \frac{1}{2} \Bigg(9\overline{b}_j + \sqrt{81\overline{b}_j^2 + 1 - \frac{1}{2\lambda}}\Bigg)
\]
in the weight functions \eqref{eq:psi}, where $\widetilde{b}_j$ is as
defined in Corollary \ref{cor:psi} and
\begin{equation}\label{def:lambda}
  \lambda \,:=\,
  \begin{cases}
  \frac{1}{2-2\delta} \quad\mbox{for some}\quad \delta\in (0,1/2) & \mbox{when}\quad q\in (0,2/3), \\
  \frac{q}{2-q} & \mbox{when}\quad q\in (2/3,1).
  \end{cases}
\end{equation}
Let the generating vector
for the randomly shifted lattice rule $\mathcal{Q}_0$ be constructed as in
\cite{gknsss:2012} with $\lambda$ as defined in \eqref{def:lambda}. Then
\begin{equation} \label{eq:our_M3}
\mathbb{V}_\bsDelta [\mathcal{Q}_\ell (F_\ell - F_{\ell-1})]
\,\lesssim\, D_{\bsgamma}(\lambda)  \, R^{-1} \,
[\varphi_{\rm tot}(N_\ell)]^{-1/\lambda}  \, h_{\ell}^4\,,\quad \text{for
    all} \ \ \ell=0,\ldots,L,
\end{equation}
where $D_{\bsgamma}(\lambda)<\infty$ is independent of $s$ and $\ell$.
\end{theorem}

\begin{proof}
First, let $\ell \ge 1$. Using \eqref{eq:cbc} and the triangle inequality,
we obtain
\begin{align*}
  \mathbb{V}_\bsDelta [\mathcal{Q}_\ell (F_\ell - F_{\ell-1})]
  \le\, 2 R^{-1} & \Bigg(\sum_{\emptyset\ne\setu\subseteq\{1:s\}}
  \gamma_\setu^\lambda\,
    \prod_{j\in\setu}\varrho_j(\lambda)\Bigg)^{1/\lambda} \,
    [\varphi_{\rm tot}(N_\ell)]^{-1/\lambda} \\
  & \Big(\|\calG(u_{s}-u_{h_{\ell},s})\|_{\calW_{s}}^2
   + \|\calG(u_{s}-u_{h_{\ell-1},s})\|_{\calW_{s}}^2 \Big)\,.
\end{align*}
The bound in \eqref{eq:our_M3} now follows from Corollary~\ref{cor:psi}
with
\[
  D_{\bsgamma}(\lambda)
  \,:=\,
  \Bigg(
   \sum_{|\setu|<\infty} \gamma_\setu^\lambda\, \prod_{j\in\setu}\varrho_j(\lambda) \Bigg)^{1/\lambda}\,
   \Bigg(
  \frac{[(|\setu|+5)!]^2}{\gamma_\setu}
  \prod_{j\in\setu} \widetilde{b}_j^2   \Bigg) \;.
\]
The fact that $D_{\bsgamma}(\lambda)$ is finite can be verified following
the same arguments as in the proofs of \cite[Thm.~20 and Cor.~21]{gknsss:2012}.

Since by definition $b_j \le \overline{b}_j$ and thus $\overline{\bsb} \in
\ell^q(\mathbb{N})$ implies $\bsb \in \ell^q(\mathbb{N})$, the result for
$\ell=0$ follows from \cite[Thm.~20]{gknsss:2012} with
$D_{\bsgamma}(\lambda) = C_{\boldsymbol{\gamma}}(\lambda)$, defined in
\cite[Eqn.~(4.19)]{gknsss:2012}.
\end{proof}

Together with \eqref{totient}, Theorem~\ref{thm:final} shows that
Assumption~M2 of Theorem~\ref{thm:mlqmc} holds with $\beta = 4$ and
$\lambda$ defined in \eqref{def:lambda}.

\begin{remark} \label{rem:support}
As an example, let us consider the case where the KL expansion in
\eqref{kle} arises from a Gaussian field with Mat\'ern covariance
with smoothness parameter $\nu$, as defined in Section
\ref{sec:numerical}.
We have from \cite[Corollary~5]{gknsss:2012} that $\mu_j \lesssim
j^{-(1+2\nu/d)}$. Moreover, we see from the proof of
\cite[Prop.~9]{gknsss:2012} that $\|\nabla\xi_j\|_{L^\infty(D)} \lesssim
\mu_j^{-\tilde{r}/r}$ for all $d/2+1 < \tilde{r} < r < d+2\nu$, allowing
us to infer that $\overline{b}_j \lesssim \mu_j^{1/2-\tilde{r}/r}$. To
ensure that the assumption $\sum_{j\ge 1} \overline{b}_j^q < \infty$ in
Theorem~\ref{thm:final} holds, we need
\[
  q\, \bigg(1+ \frac{2\nu}{d} \bigg)\,\bigg(\frac{1}{2} - \frac{d/2+1}{d+2\nu} \bigg)
  \,=\, q\,\frac{\nu-1}{d} > 1\;.
\]
Therefore, a sufficient condition for the asumption to hold with $q<1$
is $\nu>d+1$. A sufficient condition for $q<2/3$ (and thus $\lambda =
1/(2-2\delta)$)  is $\nu>3d/2+1$. As we saw in
  Section \ref{sec:numerical}, these sufficient conditions do not seem to be necessary
  ones and we observe $\lambda \approx 1/2$ even for much smaller
  values of $\nu$.
\end{remark}

\begin{remark} \label{rem:correct}
Corollary~\ref{cor:psi} could be compared with
\cite[Thm.~16]{gknsss:2012}. Unfortunately, there is a small, inconsequential
error in \cite[Eq.~(4.17)]{gknsss:2012}. The factors under the first
product in \cite[Eq.~(4.17)]{gknsss:2012} should be squared, and as a
result, the denominator in \cite[Eq.~(4.11)]{gknsss:2012} should also be
squared. However, since this only amounts to the omission of a factor $\geq 1$ in the
denominator, the estimate \cite[Eq.~(4.10)]{gknsss:2012} is valid as
stated. We have checked numerically that the weights
\cite[Eq.~(4.23)]{gknsss:2012} with the adjusted
formula for \cite[Eq.~(4.11)]{gknsss:2012} lead, in all numerical
experiments reported in \cite{gknsss:2012}, to qualitatively the same
results and therefore do not affect any of the conclusions drawn in
\cite{gknsss:2012}.
\end{remark}

\paragraph*{Acknowledgments.}

Frances Kuo and Ian Sloan acknowledge the support of the Australian
Research Council under the projects DP110100442, FT130100655, DP150101770.
Robert Scheichl and Elisabeth Ullmann acknowledge support from the EPSRC
under the project EP/H051503/1. Christoph Schwab acknowledges partial
support by the European Research Council ERC and the Swiss National
Science Foundation SNSF during the preparation of this work through Grant
ERC AdG247277. A large part of this work was performed during visits of
Robert Scheichl and Christoph Schwab to the School of Mathematics and
Statistics, University of New South Wales, Australia, as well as the visit
of Frances Kuo to the Department of Mathematical Sciences, University of
Bath, UK. The authors thank Mahadevan Ganesh for suggesting an
alternative proof strategy which led to quantitative improvements in the
bounds of Theorem~\ref{thm:key0}, and James Nichols for re-running all
numerical experiments from \cite{gknsss:2012} mentioned in
Remark~\ref{rem:correct}.

\small
\bibliographystyle{plain}

\begin{thebibliography}{10}

\bibitem{Apel:1999} T.~Apel.
\newblock {\em Anisotropic Finite Elements: Local Estimates and Applications}.
\newblock Advances in Numerical Mathematics. Teubner, 1999.

\bibitem{BarthSchwabZollinger:2011} A.~Barth, C.~Schwab, and N.~Zollinger.
\newblock {Multi-level Monte Carlo finite element method for elliptic PDE's
  with stochastic coefficients}.
\newblock {\em Numer.~Math.}, 119:123--161, 2011.

\bibitem{BG:2004} H.-J. Bungartz and M.~Griebel.
\newblock Sparse grids.
\newblock {\em Acta Numerica}, 13:147--269, 2004.

\bibitem{Charrier:2013} J.~Charrier and A.~Debussche.
\newblock Weak truncation error estimates for elliptic {PDEs} with lognormal
  coefficients.
\newblock {\em Stoch.~PDE:~Anal.~Comp.}, 1:63--93, 2013.

\bibitem{CST:2011} J.~Charrier, R.~Scheichl, and A.L. Teckentrup.
\newblock {Finite Element Error Analysis of Elliptic PDEs with Random
  Coefficients and its Application to Multilevel Monte Carlo Methods}.
\newblock {\em SIAM J.~Numer.~Anal.}, 51(1):322--352, 2013.

\bibitem{CGST:2011} K.A. Cliffe, M.B. Giles, R.~Scheichl, and A.L.
    Teckentrup.
\newblock Multilevel {M}onte {C}arlo methods and applications to elliptic
  {PDE}s with random coefficients.
\newblock {\em Comput.~Visual.~Sci.}, 14:3--15, 2011.

\bibitem{CKN06} R.~Cools, F.Y. Kuo, and D.~Nuyens.
\newblock Constructing embedded lattice rules for multivariate integration.
\newblock {\em SIAM J.~on Sci.~Comput.}, 28:2162--2188, 2006.

\bibitem{suitesparse} T.A. Davis.
\newblock Algorithm 832: {UMFPACK V4.3}---an unsymmetric-pattern multifrontal
  method.
\newblock {\em ACM Trans.~Math.~Software}, 30(2):196--199,
  2004.

\bibitem{DKLS14_1319} J.~Dick, F.Y. Kuo, Q.T. {Le Gia}, and Ch. Schwab.
\newblock Multi-level higher order {QMC} {Galerkin} discretization for affine
  parametric operator equations.
\newblock Preprint arXiv:1406.4432, Cornell University, 2014
(to appear in {\em SIAM J.\ Numer.\ Anal.}~2016).

\bibitem{DPW08} J.~Dick, F.~Pillichshammer, and B.J. Waterhouse.
\newblock The construction of good extensible rank-$1$ lattices.
\newblock {\em Math.~Comp.}, 77:2345--2374, 2008.

\bibitem{Giles:2008} M.B. Giles.
\newblock Multilevel {M}onte {C}arlo path simulation.
\newblock {\em Oper.~Res.}, 56(3):607--617, 2008.

\bibitem{GilesWaterhouse:2009} M.B. Giles and B.J. Waterhouse.
\newblock {Multilevel quasi-Monte Carlo path simulation}.
\newblock {\em {Radon Series Comp.~Appl.~Math.}}, 8:1--18, 2009.

\bibitem{gknsss:2012} I.G. Graham, F.Y. Kuo, J.A. Nichols, R.~Scheichl,
    Ch. Schwab, and I.H. Sloan.
\newblock {Quasi-Monte Carlo} finite element methods for elliptic {PDEs} with
  lognormal random coefficients.
\newblock {\em Numer.~Math.}, 131:329--368, 2015

\bibitem{Graham_etal:2011} I.G. Graham, F.Y. Kuo, D.~Nuyens, R.~Scheichl,
    and I.H. Sloan.
\newblock {Quasi-Monte-Carlo methods for elliptic PDEs with random coefficients
  and applications}.
\newblock {\em J.~Comput.~Phys.}, 230:3668--3694, 2011.

\bibitem{Hack:2010} W.~Hackbusch.
\newblock {\em Elliptic Differential Equations: {T}heory and Numerical
  Treatment}, volume~18 of {\em Springer Series in Computational Mathematics}.
\newblock Springer, 2010.

\bibitem{Harbrecht:2013} H.~Harbrecht, M.~Peters, and M.~Siebenmorgen.
\newblock Multilevel accelerated quadrature for {PDEs} with log-normal
  distributed random coefficient.
\newblock Preprint 2013-18, Math.~Institut, Universit\"at Basel, 2013
(to appear in {\em Math.~Comp.}~2016).

\bibitem{freefem} F.~Hecht.
\newblock New development in {FreeFem++}.
\newblock {\em J. Numer. Math.}, 20(3-4):251--265, 2012.

\bibitem{Heinrich:2001} S.~Heinrich.
\newblock Multilevel monte carlo methods.
\newblock In {\em Lecture Notes in Large Scale Scientific Computing}, number
  2179, pages 58--67. Springer-Verlag, 2001.

\bibitem{HN03} F.J. Hickernell and H.~Niederreiter.
\newblock The existence of good extensible rank-$1$ lattices.
\newblock {\em J.~Complexity}, 19:286--300, 2003.

\bibitem{DKS:2013} F.Y.~Kuo J.~Dick and I.H. Sloan.
\newblock {High-dimensional integration -- the Quasi-Monte Carlo way}.
\newblock {\em Acta Numerica}, 22:133--288, 2013.

\bibitem{JohnsonFrigo:2007} S.G. Johnson and M.~Frigo.
\newblock A modified split-radix {FFT} with fewer arithmetic operations.
\newblock {\em IEEE Trans.~Signal Processing}, 55(1):111--119, 2007.

\bibitem{frances_web} F.Y. Kuo.
\newblock Lattice rule generating vectors.
\newblock \url{web.maths.unsw.edu.au/~fkuo/lattice/index.html}.

\bibitem{kss:2012} F.Y. Kuo, Ch. Schwab, and I.H. Sloan.
\newblock {Quasi-Monte Carlo finite element methods for a class of elliptic
  partial differential equations with random coefficient}.
\newblock {\em SIAM J.~Numer.~Anal.}, 6(50):3351--3374, 2012.

\bibitem{kss:2015} F.Y. Kuo, Ch. Schwab, and I.H. Sloan.
\newblock Multi-level {Quasi-Monte Carlo} finite element methods for a class of
  elliptic partial differential equations with random coefficient.
\newblock {\em Found.\ Comput.\ Math.}, 15:411--449, 2015.

\bibitem{KSWWat10} F.Y. Kuo, I.H. Sloan, G.W. Wasilkowski, and B.J.
    Waterhouse.
\newblock Randomly shifted lattice rules with the optimal rate of convergence
  for unbounded integrands.
\newblock {\em J.~Complexity}, 26:135--160, 2010.

\bibitem{KWWat06} F.Y. Kuo, G.W. Wasilkowski, and B.J. Waterhouse.
\newblock Randomly shifted lattice rules for unbounded integrals.
\newblock {\em J.~Complexity}, 22:630--651, 2006.

\bibitem{Loeve:1978} M.~Lo\`{e}ve.
\newblock {\em Probability Theory}, volume~II.
\newblock Springer-Verlag, New York, 4th edition, 1978.

\bibitem{NK14} J.A. Nichols and F.Y. Kuo.
\newblock Fast component-by-component construction of randomly shifted lattice
  rules achieving $\mathcal{O}(n^{-1+\delta})$ convergence for unbounded
  integrands in $\mathbb{R}^s$ in weighted spaces with {POD} weights.
\newblock {\em J.~Complexity}, 30:444--468, 2014.

\bibitem{NC06} D.~Nuyens and R.~Cools.
\newblock Fast algorithms for component-by-component construction of rank-$1$
  lattice rules in shift-invariant reproducing kernel {H}ilbert spaces.
\newblock {\em Math.~Comp.}, 75:903--920, 2006.

\bibitem{Robbe_etal:2015} P.~Robbe, D.~Nuyens, and S.~Vanderwalle.
\newblock A practical multilevel {quasi-Monte Carlo} method for elliptic {PDEs}
  with random coefficients.
\newblock In Master Thesis ``Een parallelle multilevel
  Monte-Carlo-methode\linebreak voor de simulatie van stochastische parti\"ele
  differentiaalvergelijkingen'' by P.~Robbe, June 2015.

\bibitem{SG11_518} Ch. Schwab and C.J. Gittelson.
\newblock Sparse tensor discretizations of high-dimensional parametric and
  stochastic {PDEs}.
\newblock {\em Acta Numerica}, 20:291--467, 2011.

\bibitem{aretha_paper} A.L. Teckentrup.
\newblock {Multilevel Monte Carlo Methods for highly heterogeneous media}.
\newblock In C.~Laroque, J.~Himmelspach, R.~Pasupathy, O.~Rose, and A.M.
  Uhrmacher, editors, {\em {Proceedings of the 2012 Winter Simulation
  Conference}}. WSC, 2012.

\bibitem{MLSC} A.L. Teckentrup, P.~Jantsch, C.G. Webster, and
    M.~Gunzburger.
\newblock A multilevel stochastic collocation method for partial differential
  equations with random input data.
\newblock {\em SIAM/ASA J.~Uncertainty Quantification}, 3:1046--1074, 2015.

\bibitem{TSGU:2012} A.L. Teckentrup, R.~Scheichl, M.B. Giles, and
    E.~Ullmann.
\newblock {Further analysis of multilevel Monte Carlo methods for elliptic PDEs
  with random coefficients}.
\newblock {\em Numer.~Math.}, 125(3):569--600, 2013.

\bibitem{Vass:2008} P.S.~Vassilevski.
\newblock {\em Multilevel Block Factorization Preconditioners}.
\newblock Springer, 2008.

\bibitem{WW00} G.W. Wasilkowski and H.~Wo\'zniakowski.
\newblock Complexity of weighted approximation over $\mathbb{R}^1$.
\newblock {\em J.~Approx.\ Theory.}, 103:223--251, 2000.

\bibitem{WW02} G.W. Wasilkowski and H.~Wo\'zniakowski.
\newblock Tractability of approximation and integration for weighted tensor
  product problems over unbounded domains.
\newblock In K.T. Fang, F.J. Hickernell, and H.~Niederreiter, editors, {\em
  Monte Carlo and Quasi-Monte Carlo Methods 2000}, pages 497--522, Berlin,
  2002. Springer.

\end{thebibliography}

\newpage

\noindent
{\bf Frances Y. Kuo}\\
{\tt f.kuo@unsw.edu.au}\\
School of Mathematics and Statistics\\
University of New South Wales\\
Sydney NSW 2052\\
Australia\\[4ex]
{\bf Robert Scheichl}\\
{\tt R.Scheichl@bath.ac.uk}\\
Department of Mathematical Sciences\\
University of Bath\\
Bath BA2 7AY
UK\\[4ex]
{\bf Christoph Schwab}\\
{\tt christoph.schwab@sam.math.ethz.ch}\\
Seminar f\"ur Angewandte Mathematik\\
ETH Z\"urich\\
R\"amistrasse 101\\
8092 Z\"urich\\
Switzerland\\[4ex]
{\bf Ian H. Sloan}\\
{\tt i.sloan@unsw.edu.au}\\
School of Mathematics and Statistics\\
University of New South Wales\\
Sydney NSW 2052\\
Australia\\[4ex]
{\bf Elisabeth Ullmann}\\
{\tt elisabeth.ullmann@ma.tum.de}\\
Department of Mathematics\\
Technische Universit\"at M\"unchen\\
Boltzmannstra{\ss}e 3\\	
85748 Garching\\
Germany

\end{document}